\documentclass[11pt]{article}
\usepackage{graphicx}

\usepackage{tikz}
\usepackage{caption}
\usepackage{subcaption}
\usepackage{mathtools}

\usepackage{complexity}
\usepackage{amsthm}
\usepackage{thmtools}
\usepackage[margin=1in]{geometry}

\newcommand{\eps}{\varepsilon}
\usepackage{algorithm,algorithmicx}

\usepackage{dirtytalk}
\usepackage{csquotes}
\usepackage{bbold}

\usepackage{enumerate}
\usepackage{paralist}

\usepackage{amsmath} 
\usepackage{amssymb} 

\usepackage{natbib}

\usepackage[mathletters]{ucs}

\usepackage{amsmath,amstext,amssymb,amsfonts}
\usepackage{graphicx}
\usepackage[dvipsnames]{xcolor}
\usepackage{tabularx}
\usepackage{verbatim}
\usepackage{bbm}

\usepackage{microtype}

\usepackage[T1]{fontenc}    

\newtheorem{theorem}{Theorem}
\newtheorem*{theorem*}{Theorem}

\newtheorem{claim}[theorem]{Claim}
\newtheorem*{claim*}{Claim}

\newtheorem{proposition}[theorem]{Proposition}
\newtheorem*{proposition*}{Proposition}
\newtheorem{lemma}[theorem]{Lemma}
\newtheorem*{lemma*}{Lemma}
\newtheorem{corollary}[theorem]{Corollary}
\newtheorem*{corollary*}{Corollary}

\newtheorem*{conjecture*}{Conjecture}
\newtheorem{observation}[theorem]{Observation}
\newtheorem{fact}[theorem]{Fact}
\newtheorem*{fact*}{Fact}

\newtheorem*{hypothesis*}{Hypothesis}

\theoremstyle{definition}
\newtheorem{definition}[theorem]{Definition}

\newtheorem{remark}[theorem]{Remark}
\newtheorem*{remark*}{Remark}

\newcommand{\Sref}[1]{\hyperref[#1]{\S\ref*{#1}}}

\renewcommand{\leq}{\leqslant}
\renewcommand{\le}{\leqslant}
\renewcommand{\geq}{\geqslant}
\renewcommand{\ge}{\geqslant}

\usepackage{bm}

\newcommand{\paren}[1]{\left(#1 \right )}

\newcommand{\Brac}[1]{\left[#1\right]}

\newcommand{\set}[1]{\left\{#1\right\}}

\newcommand{\ceil}[1]{\lceil #1 \rceil}
\newcommand{\floor}[1]{\lfloor #1 \rfloor}

\newcommand{\lspan}[1]{\mathsf{span}\set{#1}}

\newcommand{\norm}[1]{\left\lVert#1\right\rVert}

\newcommand{\inprod}[1]{\left\langle #1\right\rangle}

\newcommand{\N}{{\mathbb N}}
\renewcommand{\R}{\mathbb R}

\newcommand{\Fi}{\mathbb F}

\newcommand{\val}{\mathsf {val}}

\newcommand{\Esymb}{\mathbbm{E}}
\newcommand{\Psymb}{\mathbbm{P}}

\DeclareMathOperator*{\Expect}{\Esymb}

\DeclareMathOperator*{\ProbOp}{\Psymb}

\newcommand{\ex}[2]{\Expect_{#1}\Brac{#2}}
\newcommand{\Ex}[1]{\Expect\Brac{#1}}

\renewcommand{\Pr}[1]{\ProbOp\Brac{#1}}
\newcommand{\pr}[2]{\ProbOp_{#1}\Brac{#2}}

\newcommand{\one}{\mathbbm{1}}

\definecolor{DSgray}{cmyk}{0,0,0,0.7}

\renewcommand{\deg}{\mathsf{deg}}

\newcommand{\size}{\mathsf{size}}

\usepackage[pdftex,pagebackref,colorlinks, linkcolor=teal, filecolor = blue, citecolor = Periwinkle, urlcolor  = teal]{hyperref}
\usepackage{cleveref}
\usepackage{thmtools, thm-restate}

    \definecolor{urlcolor}{rgb}{0,.145,.698}
    \definecolor{linkcolor}{rgb}{.71,0.21,0.01}
    \definecolor{citecolor}{rgb}{.12,.54,.11}

    \definecolor{ansi-black}{HTML}{3E424D}
    \definecolor{ansi-black-intense}{HTML}{282C36}
    \definecolor{ansi-red}{HTML}{E75C58}
    \definecolor{ansi-red-intense}{HTML}{B22B31}
    \definecolor{ansi-green}{HTML}{00A250}
    \definecolor{ansi-green-intense}{HTML}{007427}
    \definecolor{ansi-yellow}{HTML}{DDB62B}
    \definecolor{ansi-yellow-intense}{HTML}{B27D12}
    \definecolor{ansi-blue}{HTML}{208FFB}
    \definecolor{ansi-blue-intense}{HTML}{0065CA}
    \definecolor{ansi-magenta}{HTML}{D160C4}
    \definecolor{ansi-magenta-intense}{HTML}{A03196}
    \definecolor{ansi-cyan}{HTML}{60C6C8}
    \definecolor{ansi-cyan-intense}{HTML}{258F8F}
    \definecolor{ansi-white}{HTML}{C5C1B4}
    \definecolor{ansi-white-intense}{HTML}{A1A6B2}

\usepackage{mlmodern}

\crefname{fact}{fact}{facts}
\Crefname{fact}{Fact}{Facts}

\crefname{observation}{observation}{observations}
\Crefname{observation}{Observation}{Observations}

\crefname{claim}{claim}{claims}
\Crefname{claim}{Claim}{Claims}

\newcommand{\wt}{\mathsf{wt}}
\newcommand{\im}{\mathrm{im}}
\newcommand{\dist}{\mathsf{dist}}
\newcommand{\maj}{\mathsf{maj}}
\newcommand{\agr}{\mathsf {agr}}

\newcommand{\gapLC}[2]{\mathsf {gapLabelCover}[#1,#2] }

\newcommand{\col}{\mathtt {col}}
\newcommand{\diam}{\mathsf {diam}}
\newcommand{\sign}{\mathsf {sign}}

\newcommand{\buildings}{\mathfrak {G}} 
\newcommand{\symmetric}{\mathrm {Sym}}
\newcommand{\grassmann}{\mathrm {GP}}

\newcommand{\faces}{\mathcal {Z}}
\newcommand{\fcolors}{\mathcal {C}}
\newcommand{\F}{\mathrm F}

\title{A Simple Sub-Polynomial Degree Coboundary Expander}
\author{
Max Hopkins\footnote{Supported by NSF Award DMS-2424441}\\
Institute for Advanced Study, Princeton\\
\href{mailto:nmhopkin@ias.edu}{nmhopkin@ias.edu}
\and
Arka Ray\footnote{Supported in part by the Walmart Center for Tech Excellence at IISc (CSR Grant WMGT-23-0001).}\\
    Indian Institute of Science, Bengaluru\\
    \href{mailto:arkaray@iisc.ac.in}{arkaray@iisc.ac.in}
}
\date{}

\begin{document}

\maketitle

\begin{abstract}
High dimensional expanders simultaneously satisfying spectral and combinatorial (coboundary) expansion have recently played a major role in breakthroughs in PCP and coding theory, but the only known construction of such complexes is extremely involved, requiring deep algebraic number theory.
In this work, we give an extremely simple combinatorial construction of a sub-polynomial degree complex based on projections of the flags complex (subspace chains) that is (i) a local spectral expander, (ii) a coboundary expander, and (iii) a swap coboundary expander.
As a corollary, we also give the first near-linear size combinatorial hypergraphs with good agreement tests in the `1\%' regime, and a simple PCP construction with near-linear size.
\end{abstract}

\newpage
\tableofcontents
\newpage

\section{Introduction}
High dimensional expanders, sparse but robustly connected hypergraphs, have seen a surge of interest and application in theoretical computer science, driving breakthrough progress in coding theory~\cite{EKZ20,DDHR20, AJQST20, JST21, PK22, FK22, DELLM22, KO22,gur20253}, approximate sampling ~\cite{ALGV19, ALGVV21, AJKPV22, ALG24, LLG25}, property testing~\cite{DK17, DD19,KM20, GK23, DD24a, BM24, DDL24}, hardness of approximation~\cite{DFHT21, HL22, BMVY25}, and pseudorandomness \cite{hsieh2025explicit,Dikstein2026HighRate}. Unlike their widely-studied graph analogs, there are very few known constructions of high dimensional expanders, and all known methods are algebraic. In this work, we give a simple construction of sub-polynomial degree hypergraphs based on subspaces that are both spectral and combinatorial high dimensional expanders, and a corresponding simple near-linear size PCP construction.

We consider three core notions of expansion: local spectral expansion (see \Cref{def:local-spectral-expansion}), a generalization of spectral expansion in graphs, coboundary expansion (see \Cref{def:coboundary-expansion}), a generalization of edge expansion, and swap-coboundary expansion, a key notion related to the construction of PCPs measuring the coboundary expansion of higher level disjoint faces (see \Cref{def:faces}). Unlike the graph setting, these generalizations behave quite differently in high dimensions, e.g., there are random two-dimensional simplicial complexes that are good local spectral expanders but not good coboundary expanders with high probability~\cite{GW16}. Constructing sparse hypergraphs satisfying both spectral and coboundary expansion has been a major challenge since the notions were introduced \cite{LM06,Gromov2010,KKL16}, and has received renewed attention after it was shown such hypergraphs admit low soundness `agreement tests' \cite{GK23,DD24a,BM24} and efficient PCPs \cite{BMVY25stoc}.

The first sparse hypergraphs satisfying all three notions of expansion were constructed recently in \cite{CL25, DDL24, BLM24}, based on deep algebraic number theory. This was followed very recently (and independently to our work) by \cite{OS25}, who gave a simpler algebraic construction based on variants of the elementary coset complexes of \cite{KO18}. In this work, we show that if one is willing to sacrifice somewhat on sparsity, it is possible to give completely elementary combinatorial constructions of local-spectral and (swap)-coboundary HDX of sub-polynomial degree. Following \cite{BMVY25}, we show these give rise to a simpler PCP construction roughly matching the parameters of the breakthrough algebraic PCP of \cite{MR08} in the small constant soundness regime.

Our complexes are based upon projections of the \emph{flags complex}, a well-known dense hypergraph more typically studied as the local structure of \cite{LSV05}'s quotients of the affine building, and shown to be a swap-coboundary expander in \cite{DD24buildings}. The vertices of flags complex are the subspaces of $\mathbb F_q^d$ (excepting $0$ and $\mathbb F_q^d$). Its hyperedges are given by inclusion chains (`flags') of these subspaces.
\begin{definition}[Flags Complex]
\label{def:building}
    Given $d \in \mathbb{N}$ and prime power $q$, the flags complex is the $(d-1)$-partite complex with vertices given by all subspaces of $\mathbb{F}_q^d$, and top-level faces given by inclusion chains (flags) of subspaces:
    \[
    \buildings^d_q(d-2) = \left\{ (W_1 \subseteq W_2 \subseteq \ldots \subseteq W_{d-1}): \dim(W_i)=i\right\}.
    \]
\end{definition}

While the flags complex is known to be a high-dimensional expander under all 
three definitions~\cite{DD24buildings}, it has two fatal flaws: it is 
\emph{dense}, with roughly $q^{d^2/2}$ top faces over only 
$q^{d^2/4}$ vertices (a quadratic blow-up), and has 
\emph{extremely high max-degree}. In fact, both flaws stem from the same source: the 
number of $i$-dimensional subspaces of $\mathbb{F}_q^d$ varies wildly with $i$. In particular, since there are roughly $q^{i(d-i)}$ subspaces of dimension $i$, 
the levels are sharply imbalanced: each of the $q^{d-1}$ lines, for instance, show up in far more hyperedges than any of the $q^{d^2/4}$ subspaces of dimension $d/2$. Measured by size and max-degree (the key 
parameters used in applications), the flags complex is thus, asymptotically, essentially no better than simply taking an imbalanced product $[q^{d/4}]^{d-1} \times [q^{d^2/4}]$. As such, despite its 
fantastic expansion, the flags complex has seen little use 
outside its role as the local structure of~\cite{LSV05a}'s Ramanujan complexes in the literature.

The above, however, suggests a simple fix. The imbalance in the flags complex comes from \emph{extremal} dimensions, those near $1$ and $d$, whereas \emph{central} dimensions, i.e., those near $d/2$, are much closer to balanced, each containing $q^{d^2/4 \pm \Theta(d)}$ subspaces. As such, we should simply discard the problematic levels and only consider 
inclusion chains through the well-behaved ones.

\begin{definition}[Projected Flags Complex]
Given $d \in \mathbb{N}$, a prime power $q$, and a subset of dimensions 
$S = \{s_1 < s_2 < \cdots < s_k\} \subseteq [d-1]$, the \emph{projected 
flags complex} $\buildings_q^{d,S}$ is the $k$-partite simplicial complex with 
vertices given by subspaces $W \subset \mathbb{F}_q^d$ such that 
$\dim(W) \in S$ (partitioned into parts by dimension), and top-level 
faces given by inclusion chains using one subspace of each dimension 
in $S$:
\[
\buildings_q^{d,S}(k-1) = \left\{ W_1 \subset W_2 \subset \cdots \subset W_k 
\;:\; \dim(W_i) = s_i \right\}.
\]
Equivalently, $\buildings_q^{d,S}$ is the complex whose top faces are obtained 
by projecting each flag of $\buildings_q^d$ onto its components of dimension 
in $S$. We drop $d$ and $q$ from the notation when clear from context.
\end{definition}
Taking $S$ to be the $2k+1$ central 
dimensions of the flags complex over $\mathbb{F}_q^{2d}$,
\[
S_k = \{d-k,\, \ldots,\, d-1,\, d,\, d+1,\, \ldots,\, d+k\},
\]
every level has $q^{d^2/4 \pm \Theta(kd)}$ of 
subspaces, eliminating the imbalance up to lower order factors. Fixing $k$ and letting 
$d \to \infty$, we get a strongly explicit infinite family of complexes with 
$q^{\Theta(d^2)}$ vertices and max-degree 
$q^{O(kd)} \approx 2^{O(k\sqrt{\log N})}$ --- sub-polynomial in the number of vertices 
$N$. The main technical content of this work is to show this 
projection retains the spectral and (swap)-coboundary expansion of the 
full complex, yielding simple near-linear-size low-soundness PCPs.

\subsection{Our Results}
\paragraph{Local Spectral Expansion.}
A $k$-dimensional simplicial complex $X=\set{\emptyset}\sqcup X(0)\sqcup X(1)\sqcup \cdots \sqcup X(k)$, where the set of $i$-dimensional faces $X(i)$ contains faces with cardinality $i+1$, is a (two-sided) $\lambda$-local spectral expander if (a) the underlying graph of $X$ (i.e.\ the graph with vertices $X(0)$ and edges $X(1)$) is a (two-sided) spectral expander,
and (b) for any face $s\in X$ with dimension at most $k-2$, the underlying graph of the $s$-link $X_s=\set{\tau\setminus s| \tau \in X,\ \tau\supseteq s }$ is also a (two-sided) spectral expander.

Our first result states any projection of the flags complex is a local spectral expander.\footnote{The below also references the notion of a $\lambda$-product, a slight variant of local-spectral expansion in the partite case which promises every bi-partite component of every link has expansion $\lambda_2 \leq \lambda$. See \Cref{def:lambda-product}.}
\begin{restatable}[Spectral Expansion of the Projected Flags Complex]{theorem}{spectralProjected}
\label{thm:spectral-projected}
Fix any $d\geq 3$, any prime power $q$, and a set $S\subseteq [d-1]$ of size at least 2.
Let $s\in \buildings^{d,S}_q$ be any face with co-dimension at least 2.
Then, the second eigenvalue of the underlying graph of the $s$-link of $\buildings^{d,S}_q$ is $\frac{2}{(|S|-|s|-1)(\sqrt{q}-1)}$.
In particular, the complex $\buildings^{d,S}_q$ is a $\frac{2}{\sqrt{q}-1}$-local-spectral expander for any prime power $q$.
Moreover, the complex $\buildings^{d,S}_q$ is also a $\frac{1}{\sqrt{q}}$-product for any prime power $q$.
\end{restatable}
We remark that while the local-spectral expansion of the projected flags complex has not been studied explicitly to our knowledge, similar bounds (up to a worse constant in the denominator) can be derived from combining improved `Trickling-Down' Theorems \cite{AGO23,LLG25} with well known bounds on the spectrum of subspace inclusion (Grassmann) graphs \cite{Del76}.
We give a simple and direct proof that does not require the former machinery.

\paragraph{Coboundary Expansion.}
One way to view combinatorial (edge) expansion of a (connected) graph is that if a graph has expansion $\eta$, then for any function $f:V \to \mathbb F_2$ there exists a constant function $\delta g: V \to \mathbb{F}_2$ such that
\[
 \eta \cdot \dist(f,\delta g)\leq  \wt (\delta f)
\]
where $\delta f(u,v) = f(u)-f(v)$ is the `coboundary' of $f$ (similarly $\delta g$ is the `coboundary' of some constant $g: \{\emptyset\}\to \mathbb F_2$). In other words, $\delta f$ is the size of the cut formed by $f$, and any $f$ whose cut has few edges must be close to $0$ (a very small set) or $1$ (a very large set).

For general groups, this definition can be extended to anti-symmetric\footnote{We discuss this standard mild restriction on the family of functions in \Cref{sec:simplicial-expansion}.} functions taking values from any group $\Gamma$ instead of $\mathbb F_2$ by taking $\delta f(u,v)=f(u)f(v)^{-1}$.
For a simplicial complex, it is possible to generalize this notion to the 1-coboundary expansion with respect to a group $\Gamma$ that compares weights of the coboundary of 1-dimensional faces to their distance from 0-dimensional coboundary, as opposed to comparing weights of coboundary functions on 0-dimensional faces to their distance from the corresponding coboundary.
In particular, a simplicial complex has 1-coboundary expansion $h^1(X, \Gamma)$ if for any anti-symmetric function $f:\overrightarrow X(1)\to \Gamma$
\[
 h^1(X, \Gamma) \cdot \dist(f,\delta g)\leq  \wt (\delta f)
\]
for some anti-symmetric function $g: X(0) \to \Gamma$ on the directed edges $\overrightarrow X(1)$ and vertices $X(0)$, where $\delta f(u,v,w) = f(u,v)f(v,w)f(w,u)$ is the `coboundary' of $f$.
Our second result states that the projection of the flags complex onto its central dimensions is a coboundary expander.
\begin{restatable}[Coboundary Expansion of Projected Flags Complex]{theorem}{coboundarySpherical}
\label{thm:coboundary-spherical}
    There exists a constant $q_0$ such that the following holds.
    Let $\Gamma$ be any group.
    Let $d > k \geq 2$ be two integers and $q\geq q_0$ be a prime power. 
    Let 
\[
S_{k}=\set{d-k, \dots, d-1, d, d+1, \dots, d+k}.
\]
    Then, for any face $s\in \buildings^{2d,S_{k}}_q$ with codimension at least 3, $\buildings^{2d,S_{k}}_{q,s}$ is 1-coboundary expander with
    \[
    h^{1}(\buildings^{2d,S_{k}}_{q,s}, \Gamma) = \Omega\left(1\right).
    \]
\end{restatable}

As a corollary, we get a completely elementary construction of local-spectral coboundary expanders of sub-polynomial degree:
\begin{restatable}[Elementary Sub-polynomial Degree HDX]{corollary}{Introcor}
\label{cor:coboundary-spherical}
    There exists a constant $q_0$ such that for any $d \geq k+1 \geq 3$, prime power $q \geq q_0$, and group $\Gamma$, $X=\buildings^{2d,S_{k}}_q$ has $O(q^{d^2})$ vertices and
    \begin{enumerate}[(i)]
        \item \textbf{Local-spectral Expansion:} $\frac{2}{(j-1)(\sqrt{q}-1)}$-spectral expansion at co-dimension $j$, 
        \item \textbf{Coboundary Expansion:} $h^1(X, \Gamma) \geq \Omega(1)$,
        \item \textbf{Maximum Degree:} $\Delta_{max}(X) \leq q^{O(kd)}$.
    \end{enumerate}
    Moreover, the complex is strongly explicit.
\end{restatable}
In particular, for any $\lambda>0$ we get a family of $\lambda$-local-spectral and coboundary expanders with degree $\exp(k\sqrt{\log(N)\log(1/\lambda)})$. To the best of our knowledge, these are the first sub-polynomial degree high dimensional expanders for which the cohomology is known to vanish for \emph{every group}, rather than for all sufficiently small copies of $\symmetric_m$. As we discuss below, this is an important point toward the construction of better soundness agreement tests and PCPs.

We briefly comment on the proof before moving on, which is divided into two main steps. First, we need to argue that the projected flags complex (and all its links) have \emph{vanishing cohomology} (\Cref{lem:trivial-cohomology}), meaning that every $f: E \to \Gamma$ with $\delta f=0$ is a coboundary $\delta g$ for some $g: V \to \Gamma$. This is a well-known fact for the full flags-complex, and can be derived for projections of the flags complex via fairly standard tools in algebraic topology, e.g., from \cite{bjorner1980shellable,surowski1984covers}. We give a simplified, self-contained proof via \cite{bjorner1980shellable}'s theory of shellability.\footnote{We thank Roy Meshulam for pointing us to this approach.}

Second, we show that the central projections of the flags complex also satisfy ``cocycle'' expansion, which states that if $|\delta f|$ is small, then $f$ must be close to a function in $\text{Ker}(\delta)$ (which due to vanishing cohomology, is exactly the set of coboundaries). The proof relies on several techniques from \cite{evra2016bounded,DD24coboundary} to reduce the analysis to simple local components of the complex. In particular, we first use a local-to-global argument (\Cref{lem:localcoboundary-to-coboundary}) from~\cite{DD24coboundary} to move to links of the complex, and a `color restriction lemma' (\Cref{lem:color-restriction}) to reduce to showing most projections of the links to three colors are coboundary expanders.
Finally, we use the cones method \cite{KO21} (see \Cref{sec:cones}) to argue directly that most projections to three colors indeed have constant coboundary expansion.

\paragraph{Swap Coboundary Expansion.}
Towards construction of agreement testers, \cite{DD24buildings,BM24} define a third notion of expansion called swap coboundary expansion.
Swap coboundary expansion essentially says the faces complexes (defined below) of the simplicial complex under consideration is a 1-coboundary expander.
\begin{definition}
Let $X$ be a $d$-dimensional simplicial complex.
Let $r \leq d$.
We use $\mathrm F^rX$ to denote the simplicial complex whose vertices are $\mathrm F^rX(0) = X(r)$ and whose faces are all $\set{s_0, s_1, ..., s_j}$ such that $s_0 \sqcup s_1\sqcup \dots \sqcup s_j \in X((j+ 1)(r+ 1)-1)$.
We call this the $r$-\emph{faces complex} of $X$.
\end{definition}
Now, we can define swap coboundary expansion.
\begin{definition}
A simplicial complex $X$ is said to have \emph{$(\beta, r)$-swap coboundary expansion} if $\F^r X$
has 1-coboundary expansion at least $\beta$ with respect to every group.
\end{definition}
Finally, we show swap coboundary expansion of the projections of the flags complex.
\begin{theorem}[Swap coboundary expansion of the projected flags complex]
\label{thm:informal-swap-coboundary-projected-flags}
Let $d,k, r$ be integers such that $d > k \geq r^5$.
There is some $q_0 = q_0 (k,r)$ such that the following holds. 
Let $q \geq q_0$ be any prime power, $\buildings$ the $\buildings^{2d}_q$-flags complex, and 
\[
S_{k}=\set{d-k, \dots, d-1, d, d+1, \dots, d+k}.
\]
Then $\buildings^{S_{k}}$ is a $(\exp(-O(\sqrt r \log^2 r)),r)$-swap coboundary expander.
\end{theorem}
\Cref{thm:informal-swap-coboundary-projected-flags} is an adaptation of the strategy of \cite{DD24buildings} for the full un-projected flags complex.
We break the faces complex $\F^r (\buildings^{S_{k}})$ into progressively simpler complexes (projections and deep links) until we can bound the coboundary expansion directly via the cones method. At a high level, the critical observation is simply that once one has moved to the link of a vertex, the projected flags complex behaves very similarly to a product (join) of full flags complexes, allowing us to apply \cite{DD24buildings}'s machinery without substantial modification.

\paragraph{Applications.} Following \cite{DD24a,BM24,BMVY25}, we give two main applications of our swap coboundary expansion result: (i) an agreement test in the low-soundness regime, and (ii) a corresponding simple, near-linear size PCP construction.
\subparagraph{Agreement testing:}
An agreement test is a randomized property tester that, given an ensemble of local functions $\set{f_s : s\to \Sigma | s\in \mathcal F}$ on $\mathcal F \subseteq 2^{[n]}$, tries to determine whether there is a global function $G: [n] \to \Sigma$ such that $f_s = G|_s$ for all $s\in \mathcal F$. Many recent works have examined the case where the local functions are defined on some level of a high-dimensional expander.
In this work, we consider the standard two-query agreement test (the `V test') where we pick two functions $f_s,f_{s'}$ on sets of size $\ell$ where $|s\cap s'| = \sqrt \ell$,
and the test accepts if $f_s|_{s\cap s'}=f_{s\cap s'}$, i.e., the two functions agree on all the common points in their domain.
\cite{DD24a, BM24, BMVY25} recently showed that if a complex has good local-spectral and swap coboundary expansion, then it supports this two-query test. Thus, as a corollary of our results above, we immediately get that the projected flags complex gives rise to a simple near-linear size low soundness agreement tester.
\begin{restatable}[Agreement testing over the projected flags complex]{theorem}{flagsAgrtest}
\label{thm:flags-agrtest}
Fix a $\delta>0$.
There is some $q_0$ such that the following holds. 

Let $q \geq q_0$ be any prime power.
Let $d,k,\ell$ be such that $d> k \geq \poly(\ell)$ and $\ell \geq \exp(\poly(1/\delta))$.
Let $\buildings$ denote the $\buildings^{2d}_q$-flags complex
and 
\[
S_{k}=\set{d-k, \dots, d-1, d, d+1, \dots, d+k}.
\]
Then, the $(\ell, \sqrt \ell)$-agreement test over $\buildings^{S_{k}}(\ell)$ with respect to every alphabet $\Sigma$ has soundness $\delta$,
i.e., if $F : \buildings^{S_{k}}(\ell-1) \to \Sigma^\ell$ passes the $(\ell, \sqrt \ell)$-agreement test with respect to $\buildings^{S_{k}}$ with probability at least $\delta$,
then there is $f : \buildings^{S_{k}}(0) \to \Sigma$ such that
\[
\pr{A\sim \mu_\ell}{\dist(F [A], f |A) \leq \delta} \geq \poly(\delta).
\]
\end{restatable}

We remark that on the complete complex, it is known that the V test has soundness $\delta=1/\poly(\ell)$, rather then $1/\poly\log(\ell)$ as given above. In prior HDX constructions, there is a barrier toward improving the soundness that comes from the cohomology of the complex becoming non-trivial for larger $\symmetric_m=\Gamma$. Since our complexes are coboundary expanders over \emph{every group}, this obstruction vanishes and raises the interesting possibility of showing agreement tests with optimal soundness (an important factor in the eventual goal of constructing subconstant soundness PCPs from HDX). Indeed, it is plausible that 3-query tests like the Z-test \cite{impagliazzo2009new,DN23} may have soundness $2^{-\Omega(\ell)}$.

\subparagraph{Probabilistically Checkable Proofs (PCPs):}
The PCP theorem~\cite{FGLSS96, AS98, ALMSS98, Din07} is one of the most celebrated results in theoretical computer science.
Here, we will take the classic view of PCPs through hardness of approximation for the Label Cover problem.
\begin{definition} 
An instance of \emph{Label Cover} $\Psi = (G = (L\cup R, E), \Sigma_\ell, \Sigma_r, \Phi = \set{\phi_e}_{e \in E} )$ 
consists of a bipartite graph $G$, alphabets $\Sigma_\ell$, $\Sigma_r$ and constraints 
$\phi_e : \Sigma_\ell \to \Sigma_r$, for each edge $e\in E$.
\end{definition}

Given a label cover instance $\Phi$, the goal is to find assignments $A_\ell : L \to \Sigma_\ell$ and $A_r : R \to \Sigma_r$ that satisfy as many of the constraints as possible,
namely that maximize the quantity
\[
\val_{\Phi}(A_\ell, A_r) = \frac{1}{|E|}\big|\set{e = (u, v) \in E ~|~\phi_e(A_\ell(u))= A_r(v)}\big|.
\]
This maximum value is referred to as the value of the instance $\Phi$ and denoted by
\[
\val(\Phi) = \max_{A_\ell,A_r} \val_\Phi(A_\ell, A_r).
\]
When viewed through the label cover problem, the PCP theorem takes the form of NP-hardness of distinguishing between satisfiable label cover instances and instances with value less than $\delta$ provided the alphabet size is large enough.
Proofs of the PCP theorem proceed by reducing the 3SAT problem to a label cover instance. A longstanding goal in such reductions is to achieve small $\delta$ (soundness) without dramatically blowing up the size of the resulting label-cover instance.

\cite{BMVY25} give the first construction of quasi-linear sized PCP with low constant soundness.
They do this by showing how to use a local-spectral expander with few additional structural properties and that supports the aforementioned agreement test to construct such instances.
Finally, they show that variants of the complexes from \cite{CL25} have the required properties. Unfortunately, these complexes also rely on extremely involved algebra. In contrast, our construction is elementary and combinatorial,
and as a corollary to our agreement testing result (\Cref{thm:flags-agrtest}) we obtain the following elementary low-soundness PCP construction.
\begin{restatable}{theorem}{pcpTheorem}
\label{thm:pcp-theorem}
For all $\delta \in (0, 1)$, there exist a constant $C > 0$ such that for all sufficiently large integers $d$ and $n$ there is a polynomial time reduction mapping a 3SAT instance $\phi$ of size $n$ to a label cover instance $\Psi$ over a bipartite graph $G$ of size $n'$ for some $n'\leq  n \exp \paren{C\sqrt{\log n \log \log n}}$, such that
\begin{enumerate}
    \item If $\phi$ is satisfiable, then $\Psi$ is satisfiable.
    \item If $\phi$ is unsatisfiable, then $\val(\Psi)\leq \delta$.
    \item The left alphabet of $\Psi$ is $\Sigma^{\exp(C/\delta)}$ and right alphabet is $\Sigma^{\sqrt {\exp(C/\delta)}}$ for some alphabet $\Sigma$ with $\log |\Sigma| \leq \exp \paren{C\sqrt{\log n \log \log n}}$.
\end{enumerate}
\end{restatable}
We remark that instantiating the base PCP with \cite{Meir16}, one can make the entire construction combinatorial with the exception of the use of a multiplication code, of which only algebraic constructions are known. Finally, we note the alphabet size in the above statement is considerably larger compared to \cite{BMVY25}. There are standard methods to reduce this \cite{dinur2013composition}, but they also rely on algebraic techniques. It would be interesting to give a purely combinatorial method, and we leave this as a direction for future work.

\subsection{Further Related Works}
\paragraph{Simple High Dimensional Expanders.} Beyond the celebrated construction of Kaufman and Oppenheim \cite{KO18} and its variants \cite{OP22,PVB25}, a number of recent works have considered simplified constructions of local-spectral expanders. \cite{LMY20, Gol21} construct bounded degree local-spectral HDX via a product of expanders and the complete complex, but cannot achieve strong enough expansion for most applications \cite{dikstein2024chernoff}. \cite{Gol23} constructs local-spectral HDX of any expansion from low-rank matrices, achieving roughly the same degree as our construction. This was improved in \cite{DLW25} via a recursive procedure to degree $2^{\log(N)^\varepsilon}$ for any constant $\varepsilon>0$. Neither construction has vanishing cohomology, and therefore cannot be used in the construction of PCPs.

\paragraph{Comparison with \cite{OS25}.} 

In recent independent work, O'Donnell and Singer \cite{OS25}, building on \cite{KO25}, showed that \cite{BMVY25}'s PCPs may be instantiated on a variant of the coset-complex HDX construction of Kaufman and Oppenheim \cite{KO18}. These complexes are remarkably simpler than those originally used in \cite{BMVY25}, relying only elementary matrix groups, and achieve parameters matching \cite{BMVY25}'s construction while our flags-complex based PCP is closer to the original construction of Moshkovitz-Raz \cite{MR08}. Nevertheless, \cite{OS25} still requires a fairly substantial amount of algebraic manipulation, and while the complexes are undeniably simpler than \cite{CL25}, they are certainly more involved than the completely elementary subspace-based construction suggested in this work. Finally, as discussed our complexes have vanishing cohomology over \emph{all} groups, rather than just $\symmetric_m$ for small $m$, removing a major barrier toward better agreement tests and PCPs with strong sub-constant soundness.

\section{Preliminaries}
\paragraph{General Notations.} We will use $[n]$ to denote the set $\set{1,2, \dots, n}$.
We will use $\symmetric_A$ to denote the group of permutations on a set $A$.
We will simply use $\symmetric_n$ for the group $\symmetric_{[n]}$.
For a tuple $t$, we will write $\mathsf {set}(t)$ to denote the set of elements that occur in the tuple.
For a set of sets $J$, we write $\cup J=\bigcup_{s\in J}s$.
We abuse the symbol $0$ to mean the subspace $\set{0}$ as well; such usage should be clear from context.
We denote by $\lambda(\mathsf A)$ the second largest eigenvalue of an operator $\mathsf A$.
We denote by $|\lambda|(\mathsf A)$ the second largest eigenvalue of an operator  $\mathsf A$ in absolute value.
We will use $\Fi_q$ to denote the field of order $q$.

Let $A,B$ be two set and $\mathcal D$ and distribution on $A$.
Then, for any two functions $f:A\to B$ and $g:A\to B$, we define
\[
\dist_{\mathcal D}(f,g)=\pr{x\sim \mathcal D}{f(x)\ne g(x)}.
\]
When the distribution is not specified the uniform distribution is assumed.

\subsection{Graphs and Expansion}
We will use $V(G)$ and $E(G)$ to denote the vertices and the edges of a graph $G$, respectively.
We will drop the parenthetical and simply write $V$ and $E$ when $G$ is clear from the context.
We will use $\Delta(G)$ to denote the max degree of the graph $G$.
We will use $\diam(G)$ denote the diameter of the graph $G$, i.e., the maximum (over all pairs of vertices) length of shortest paths.

\paragraph{Probability density over Graphs.}
We can define an arbitrary probability density $\pi_1$ over the edges of a graph.
This induces a distribution $\pi_0$ over the vertices in which we first sample an edge $\set{u,v}\sim \pi_1$ and then uniformly pick $u$ or $v$.
Note that this is the stationary distribution for the random walk on the graph.

\paragraph{Partite Graphs.}
A \emph{$k$-partite} graph is a graph such that we can decompose $V=V_1\sqcup V_2\sqcup \dots \sqcup V_k$
such that for every edge $\set{u,v}\in E$ we have $u\in V_i$ and $v\in V_j$ for $i\ne j$.
If $v\in V_i$, we define the color $\col(v)=i$.
Let $c\subseteq [k]$ be a set of colors.
We define the \emph{projection} of a graph $G$ onto $c$ as the graph $G^c$ induced by the vertex set $V(G^c)=\set{v\in V| \col(v)\in c}$.
We will often shorten $V(G^c)$ as $V^c$ and $E(G^c)$ as $E^c$.

\paragraph{Expansion.}
Let $G = (V, E)$ be a graph and let $\pi_1 : E \to (0, 1]$ be a probability distribution.
Let $\mathsf A_G$ be the \emph{normalized adjacency operator}.
This operator takes as input $f : V \to \R$ and outputs $\mathsf A_Gf : V \to \R$, with 
\[
\mathsf A_Gf(v) = \frac{1}{2\pi_0(v)} \sum_{u\sim v}\pi_1(uv)f(u).
\]
The (normalized) adjacency operator is self adjoint with respect to the inner product on $\ell_2(V) = \set{f : V \to \R}$ given by
\[
\inprod{f, g}=\sum_{v\in V}\pi_0(v)f(v)g(v).
\]
We write $\lambda(G):=\lambda(\mathsf A_G)$ and $|\lambda|(G):=|\lambda|(\mathsf A_G)$ for the the second largest eigenvalue and the second largest absolute value of the eigenvalue of the (normalized) adjacency operator, respectively.
We say that $G$ is a \emph{$\lambda$-one sided spectral expander} if for every $\lambda(G) \le \lambda$
and say that $G$ is a \emph{$\lambda$-two sided spectral expander} if $|\lambda|(G) \leq \lambda$.
We say that $G$ is an $\eta$-edge expander if for every subset $S \subseteq V$, $S\ne \emptyset$, it holds that 
\[
\pr{uv\sim \mu_1}{u\in S, v\in V \setminus S} \geq \eta \pr{\mu_0}{S} \pr{\mu_0}{V \setminus S}.
\]
Spectral and edge expansion are classically related to each other via Cheeger's Inequality \cite{Alo86,AM85}. We will only use the following `easy direction' of Cheeger.
\begin{lemma}[Folklore]
\label{lem:spectral-to-edge}
Let $G$ be a $\lambda$-one sided spectral expander, then $G$ is a $(1-\lambda)$-edge expander.
\end{lemma}

We will also need the standard expander-mixing lemma, which states that the weight of edges between any two sets $A,B$ is roughly the expected weight based on the sizes of $A$ and $B$
\begin{lemma}[Expander Mixing Lemma]
\label{lem:expander-mixing}
Let $G = (U, V, E)$ be a bipartite graph in which the $\lambda(G)\leq \lambda$, and let $\mu$ be the stationary distribution over $G$. Then for all $A \subseteq U$ and $B \subseteq V$ we have that

\[
\left|\pr{(u,v)\in E}{u \in A, v \in B} - \mu(A)\mu(B)\right| \leq \lambda \sqrt{\mu(A)(1 - \mu(A))\mu(B)(1 - \mu(B))}.
\]
\end{lemma}
We also note a well-known sampling property of the bipartite expanders.
\begin{lemma}
\label{lem:expander-sampling}
Let $G = (U, V, E)$ be a weighted bipartite graph with $\lambda(G)\leq \lambda$.
Let $B \subseteq U$ be a subset with $\mu(B) = \delta$ and set
\[
T = \set{v\in V \left|\pr{u\text{ neighbour of }v}{u \in B} - \delta > \eps\right.} .
\]
Then $\Pr{T} \leq \lambda^2 \delta/\eps^2$.
\end{lemma}

\paragraph{Majority and expansion.}

A second very useful fact about expander graphs is that local agreement on the graph implies global agreement with majority.
Let $G= (V, E)$ be a graph and $g : V \to \set{1, 2,..., n}$ be some function.
Let $S_i = \set{v\in V | g(v) = i}$.
The majority assignment $\maj(g) \in \set{1, 2,..., n}$ is the $i$ such that $\Pr{S_i}$ is largest (ties broken arbitrarily).
Observe that if $\pr{v}{g(v) = \maj(g)} \approx 1$ then for most edges $uv \in E$,
it holds that $g(v) = g(u)$ (since with high probability they are both equal to $\maj(g)$).
In expander graphs a converse to this statement also holds.
That is, if for most edges $g(v) = g(u)$, then $\pr{v}{g(v) = \maj(g)} \approx 1$.
\begin{lemma}
\label{lem:expansion-majority}
Let $G$ be an $\eta$-edge expander.
Let $S_1, S_2,\dots, S_m$ be a partition of the vertices of $V$.
Assume that 
\[
\pr{uv\sim \mu_1}{\exists i, u\in S_i\text{ and }v \not \in S_i} \leq \eps.
\]
Then, there exists $i$ such that $\Pr{S_i} \ge 1-\frac{\eps}{\eta}$.
Stated differently, for every $g : V \to \set{1, 2,..., n}$, we have
\[
\pr{v}{g(v) \ne \maj(g)} \leq \pr{uv\sim \mu_1} {g(u) \ne g(v)}.
\]
\end{lemma}

\subsection{Basics of Simplicial Complexes}
We will now describe simplicial complexes and related notions.
Our exposition and notation in this section closely follow~\cite{DD24coboundary, DD24buildings}.
\paragraph{Simplicial Complexes.}

A \emph{simplicial complex} $X$ is a downward closed set system, i.e., if $s\in X$ and $t\subseteq s$, then $t\in X$.

\begin{itemize}
    \item The sets in a simplicial complex are called \emph{faces}. The \emph{dimension} of a face $s$ is $|s|-1$. 
    \item The dimension of a complex is the dimension of the largest face. In a \emph{pure} simplicial complex, all maximal faces have the same dimension. 
    \item $X(i)$ is the set of all the faces of dimension $i$ of $X$. In particular, we have $X(-1)=\set{\emptyset}$. We will sometimes refer to $X(i)$ as the level $i$ of the complex.
    \item The set of \emph{oriented} $k$-faces is $\overrightarrow X(k)=\set{(v_0, v_1, \dots, v_{k})| \set{v_0, v_1, \dots, v_k}\in X(k)}$.
    \item Let $s=(v_0,v_1, \dots, v_k)$ be an oriented face and $i\in \set{0, 1, \dots, k}$ be an index. Then, we get the face $s_i$ by removing $v_i$ from $s$.
    \item For convenience, typically we drop the set brackets and the tuple brackets for faces.

\end{itemize}
We will use $\Delta_n$ to denote the \emph{complete complex} on $n$ vertices, i.e., the $n-1$-dimensional complex with $[n]$ as the top face.
\paragraph{Partite Complexes.}
A $(d + 1)$-partite simplicial complex is a $d$-dimensional simplicial complex such that one can decompose $X(0) = A_0 \sqcup A_1 \sqcup \dots \sqcup A_d$ such that for every $s \in X (d)$ and $i \in \set{0, 1,\dots, d}$ it
holds that $|s \cap A_i | = 1$.
The color of a vertex $\col (v) = i$ such that $v \in A_i$.
More generally, the color of a face $s$ is $c = \col (s) = \set{\col (v) | v \in s}$.
We denote by $X^c$, the \emph{projection} of X to $c$, the set of faces $s\in X$ with
 the color $\col(s)\subseteq c$, and for a singleton $\set{i}$ we sometimes write $X^i$ instead of $X^{\set{i}}$.
Note $X^c$ is a $|c|$-partite complex.

\paragraph{Links, Underlying Graphs and Skeletons.} Let $X$ be a $d$-dimensional simplicial complex.
For $k \leq d$, the $k$-\emph{skeleton} of $X$ is the complex $X ^{\leq k} = \bigcup^k_{j=-1} X(j)$.
When $k = 1$, we call this complex the \emph{underlying graph} of $X$, as it consists of the vertices and edges in $X$ (as well as the empty face). A $d$-dimensional complex is called a \textit{clique complex} if every $(d+1)$-clique in its underlying graph is a face in $X$.
The \emph{diameter} $\diam(X)$, and the \emph{max degree $\Delta(X)$} are the diameter $\diam(G=(X(0),X(1)))$, and the max degree $\Delta(G)$ of the underlying graph of $X$.
The \emph{link} of a face $s\in X$ is the $d' = (d - |s|)$-dimensional complex
$X_s = \set{t \setminus s | t \in X, t \supseteq s}$.

\paragraph{Joins.}
Given two simplicial complexes $X_1$ and $X_2$, we define the \emph{join} $X_1*X_2$ to be the simplicial complex defined on the vertices $X_1(0) \sqcup X_2(0)$ with
\[
X_1*X_2= \set{\sigma\sqcup \tau | \sigma \in X_1, \tau\in X_2}.
\]
Extend this definition inductively to define joins of multiple complexes.

\paragraph{Probability density over a simplicial complex.}
Let $X$ be a $d$-dimensional (pure) simplicial complex.
The top faces can be assigned an arbitrary probability density $\pi_d$.
The probability density $\pi_d$ induces a probability density $\pi_k$ on the $k$-dimensional faces in which we first sample $t\sim \pi_d$ and then uniformly sample a subface $s\subseteq t$. We observe that 
\[
\pi_{k}(s)=\frac{1}{\binom{d+1}{k+1}}\sum_{t\supseteq s} \pi_{d}(t).
\]
This also defines a probability density on oriented faces where all permutations of a face have the same probability. So, we have $\pi_{k}(s)=\frac{1}{(k+1)!}\pi_{k}(\textsf{set}(s))$ for any face $s\in \overrightarrow X(k)$.
For a simplicial complex $X$ with a measure $\pi_d : X (d) \to (0, 1]$, the induced measure
$\pi_{d',X_s} : X_s (d - |s|) \to (0, 1]$ on $X_s$, for $t\in X_s(d-|s|)$, i.e., the probability $t\sim \pi_{d',X_s}$ is the probability of sampling $t\sqcup s$ from $\pi_d$ conditioned on containing $s$.
Observe that we have
\[
\pi_{d',X_s}(t \setminus s)= \frac{\pi_{d}(t)}{\sum_{t'\supseteq s}\pi_{d}(t')}.
\]
Whenever the complex $X$ is clear we shall write $\pi_{d',s}$ instead of $\pi_{d',X_{s}}$.

\subsection{Expansion on a Simplicial Complex}
\label{sec:simplicial-expansion}
Expansion on a simplicial complex can be defined in numerous ways.
We discuss \emph{local spectral expansion} and \emph{coboundary expansion} here.

\paragraph{Local Spectral Expansion.} The notion of local spectral expansion requires that underlying graph of each of its links is a spectral expander.
Before formally defining local spectral expansion, let us setup some notation.
We use $\lambda(X_s)$ to denote the second largest eigenvalue of the (normalized) adjacency operator of the underlying graph of $X_s$
and we use $|\lambda|(X_s)$ to denote $|\lambda|(G_{X_s})$, where $G_{X_s}$ is the underlying graph of $X_s$.

\begin{definition}[Local Spectral Expansion {\cite{KKL16,Opp18}}]
\label{def:local-spectral-expansion}
Let $X$ be a $d$-dimensional simplicial complex and let $\lambda \in (0, 1)$.
$X$ is a $\lambda$-one sided local spectral expander if for every $s \in X ^{\leq d-2}$ it holds that $\lambda(X_s) \leq \lambda$.
Similarly, $X$ is a $\lambda$-two sided local spectral expander if for every $s \in X ^{\leq d-2}$ it holds that $|\lambda|(X_s)\leq \lambda$.
\end{definition}
\begin{remark}
In the literature, the term \emph{high-dimensional expander} is also sometimes used to mean local spectral expanders. 
\end{remark}

The following useful result relates the expansion of the links of faces with codimension 2 with local spectral expansion of the entire complex.
\begin{theorem}[Trickle Down Theorem \cite{Opp18}]
\label{thm:trickle-down}
Let $X$ be a $d$-dimensional simplicial complex such that the 1-skeleton of every link is connected and for every $s\in X(d-2)$, the link $X_s$ is a one-sided $\lambda$-expander.
Then X is a $\mu$-local spectral expander, where $\mu=\frac{\lambda}{1-(d-1)\lambda}$.
\end{theorem}

We will also use the following variant of local-spectral expansion for partite complexes which asks that the projection of any link onto two colors is a good expander.

\begin{definition}[$\lambda$-product {\cite{GLL22}}]\label{def:lambda-product}
  A $d$-partite complex $X$ is called a $\lambda$-product if for every link $X_\tau$ of co-dimension
  at least 2 and colors $i, j \not \in  \col(\tau)$ , we have:
  \[
    \lambda(X^{ij}_\tau)\leq \lambda.
  \]
\end{definition}

\paragraph{Coboundary Expansion.}
Another useful notion of expansion in a simplicial complex is coboundary expansion.
Let $X$ be a $d$-dimensional simplicial complex for $d \geq 1$ and let $\Gamma$ be any group.
First, we define the notion of cochains.
For $i=-1,0,$ the set of $i$-\emph{cochains} is $C^i(X, \Gamma) = \set{f:X(i)\to \Gamma}$.
We sometimes identify $C^{-1}(X, \Gamma)\cong \Gamma$.
For $i = 1, 2,$ the $i$-cochains are defined as follows.
The 1-cochains
\[
C^1(X,\Gamma)=\set{f:\overrightarrow{X}(1)\to \Gamma ~\big|~ f(u,v)=f(v,u)^{-1}}
\]
and the 2-cochains
\[
C^2(X,\Gamma) = \set{f:\overrightarrow{X}(i) \to \Gamma ~\big|~ \forall \pi \in \symmetric_3, (v_0,v_1,v_2) \in X(2)\quad f(v_{\pi(0)} , v_{\pi(1)} , v_{\pi(2)} ) = f(v_0, v_1, v_2)^{\sign(\pi)}}
\]
are the spaces of the so-called anti-symmetric functions on edges and triangles, respectively.
For $i = -1, 0, 1,$ we define the `coboundary operators' as
$\delta_i:C^i(X, \Gamma) \to C^{i+1}(X, \Gamma)$ by
\begin{enumerate}[(i)]
    \item $\delta_{-1}:C^{-1}(X, \Gamma) \to C^0 (X, \Gamma)$ is $\delta_{-1}h(v) = h(\emptyset)$,
    \item $\delta_0:C^0(X,\Gamma)\to C^1(X,\Gamma)$ is $\delta_0 h(v, u) = h(v )h(u)^{-1}$,
    \item $\delta_1:C^1(X,\Gamma)\to C^2 (X, \Gamma)$ is $\delta_1 h(v, u, w) = h(v, u)h(u, w)h(w, v)$.
\end{enumerate}
We write just $\delta f$ for $\delta_i f$ when the dimension $i$ is clear from context.

Although we want results on a general group $\Gamma$, it is instructive to consider the case where $\Gamma = \mathbb F_2$.
Specifically, in this case, each $i$-cochain $f$ is an indicator for some subset $S_f\subseteq X(i)$ of $i$-faces.
Furthermore, the coboundary of an 0-cochain $f$ corresponds to the indicator edges crossing the set $S_f$.

Let $Id = Id_i \in C^i (X, \Gamma)$ be the function that always outputs the identity element.
It is easy to check that $\delta_{i+1} \circ \delta_i h = Id_{i+2}$ for all $i = -1, 0$ and $h \in C^i (X, \Gamma)$.
Thus, if we write
\[
Z^i(X,\Gamma)=\ker(\delta_i)\subseteq C^i,
\]
and
\[
B^i(X,\Gamma)=\im(\delta_{i-1})\subseteq C^i,
\]
for the sets of the so-called \emph{cocycles} and \emph{coboundaries}, respectively, we see that $B^i\subseteq Z^i$.
We say $i$-th cohomology is \emph{trivial} whenever $B^i = Z^i$.

\begin{definition}[Cocycle and coboundary expansion {\cite{LM06,Gromov2010}}]
\label{def:coboundary-expansion}
Let $X$ be a $d$-dimensional simplicial complex for $d \geq 2$. Let $-1\leq k\leq d-1$.
The $k$-cocycle expansion $h^k(X, \Gamma)$ of $X$ with respect to a group $\Gamma$ is given by
\[
    h^k(X,\Gamma)= \min_{f\in C^k\setminus Z^k} \frac{\wt(\delta f)}{\dist(f, Z^k)}.
\]
Additionally, if the $k$-th cohomology is trivial, i.e., we have $B^k(X,\Gamma) = Z^k(X, \Gamma)$, then we say $h^k(X, \Gamma)$ is the $k$-coboundary expansion of $X$ with respect to a group $\Gamma$.
\end{definition}
We remark the above is often called `cosystolic' rather than `cocycle' expansion in the literature. However the former originally referred to complexes with other additional properties \cite{KKL16}, so there has been a recent shift in the it the literature to using `cocycle expansion' for this notion instead.

\subsection{Walks on Simplicial Complexes.}
Simplicial complexes come equipped with several natural families of random walks generalizing the standard
random walk on a graph.

\paragraph{Up-down and Down-up walks.}
Let $\ell \leq k \leq d$.
The $(k, \ell)$-\emph{containment graph} $G_{k,\ell} = G_{k,\ell}(X)$ is the bipartite graph whose vertices are $L= X(k)$, $R= X(\ell)$ and whose edges are all $(t, s)$ such that $t \supseteq s$.
The distribution we associate to the edges is the natural distribution induced by the complex $X$, i.e.,
sampling $t\sim \pi_k$ and then uniformly sampling $s\subseteq t$ of size $|s|= \ell + 1$.

The \emph{down operator} $\mathsf D_{k,\ell} : \ell_2(X(k)) \to \ell_2(X(\ell))$ is the bipartite operator of
the containment graph $G_{k,l}$, i.e.,
\[
\mathsf D_{k,\ell}f (s) = \ex{t \supseteq s}{f (t)}.
\]
The adjoint of $\mathsf D_{k,\ell}$ is the \emph{up operator}, $\mathsf U_{\ell,k} : \ell_2(X(\ell)) \to \ell_2(X(k))$, given by:
\[
\mathsf U_{\ell,k} g(t) = \ex{s \subseteq t} {g(s)}.
\]
Then the composition of the up and down operators $\mathsf U_{\ell,k} \mathsf D_{k,\ell}$ is called the `down-up' walk, and the `down-up' walk is similarly defined.

\paragraph{Swap Walks.}
Let $X$ be a $d$-dimensional simplicial complex.
Let $i, j$ be so that $i + j + 1 \leq  d$.
The swap walk $\mathsf S_{i,j} = \mathsf S_{i,j}(X)$ is the bipartite adjacency operator of the graph $(X(i), X(j), E)$, where an edge $\set{s_i, s_j}$ is chosen in this graph by first selecting a face $t \sim \pi_{i +j+1}$, and then partitioning it to $t= s_i \sqcup s_j$ where $s_i \in X(i)$ and $s_j \in X(j)$.
This walk was defined and studied independently by \cite{AJT19} and \cite{DD19}.
\begin{lemma}[\cite{GLL22}]
Let $X$ be a $\lambda$-two-sided local spectral expander.
Then the second largest eigenvalue of $\mathsf S_{k,\ell}(X)$ is upper bounded by $\sqrt{(k+ 1)(\ell+ 1)}\lambda$.
\end{lemma}

\paragraph{Colored Swap Walks.}
The standard swap walks are not well behaved on partite complexes, but there is a useful analog for this setting defined in \cite{DD19}.
Let $X$ be a $d$-partite simplicial complex, and $F_1,F_2 \subseteq [d]$ be two disjoint subsets.
The colored swap walk $\mathsf S_{F_1,F_2} = \mathsf S_{F_1,F_2} (X )$ is the bipartite adjacency operator of the graph $(X^{F_1}, X^{F_2}, E)$, where an edge $\set{s_1, s_2}$ is chosen in this graph by first selecting a face $t \in X^{F_1 \sqcup F_2}$, and then partitioning it to $t= s_1\sqcup s_2$ according to its colors.

\begin{lemma}[\cite{GLL22}]
\label{lem:color-swap-expansion}
Let $X$ be a $d$-partite $\lambda$-product, and $F_1,F_2\subseteq [d]$ disjoint subsets. Then \[\lambda(\mathsf S_{F_1 ,F_2}) \leq \sqrt{|F_1 | |F_2 |} \lambda. \]
\end{lemma}

\subsection{Cones}
\label{sec:cones}
Cones are the main machinery used in most proofs of coboundary expansion. We define (non-abelian) cones.
First, we define the composition of two paths $P_0 = (v_0, v_1,\dots, v_m)$ and $P_1 = (v_m, v_{m+1},\dots, v_n)$ by $P_0 \circ P_1 = (v_0, v_1,\dots, v_m, v_{m+1},\dots, v_n)$.
Note that composition is only defined when the end point of the first path is equal to the starting point of the
second.

Now, fix a simplicial complex $X$ and some $v_0 \in X(0)$.
We define two symmetric relations on loops around $v_0$:
\begin{description}
    \item[(BT)] We say that $P_0 \overset{BT}{\sim} P_1$ if $P_i = Q_0 \circ (u, v, u) \circ Q_1$ and
    $P_{1-i} = Q_0 \circ (u) \circ Q_1$ for $i= 0, 1$ (i.e. going
from $u$ to $v$ and then backtracking is trivial).
    \item[(TR)] We say that $P_0 \overset{TR}{\sim} P_1$ if $P_i = Q_0 \circ (u, v) \circ Q_1$
    and $P_{1-i} = Q_0 \circ (u, w, v) \circ Q_1$ for some triangle $uvw \in X(2)$ and $i= 0, 1$.
\end{description}
We denote by $P\sim_1 P'$ if there is a sequence of loops $(P_0 = P, P_1, \dots, P_m = P')$ and $j \in [m-1]$ such that:
\begin{enumerate}
    \item $P_j \overset{TR}{\sim} P_{j+1}$ and
    \item For every $j' \ne j$, $P_{j'} \overset{BT}{\sim} P_{j'+1}$. 
\end{enumerate}
In other words, We can get from $P$ to $P'$ by a sequence of `equivalences', where exactly one equivalence is by (TR).

Finally, for a path $P = (u_0, u_1, \dots, u_m)$ on the graph underlying $X$,
we denote by $P^{-1}$ the inverse path $(u_m,\dots, u_1, u_0)$.

\begin{definition}[Cone]
A cone is a triple $C = (v_0, \set{P_u}_{u \in X(0)}, \set{T_{uw}}_{uw\in X(1)})$ such that
\begin{enumerate}
    \item $v_0 \in X(0)$,
    \item for every $v_0 u \in X(0)$, we have a walk $P_u$ from $v_0$ to $u$,
    \item for $u= v_0$, we take $P_{v_0}$ to be the loop with no edges from $v_0$,
    \item and for every $uw \in X(1)$, $T_{uw}$ is a sequence of walks $(P_0, P_1, \dots, P_m)$ such that
    \begin{enumerate}
        \item $P_0 = P_u\circ (u, w) \circ P^{-1}_w$,
        \item for every $i= 0, 1, \dots, m-1$, $P_i \sim_1 P_{i+1}$ and
        \item $P_m$ is equivalent to the trivial loop by a sequence of (BT) relations.
    \end{enumerate}
\end{enumerate}
We call $T_{uw}$ a contraction, and we write $|T_{uw}|= m$.
The diameter of a cone is defined as follows.
\[
\mathrm {diam}(C) = \max_{uw\in X(1)} |T_{uw}|.
\]
\end{definition}

\begin{lemma}[\cite{KO21,DD24buildings}]
\label{lem:cone-argument}
Let $X$ be a simplicial complex such that $\mathrm{Aut}(X)$ is transitive on $k$-faces.
Suppose that there exists a cone $C$ with diameter $R$.
Then the 1-coboundary expansion of $X$ at least $\frac{1}{\binom{k+1}{3}\cdot R}$.
\end{lemma}

\subsection{Faces Complexes}
\begin{definition}\label{def:faces}
Let $X$ be a $d$-dimensional simplicial complex.
Let $r \leq d$.
We use $\mathrm F^rX$ to denote the simplicial complex whose vertices are $\mathrm F^rX(0) = X(r)$ and whose faces are all $\set{s_0, s_1, ..., s_j}$ such that $s_0 \sqcup s_1\sqcup \dots \sqcup s_j \in X((j+ 1)(r+ 1)-1)$.
We call this the $r$-\emph{faces complex} of $X$.
\end{definition}
Observe that this is a $\paren{\floor{\frac{d+1}{r+1}}-1}$-dimensional complex.

\paragraph{Probability density over a faces complex.}
Let $X$ be a $d$-dimensional simplicial complex, and let $r < d$.
The distribution on the top-level faces of $\mathrm F^r X$ is given by the following.
Let $m= \floor{\frac{d+1}{r+1}}-1$.
\begin{enumerate}
    \item Sample a $d$-face $t= \set{v_0, v_1,\dots , v_d} \in X(d)$.
    \item Sample $s_0, s_1, \dots, s_m \subseteq t$ such that $|s_i|= r+ 1$,
    $s_i\cap s_j= \emptyset$ and output $\set{s_0, s_1, \dots, s_m}$.
\end{enumerate}
A faces complex can be viewed as a subcomplex of the following (non-pure) simplicial complex.
\begin{definition}
Let $X$ be a $d$-dimensional simplicial complex.
We denote by $\mathrm FX$ the simplicial complex whose vertices are $\mathrm FX(0) = X\setminus \set{\emptyset}$ and whose faces are all $\set{s_0, s_1, ..., s_j}$ such that $s_0 \sqcup s_1\sqcup \dots \sqcup s_j \in X$.
\end{definition}
This complex is not pure in general, hence we do not define a distribution on it.

\paragraph{Links of a faces complex.}
The links of a faces complex correspond to a faces complexes of the links of the original complex.
More precisely, we have the following statement.
\begin{observation}
\label{obs:face-link}
Let $s \in \mathrm FX$. Then $\mathrm F X_s = \mathrm F (X_{\cup s})$, where $\cup s = \bigcup_{t\in s} t$.
Furthermore, if $s\in \F^r X$, then we have $\mathrm F^rX_s = \mathrm F^r (X_{\cup s})$.
\end{observation}
We introduce the notion of \emph{generalized links} to simplify notation for links of faces complex.
\begin{definition}[Generalized links]
Let $w \in X$.
We write $\mathrm FX_w$ for $\mathrm F(X_w)$.
We also write $\mathrm F^rX_w$ for $\mathrm F^rX\cap \mathrm FX_w$.
\end{definition}
Note that $\mathrm F^rX_w$ is not necessarily a proper link of $\mathrm F^rX$.
However, using this notation we can write the conclusions of~\Cref{obs:face-link} more compactly as $\mathrm F X_s = \mathrm F X_{\cup s}$ and $\mathrm F^rX_s = \mathrm F^r X_{\cup s}$ for a face $s\in \F^r X$.

\paragraph{Colors on a faces complex.}
Let $X$ be a $k$-partite complex.
We denote the set of colors of $\F^rX$ by $\fcolors = \binom{[k]}{r+1}$.
Fix a set $J \subset \Delta_k$, denoted $J= \set{c_1, \dots, c_m}$, such that $c_j \subset [k]$ are pairwise disjoint.
Let $\F^JX=\set{s\in \F X | \col(s) \subseteq J}$ be the sub-complex of $\F X$ whose vertex colors are in $J$, so $\F^JX(0)=\bigcup^m_{j=1} X^{c_j}$.
We call such a complex a $J$-faces complex and also a \emph{colored} faces complex.
We will be particularly interested in the case where $J \in \F^r\Delta$, namely, $J$ consists of pairwise disjoint subsets of cardinality $r+1$.
In this case, $\F^JX$ is $|J|$-partite.
We abuse notation in this section allowing multiple $c_j$’s to be empty sets.
In this case $X^{c_j}$ are copies of ${\emptyset}$, and every empty set is in all top level faces of $\F^JX$.

The measure induced on the top level faces of $\F^J X$ is obtained by sampling $t \in X^{\cup J}$ and
partitioning it into $t= s_1\sqcup s_2\sqcup \dots \sqcup s_m$ such that $s_i \in X^{c_i}$.
Finally, throughout the paper, we use the following notation.
Let $J', J \subseteq \F \Delta$.
We write $J' \leq J$, if $J = \set{c_1, c_2, \dots, c_m}$ and $J'= \set{c'_1,\dots, c'_m}$ where $c'_j \subseteq c_j$.

\paragraph{Swap Coboundary Expansion.} We end this section by defining swap coboundary expansion.
\begin{definition}
A simplicial complex $X$ is said to have \emph{$(\beta, r)$-swap coboundary (cocycle) expansion} if $\F^r X$
has 1-coboundary (cocycle) expansion (with respect to every group) at least $\beta$.
\end{definition}
In general, our goal is to prove swap co-boundary expansion for $\beta \geq 2^{-o(r)}$, which is the pre-condition to a $1\%$ agreement test and resulting PCP construction \cite{BM24,DD24a,BMVY25}.

\section{Expansion of the Projected Flags Complex}

\subsection{Local Spectral Expansion Bound}

In this section, we show \Cref{thm:spectral-projected}, that any projection of the flags complex is a good local-spectral-expander.
\spectralProjected*
Towards proving \Cref{thm:spectral-projected}, we first show \Cref{prop:kpartite-expansion} which relates the spectral expansion of a $k$-partite graph to its bipartite components.
\begin{restatable}{proposition}{kpartExpansion}
\label{prop:kpartite-expansion}
Let $G$ be a $k$-partite graph with the vertex-set $V=\bigsqcup_{i\in [k]}P_i$ with (normalized) edge weights given by $\pi_2:E\to [0,1]$ and stationary distribution given by $\pi_1:V\to [0,1]$.
Suppose $\lambda_2(G^{\set{i,j}}) \leq \lambda_{ij}$ for all $i \neq j \in [k]$.
Let $K$ be the complete graph on $k$ vertices with edge weight $\kappa_2(i,j) = \pi_2(E^{\set{i,j}})$.
Then, the second eigenvalue of $G$ is at most 
\[
\lambda_2(G) \leq \lambda_2(K)+ \max_{i\in [k]}\ex{j\sim \kappa_2[i]}{\lambda_{ij}}.
\]
where $\kappa_2[i]$ denotes the distribution on the neighbors of $i$ induced by $\kappa_2$.
\end{restatable}

The proof is fairly elementary linear algebra, and we include it for completeness in \Cref{app:omitted}. We will also need the following bound on the expansion of the Grassmann graphs.

\begin{fact}[\cite{Del76, GHKLZ22}]
\label{fact:grassmann_expansion}
Fix any $d\geq 3$. For any $i<j\leq d-1$, the second eigenvalue of the underlying graph of $\buildings_d^q[i,j]$ is at most $q^{-(j-i)/2}$.
\end{fact}

\begin{proof}[Proof of \Cref{thm:spectral-projected}]

The underlying graph of the complex is a $k$-partite graph where $k=|S|$ with vertex-set $V=\bigsqcup_{i\in S}P_i$ for $P_i$ subspaces of dimension $i$ with edges given by subspace inclusion and weight given by the number of chains containing end points of the edge with appropriate normalization.
Moreover, the weights $\pi_2(E^{\set{i,j}})$ are all equal.
Thus, using \Cref{prop:kpartite-expansion} we can bound the second eigenvalue of this graph by

\begin{align*}
0+\max_{i\in S}\ex{j\in S\setminus \set{i}}{\lambda_{ij}}&\leq \max_{i\in [d-1]}\ex{j\in S\setminus \set{i}}{\lambda_{ij}}\\
&\leq \frac{1}{(|S|-1)}\max_{i\in [d-1]}\sum_{j\in S\setminus\set{i}}\lambda_{ij}\\
&\leq \frac{1}{(|S|-1)}\max_{i\in [d-1]}\sum_{j\in [d-1]\setminus\set{i}}\lambda_{ij}
\end{align*}
Using the bounds on random walks on the projections of flags complex stated in \Cref{fact:grassmann_expansion} we get the final bound of
\begin{align*}
\frac{1}{(|S|-1)}\max_{i\in [d-1]}\sum_{j\in [d-1]\setminus\set{i}}\lambda_{ij} &\leq \frac{1}{(|S|-1)}\max_{i\in [d-1]}\left(\sum_{j=1}^{i-1}q^{-(i-j)/2}+\sum_{j=i+1}^{d-1}q^{-(j-i)/2}\right)\\
&\leq \frac{1}{(|S|-1)}\max_{i\in [d-1]} \frac{(1-q^{-(i-1)/2})+(1-q^{-(d-i-1)/2})}{\sqrt{q}-1}\\
&\leq \frac{2}{(|S|-1)(\sqrt{q}- 1)}.
\end{align*}

Now, consider a face $(W_{i_1} \subseteq W_{i_2} \subseteq \ldots \subseteq W_{i_{\ell}})$ of $\buildings^S(\ell-1)$.
Notice that the link of this face is isomorphic to the join of projected complexes of form $\buildings^{i_j-i_{j-1},S_j}$ where $S_j\subseteq [i_j-i_{j-1}]$ with $i_0=0$. 
As before, the underlying graph is some $(k-\ell)$-partite graph $G$ with each $G^{\set{p,q}}$ being either the underlying graph of $\buildings^{i_j-i_{j-1},S_j}$ or a complete bipartite graph.
Again, similarly applying \Cref{prop:kpartite-expansion}, \Cref{fact:grassmann_expansion}, and the fact that the second eigenvalue of a complete bipartite graph is 0 gives us the required bound of $\frac{2}{(|S|-\ell-1)(\sqrt{q}-1)}$.
This observation about the structure of the links, \Cref{fact:grassmann_expansion},
and the fact that the second eigenvalue of a complete bipartite graph is 0 also lets us conclude that the complex is a $\frac{1}{\sqrt q}$-product.
\end{proof}

\subsection{Coboundary Expansion}
In this section we prove \Cref{thm:coboundary-spherical}, that central projections of the Flags complex and their links have good 1-coboundary expansion over any group.
\coboundarySpherical*

Broadly speaking, the proof of \Cref{thm:coboundary-spherical} has two main steps, vanishing cohomology (\Cref{lem:trivial-cohomology}) and co-cycle expansion (\Cref{lem:localcoboundary-to-coboundary} + \Cref{lem:coboundary-link-spherical}). We start with the former, which we prove via the theory of shellability.

\subsubsection{Shellability of the Flags Complex}

In this section, we prove projections of the flags complex have trivial cohomology over any group $\Gamma$.

\begin{restatable}{lemma}{trivialCohomology}
\label{lem:trivial-cohomology}
Fix any $d\geq 3$, any prime power $q$, any group $\Gamma$, and a set $S\subseteq [d-1]$ of size at least $2$.
Let $s\in\buildings^{d,S}_q$ be any face of codimension at least 3.
Then, the 1-cohomology of $\buildings^{d,S}_{q,s}$ with respect to $\Gamma$ is trivial.
\end{restatable}

\paragraph{Shellability and Vanishing Cohomology.}

To prove \Cref{lem:trivial-cohomology}, it is enough to show projections of the Flags complex are \emph{shellable}, a condition roughly stating that the complex can be built up one (top-level) face at a time in such a way that each new face is glued to the prior faces along co-dimension 1 facets. It is a standard fact that shellable complexes have vanishing 1-cohomology, since they are homotopy equivalent to a wedge of spheres \cite{bjorner1980shellable}. Here, we give a simple, self-contained proof of vanishing cohomology that does not rely on the latter machinery.

First, we formalize the standard notion of shellability.

\begin{definition}[Topological Shellability]
Suppose that $X$ is a simplicial complex. A total ordering $\sigma_1,\sigma_2,\dots, \sigma_k$ of maximal faces of $X$ is a \emph{topological shelling} of $X$ if 
\[
\dim(\sigma_1)\geq \dim(\sigma_2)\geq \dots \geq \dim(\sigma_k),
\]
and for any $j>1$,
\[
\paren{\bigcup_{i<j}\sigma_i}\cap \sigma_j
\]
is a pure $(\dim(\sigma_j)-1)$-dimensional complex.\footnote{Note the union denotes the simplicial complex given by taking all the $\sigma_i$ and their downward closures.}
A complex $X$ is called \emph{shellable} whenever a topological shelling exists.
\end{definition}
 In other words, a complex is shellable if it may be constructed by sequentially gluing top level faces along a (union of) co-dimension 1 facets. As discussed, it is a standard fact that shellable complexes have vanishing 1-cohomology over every group (see \cite[Proposition B.44]{Lon13}, \cite[Corollary 2.11]{Hat02}, and \cite[Proposition B.42]{Lon13}). We give an elementary proof of this fact by inductively constructing a contraction for any loop along the shelling order, where the inductive step is essentially a simple special case of Van Kampen's Theorem we handle directly.

\begin{theorem}
\label{cor:shellable-trivial}
If $X$ is a shellable complex of dimension at least $2$, its 1-cohomology vanishes with respect to every group $\Gamma$.
\end{theorem}
\begin{proof}
    It is enough to prove any finite loop in the $1$-skeleton of $X$ can be contracted via a series of triangle and backtracking relations. We show this inductively. Given a shelling order $\sigma_1 \ldots \sigma_k$, define
    \[
    X_j \coloneqq \bigcup_{i \leq j} \sigma_j
    \]
    Observe $X_1=\sigma_1$, and $X_k=X$. We will prove by induction that every loop in $X_i$ can be contracted via a series of triangle and backtracking relations. The base case is immediate from the fact that $X_1$ is a simplex (complete complex), for which it is a standard exercise to construct such a contraction.

    Assume the inductive hypothesis holds for $X_{j-1}$, and let $\gamma$ be a loop in $X_j$. Assume further that the loop $\gamma$ does not lie entirely inside $\sigma_j$, else a contraction exists within the full simplex $\sigma_j$ itself. We will prove in this case it is always possible to contract the loop $\gamma$ to a loop entirely contained within $X_{j-1}$, at which point we may apply the inductive hypothesis.
    
    Assume $\gamma$ is not already entirely inside $X_{j-1}$ (else we are done). Let $v \in \gamma$ be an offending vertex in $\sigma_j \setminus X_{j-1}$,
    and let $(v_i,\ldots,v_j)$ be the largest arc of $\gamma$ containing $v$ such that the entire arc lies inside $\sigma_j$. We will rely on the following two useful facts:
    \begin{enumerate}
        \item $v_i$ and $v_j$ are in $X_{j-1} \cap \sigma_j$, and
        \item The $1$-skeleton of $X_{j-1} \cap \sigma_j$ is connected.
    \end{enumerate}
    Let's complete the proof assuming these facts. Since $X_{j-1} \cap \sigma_j$ is connected, we may find a path between $v_i$ and $v_j$ entirely inside $X_{j-1} \cap \sigma_j$. Concatenating this path with the arc $(v_i,v_j)$ gives a loop entirely inside $\sigma_j$, so in particular, the arc can be contracted to this path by triangle and backtracking relations. Repeating this procedure until there are no further loop elements outside $X_{j-1}$ completes the proof.\footnote{Note the process indeed terminates by an elementary potential argument measuring the number of loop vertices in $\sigma_j \setminus X_{j-1}$ since 1) the loop is finite and 2) in every round we strictly decrease the number of elements in $\sigma_j \setminus X_{j-1}$.}

    It is left to prove the two facts.
    The latter is immediate from shellability, which promises $X_{j-1} \cap \sigma_j$ is a pure $(d-1)$-dimensional complex on at most $(d+1)$ vertices.
    Such complexes are automatically connected (indeed have diameter at most 2), since for any vertices $w,w'$ in the complex,
    either $(w,w')$ are already in a shared face, or the top-level faces $\sigma_w$ and $\sigma_{w'}$ containing them respectively share a joint vertex $v$,
    so the path $(w,v,w')$ exists on the 1-skeleton as desired.
    
    Toward the former, observe that since we may assume $(v_i,\ldots,v_j)$ is not the entire loop (else $\gamma$ is fully contained in $\sigma_j$), the arc has \emph{boundary point(s)} $v_{i-1}$ and $v_{i+1}$ which are in $X_{j-1} \setminus \sigma_j$. Further, since $(v_{i-1},v_i)$ and $(v_{j},v_{j+1})$ are edges in $X_j$, observe by definition they must come from some face $\sigma_i$ for $i \leq j$ in the shelling order. Since $v_{i-1}$ and $v_{j+1}$ are not in $\sigma_j$, the edges cannot come from $\sigma_j$, so in particular must come from a prior $\sigma_i$ included in $X_{j-1}$, so $v_i$ and $v_j$ are also in $X_{j-1}$ as desired.
\end{proof}

\paragraph{Posets and Shellability of Projected Flags}

It is left to prove that projections of the flags complex are shellable. To do so, we will rely on the fact that the flags complex arises from taking maximal chains of a graded poset, i.e., it is the so-called ``order complex'' of the Grassmann poset. Such complexes have a simple combinatorial characterization of shellability. Moreover, if the poset is ranked, the notion is closed under projecting to a subset of ranks, exactly our projecting operation.

We need a few background definitions concerning posets.
A poset $P$ is a set $P$ endowed with a relation $\preceq_P$ that is (i) reflexive, (ii) anti-symmetric, and (iii) transitive.
Whenever it is clear we drop the subscript $P$ and use $\preceq$.
Also, if $a\ne b$ and $a\preceq b$, we write $a \prec b$.

\paragraph{Graded Poset.} A graded poset additionally has a rank function $\rho: P \to \mathbb N$ that has the following properties.
\begin{enumerate}[(i)]
\item If $x\prec y$, then $\rho(x)<\rho(y)$.
\item If $x\prec y$ and there is no element $z$ such that $x\prec z \prec y$, then $\rho(y) = \rho(x) + 1$.
\end{enumerate}
We stress that such a rank function may not exist for every poset. We assume all posets we work with this in this section are graded.

\paragraph{Joins.}
Given two posets $P_1$ and $P_2$, the \emph{join} $P_1*P_2$ is the poset on the set $P_1 \sqcup P_2$ with $\leq_{P_1*P_2}$ being such that if $a\leq_{P_1*P_2}b$, then either (i) $a\preceq_{P_1}b$, or (ii) $a\preceq_{P_2} b$, or (iii)$a\in P_1$ and $b\in P_2$.
Extend this definition inductively to define joins of multiple posets.

\begin{definition}[Order Complex]
Let $P$ be a poset.
The \emph{order complex} $X_P$ of the poset $P$ is a simplicial complex consisting of chains of the poset.
In other words, whenever $a_1\preceq_P a_2 \preceq_P \dots \preceq_P a_k$ we have $\set{a_1, a_2, \dots, a_k}\in X_P$.
\end{definition}

\begin{definition}[Grassmann Poset]
Given $d \in \mathbb {N}$ and prime power $q$, the \emph{Grassmann poset} $\grassmann^d_q$ is the poset whose elements are the subspaces of $\mathbb F_q^d$ and $V_1\preceq_{\grassmann_q^d} V_2$ whenever $V_1\subseteq V_2$.
\end{definition}

Thus the flags complex is exactly the order complex of the Grassmann poset.
We now move to defining shellability of posets.

\begin{definition}[Shellability of Posets]
Suppose $(P,\preceq)$ is a poset. Moreover, assume that $P$ is \emph{pure}, i.e., all its maximal chains are of the same length. If its maximal chains can be ordered as $m_1,m_2,\dots,m_t$ so that
for every $1\leq i<j\leq t$, there exists a $1\leq k<j$ and $x\in m_j$ such that
\[
    m_i\cap m_j\subseteq m_k\cap m_j = m_j\setminus \set{x},
\]
then we say $P$ is a \emph{(pure) shellable poset}.
Moreover, such an ordering is called a \emph{shelling order}.
\end{definition}

The following two lemmas, which state respectively that 1) projections of shellable posets are shellable, and 2) the order complex of a shellable poset is topologically shellable, are simple exercises and can be found in \cite{Lon13}.
\begin{lemma}[{\cite[Theorem C.9]{Lon13}}]
Suppose $(P,\preceq, \rho)$ is a pure ranked poset and $r$ is the rank of the maximal elements.
If $P$ is shellable, then for any $S\subseteq [r]$, $P_S=\set{x\in P| \rho(x)\in S}$ is a shellable poset.
\end{lemma}

\begin{lemma}[{\cite[Proposition C.6]{Lon13}}]
\label{lem:pshellability-implies-shellability}
Let $P$ be a shellable poset. Then, the order complex $X_P$ is pure and shellable.
\end{lemma}

Thus it suffices for us to prove the Grassmann poset is shellable. This fact is fairly standard, and can be considered a special case of \cite[Theorem 3.7]{bjorner1980shellable}. We give a self-contained direct proof here adapted from the proof of shellability of the Boolean poset in \cite[Proposition C.8]{Lon13}.
\begin{lemma}
\label{lem:grassmann-pshellability}
For $d$ any positive integer and $q$ a prime power, the Grassmann poset $\mathrm {GP}^d_q$ is shellable.
\end{lemma}
\begin{proof}
The idea is to label each maximal chain by an ordered basis $(v_1, \dots, v_d)$ of $\mathbb{F}_q^d$ with $W_i = \operatorname{span}\{v_1, \dots, v_i\}$ for all $i$, taking each $v_i$ to be the smallest vector in $W_i \setminus W_{i-1}$ under a fixed total order on $\mathbb{F}_q^d$. We show lexicographically ordering chains by these labels gives a shelling. We formalize the argument below.

\paragraph{Labelling maximal chains.}
Fix an arbitrary total order on $\mathbb{F}_q^d$. To each maximal chain
\[
  m\colon\quad 0 = W_0 \subseteq W_1 \subseteq \cdots \subseteq W_d = \mathbb{F}_q^d,
\]
associate the label-vector $\lambda(m) = (v_1, \dots, v_d)$, where $v_i$ is the
smallest vector in $W_i \setminus W_{i-1}$ under our ordering. Note by definition we have
$W_i = \operatorname{span}\{v_1, \dots, v_i\}$.

We also extend the labeling to maximal chains lying \emph{between} two subspaces $M \subseteq N$.\footnote{A chain from $A$
to $B$ is \emph{maximal between them} if no element $A \preceq C \preceq B$ can
be added.} In particular, given
$M \subseteq N$ and a maximal chain
$c\colon M = W_0 \subseteq \cdots \subseteq W_k = N$, let $\lambda(c) = (v_1, \dots, v_k)$ where $v_i$ is the smallest
vector in $W_i \setminus W_{i-1}$, so that
$W_i = W_0 \oplus \operatorname{span}\{v_1, \dots, v_i\}$.

\paragraph{The proposed shelling order.}
Order all maximal chains of $\mathrm{GP}^d_q$ as $m_1, m_2, \dots$ in
lexicographic order of their labels:
\[
  \lambda(m_1) < \lambda(m_2) < \lambda(m_3) < \cdots.
\]
We claim this is a shelling.  Fix $i < j$; we must produce some $k < j$
together with a vertex $V_t \in m_j$ such that
\[
  m_j \cap m_k \;=\; m_j \setminus \{V_t\}
  \quad\text{and}\quad
  m_j \cap m_i \;\subseteq\; m_j \cap m_k.
\]

\paragraph{Setup.}
Write the two chains as
\begin{align*}
  m_i &\colon\quad 0 = U_0 \subseteq U_1 \subseteq \cdots \subseteq U_d = \mathbb{F}_q^d,\\
  m_j &\colon\quad 0 = V_0 \subseteq V_1 \subseteq \cdots \subseteq V_d = \mathbb{F}_q^d,
\end{align*}
and let $r\geq 0$ be the largest index such that $V_i=U_i$ for all $i\leq r$, and $s>r$ be the smallest index such that we again have $V_s=U_s$.
Abusing notation, we
write $\lambda(V_a)$ for the $a$-th entry of $\lambda(m_j)$.

\paragraph{Reduction to finding a descent.}
We claim it suffices to produce an index $r < t < s$ at which the labels of
$m_j$ \emph{descend}, i.e.,\ $\lambda(V_t) > \lambda(V_{t+1})$.  Suppose we have
such a $t$, and write $\beta = \lambda(V_{t+1})$.  Set
\[
  V \;=\; V_{t-1} + \operatorname{span}\{\beta\}.
\]
Since $\beta \in V_{t+1} \setminus V_{t-1}$, $V$ is a subspace strictly between
$V_{t-1}$ and $V_{t+1}$ and distinct from $V_t$, so replacing $V_t$ with $V$
in $m_j$ produces a maximal chain
\[
  m_k\colon\quad V_0 \subseteq \cdots \subseteq V_{t-1} \subseteq V \subseteq V_{t+1} \subseteq \cdots \subseteq V_d.
\]
We now argue that
\begin{enumerate}
    \item $m_k$ comes before $m_j$ in the order, and
    \item $m_j \cap m_i \subseteq m_j \setminus \{V_t\} = m_j \cap m_k$.
\end{enumerate}
We start with the former. Because $\beta \in V \setminus V_{t-1}$, the $t$-th entry of $\lambda(m_k)$ is at most $\beta$. On the other hand, $\beta = \lambda(V_{t+1}) < \lambda(V_t)$ since $t$ is our descent. Since $\lambda(m_k)$ and $\lambda(m_j)$ agree
in positions $1,\dots,t-1$, we therefore have $\lambda(m_k) < \lambda(m_j)$ lexicographically as desired.

Towards the latter, first observe that by construction $m_j \cap m_k = m_j \setminus \{V_t\}$.
As $r < t < s$, the chains $m_i$ and $m_j$ disagree on $V_t$.
Thus, we have $V_t \notin m_i$ and
$m_j \cap m_i \subseteq m_j \setminus \{V_t\} = m_j \cap m_k$ as desired.

\paragraph{Existence of a descent.}
It remains to produce such a $t$. This is a consequence of the following fact.

\begin{claim}\label{cl:unique-min}
Let $M\subseteq N$ be two subspaces in this poset.
Then, there is a unique maximal chain $c$ starting at $M$ and ending at $N$ with increasing label-vector $\lambda(c)$. Moreover, $\lambda(c)$ is lexicogaphically minimal among labelings of maximal chains between $M$ and $N$.
\end{claim} 

Namely, consider the sub-chains of $m_i$ and $m_j$ between the matching endpoints $V_r = U_r$ and
$V_s = U_s$.  Because $\lambda(m_i) < \lambda(m_j)$ and the two chains first
disagree at step $r+1$, the labels of the $m_i$-sub-chain are
lexicographically smaller than those of the $m_j$-sub-chain. In particular, the $m_j$-sub-chain is \emph{not} the
lexicographical minimum, and so its label-vector is not strictly increasing and there must exist
a descent. It is left to prove the claim.

\begin{proof}[Proof of \Cref{cl:unique-min}]
We start with uniqueness. First note that a maximal chain $M=W_0\subseteq W_1 \subseteq \dots \subseteq W_k=N$ starting at $M$ and ending at $N$ with increasing label-vector can be constructed by taking $W_0=M$ and then at each step $i$ adding the smallest vector $w_i\in N\setminus W_{i-1}$ to $W_{i-1}$ to get $W_i=W_{i-1}\oplus \lspan{w_i}$.
This implies $(w_1<w_2<\dots<w_k)$ is the label-vector for the constructed chain.

Towards a contradiction, assume that there are two (distinct) maximal chains $c,c'$ starting at $M$ and ending at $N$ both with increasing label vector.
Suppose the chains $c,c'$ are given by
\begin{align*}
M=W_0\subseteq W_1 \subseteq \dots \subseteq W_k=N,\\
M=W'_0\subseteq W'_1 \subseteq \dots \subseteq W'_k=N,
\end{align*}
respectively, and the label-vectors are given by $(w_1,w_2,\dots, w_k)$ and $(w'_1,w'_2,\dots, w'_k)$, respectively.
Suppose $i$ is the first index such that $w_i\ne w'_i$ and without loss of generality assume that $w_i<w'_i$.
We must have $w_i\not \in W'_i$; otherwise, $W_i=W'_i$ which implies $w_i$ is the correct vector to associate with $W'_i\setminus W'_{i-1}$.
This also means that the $w_i$ must be contained in some $W'_j$ with $k\geq j>i$.
So, the corresponding label $w'_j\leq w_i$.
On the other hand, since $(w'_1,w'_2,\dots, w'_k)$ is increasing we must have $w'_j>w'_i > w_i$, which gives us a contradiction.

To see $\lambda(c)$ is minimal, consider another chain $\tilde c$ given by
\[
M=S_0\subseteq S_1 \subseteq \dots \subseteq S_k=N
\]
with label-vector $(s_1,\dots, s_k)$.
If $s_1>w_1$, then we are done.
Otherwise, let $i$ be the smallest index such that $s_i \ne w_i$ but $s_{i-1}=w_{i-1}$.
Then, as $w_i$ is the smallest vector in $N\setminus W_{i-1}=S_{i-1}$, we must have $s_i>w_i$.
\end{proof}
\end{proof}

We already have everything we need to conclude that the cohomology of the projected flags complex is trivial.
Furthermore, we can use the following observations to conclude that every link of the the projected flags complex also has trivial cohomology.

The following observation trivially follows from the definitions of join of simplicial complexes and join of posets.
\begin{observation}
Let $P_1, P_2$ be two posets. Then, we have $X_{P_1*P_2}=X_{P_1}*X_{P_2}$.
\end{observation}

\begin{observation}
Let $P_1$ and $P_2$ be shellable posets. Then, the join $P_1*P_2$ is also shellable.
\end{observation}
\begin{proof}
Let $m_1, m_2, \dots, m_t$ be the maximal chains of $P_1$ in a shelling order.
Let $m'_1, m'_2, \dots, m'_{t'}$ be the maximal chains of $P_2$ in a shelling order.
Then, it is easy to verify that the following is the sequence of maximal chains of $P_1*P_2$ in a shelling order:
\[
m_1\circ m'_1, m_2\circ m'_1, \dots, m_t\circ m'_1, m_1\circ m'_2, m_2\circ m'_2, \dots, m_t\circ m'_2,\dots \dots, m_1\circ m'_{t'}, m_2\circ m'_{t'}, \dots, m_t\circ m'_{t'},
\]
where $\circ$ is used to denote concatenation.
\end{proof}

Since links of the flags complex are joins of flags complexes in lower dimension, \Cref{lem:trivial-cohomology} immediately follows.

\subsubsection{Cocycle Expansion of the Projected Flags Complex}
With vanishing co-homology out of the way, it is enough to prove the complexes are also good $1$-cocycle expanders. Toward this end, we'll rely on some useful techniques from \cite{DD24coboundary} for coboundary and cocycle expansion of a complex with respect to a general (possibly non-abelian) group. The first allows us to derive cocycle expansion from spectral and coboundary expansion of the links of $X$.

\begin{lemma}[\cite{DD24coboundary}]
\label{lem:localcoboundary-to-coboundary}
Let $\beta,\lambda>0$ and let $k>0$ be an integer. Let $X$ be a $d$-dimensional simplicial complex for
$d\geq k+ 2$ and assume that $X$ is a $\lambda$-one-sided local spectral expander. Let $\Gamma$ be any group. Assume that for every non-empty $r\in X$, $X_r$ is a coboundary expander and that $h^{k+1-|r|}(X_r,\Gamma) \geq \beta$. Then
\[
    h^{k}(X,\Gamma) \geq \frac{\beta^{k+1}}{(k+1)!\cdot 4}-e\lambda.
\]
Here, $e\approx 2.71$ is the Euler's number.
\end{lemma}

The second tool is the `color restriction lemma' which states that if most projections of a complex are coboundary expanders, then the complex itself is a coboundary expander.
\begin{lemma}[\cite{DD24coboundary}]
\label{lem:color-restriction}
Let $\ell, d$ be integers so that $3 \leq \ell \leq d$ and let $\beta, p \in (0, 1]$.
Let $\Gamma$ be some group. Let $X$ be a $(d+1)$-partite simplicial complex so that
\[
\pr{I\in \binom{[d+1]}{\ell}}{X^I\text{ has }1\text{-coboundary expansion }\beta
\text{ and }
\forall s\in X(0), X_s^I \text{ is a } \frac{1}{2}\text{-local spectral expander}} \geq p.
\]
Then X is a coboundary expander with $h^1(X)\geq \frac{p\beta}{36}$. 
\end{lemma}

\begin{figure}
    \centering
    \begin{subfigure}[b]{0.45\textwidth}
        \centering
        \tikzset{every picture/.style={line width=0.75pt}} 

\begin{tikzpicture}[x=0.75pt,y=0.75pt,yscale=-0.45,xscale=0.45]

\draw  [fill={rgb, 255:red, 245; green, 166; blue, 35 }  ,fill opacity=0.5 ][line width=0.75]  (124,122.25) .. controls (124,119.35) and (126.46,117) .. (129.5,117) .. controls (132.54,117) and (135,119.35) .. (135,122.25) .. controls (135,125.15) and (132.54,127.5) .. (129.5,127.5) .. controls (126.46,127.5) and (124,125.15) .. (124,122.25) -- cycle ;
\draw  [fill={rgb, 255:red, 245; green, 166; blue, 35 }  ,fill opacity=0.5 ][line width=0.75]  (364,145.25) .. controls (364,142.35) and (366.46,140) .. (369.5,140) .. controls (372.54,140) and (375,142.35) .. (375,145.25) .. controls (375,148.15) and (372.54,150.5) .. (369.5,150.5) .. controls (366.46,150.5) and (364,148.15) .. (364,145.25) -- cycle ;
\draw  [fill={rgb, 255:red, 245; green, 166; blue, 35 }  ,fill opacity=0.5 ][line width=0.75]  (365,218.25) .. controls (365,215.35) and (367.46,213) .. (370.5,213) .. controls (373.54,213) and (376,215.35) .. (376,218.25) .. controls (376,221.15) and (373.54,223.5) .. (370.5,223.5) .. controls (367.46,223.5) and (365,221.15) .. (365,218.25) -- cycle ;
\draw  [fill={rgb, 255:red, 245; green, 166; blue, 35 }  ,fill opacity=0.5 ][line width=0.75]  (125,210.25) .. controls (125,207.35) and (127.46,205) .. (130.5,205) .. controls (133.54,205) and (136,207.35) .. (136,210.25) .. controls (136,213.15) and (133.54,215.5) .. (130.5,215.5) .. controls (127.46,215.5) and (125,213.15) .. (125,210.25) -- cycle ;
\draw  [fill={rgb, 255:red, 245; green, 166; blue, 35 }  ,fill opacity=0.5 ][line width=0.75]  (127,296.25) .. controls (127,293.35) and (129.46,291) .. (132.5,291) .. controls (135.54,291) and (138,293.35) .. (138,296.25) .. controls (138,299.15) and (135.54,301.5) .. (132.5,301.5) .. controls (129.46,301.5) and (127,299.15) .. (127,296.25) -- cycle ;
\draw  [fill={rgb, 255:red, 245; green, 166; blue, 35 }  ,fill opacity=0.5 ][line width=0.75]  (366,296.25) .. controls (366,293.35) and (368.46,291) .. (371.5,291) .. controls (374.54,291) and (377,293.35) .. (377,296.25) .. controls (377,299.15) and (374.54,301.5) .. (371.5,301.5) .. controls (368.46,301.5) and (366,299.15) .. (366,296.25) -- cycle ;
\draw  [fill={rgb, 255:red, 155; green, 155; blue, 155 }  ,fill opacity=0.5 ][line width=0.75]  (599,281.25) .. controls (599,278.35) and (601.46,276) .. (604.5,276) .. controls (607.54,276) and (610,278.35) .. (610,281.25) .. controls (610,284.15) and (607.54,286.5) .. (604.5,286.5) .. controls (601.46,286.5) and (599,284.15) .. (599,281.25) -- cycle ;
\draw    (129.5,122.25) -- (370.5,218.25) ;
\draw [shift={(254.65,172.1)}, rotate = 201.72] [fill={rgb, 255:red, 0; green, 0; blue, 0 }  ][line width=0.08]  [draw opacity=0] (10.72,-5.15) -- (0,0) -- (10.72,5.15) -- (7.12,0) -- cycle    ;
\draw    (370.5,218.25) -- (132.5,296.25) ;
\draw [shift={(246.75,258.81)}, rotate = 341.85] [fill={rgb, 255:red, 0; green, 0; blue, 0 }  ][line width=0.08]  [draw opacity=0] (10.72,-5.15) -- (0,0) -- (10.72,5.15) -- (7.12,0) -- cycle    ;
\draw    (132.5,296.25) -- (371.5,296.25) ;
\draw [shift={(257,296.25)}, rotate = 180] [fill={rgb, 255:red, 0; green, 0; blue, 0 }  ][line width=0.08]  [draw opacity=0] (10.72,-5.15) -- (0,0) -- (10.72,5.15) -- (7.12,0) -- cycle    ;
\draw    (371.5,296.25) -- (130.5,210.25) ;
\draw [shift={(246.29,251.57)}, rotate = 19.64] [fill={rgb, 255:red, 0; green, 0; blue, 0 }  ][line width=0.08]  [draw opacity=0] (10.72,-5.15) -- (0,0) -- (10.72,5.15) -- (7.12,0) -- cycle    ;
\draw    (130.5,210.25) -- (369.5,145.25) ;
\draw [shift={(254.82,176.44)}, rotate = 164.79] [fill={rgb, 255:red, 0; green, 0; blue, 0 }  ][line width=0.08]  [draw opacity=0] (10.72,-5.15) -- (0,0) -- (10.72,5.15) -- (7.12,0) -- cycle    ;
\draw    (369.5,145.25) -- (129.5,122.25) ;
\draw [shift={(244.52,133.27)}, rotate = 5.47] [fill={rgb, 255:red, 0; green, 0; blue, 0 }  ][line width=0.08]  [draw opacity=0] (10.72,-5.15) -- (0,0) -- (10.72,5.15) -- (7.12,0) -- cycle    ;
\draw  [dash pattern={on 4.5pt off 4.5pt}]  (604.5,281.25) -- (370.5,218.25) ;
\draw  [dash pattern={on 4.5pt off 4.5pt}]  (371.5,296.25) -- (604.5,281.25) ;

\draw (124,75.4) node [anchor=north west][inner sep=0.75pt]    {$i_{0}$};
\draw (364,75.4) node [anchor=north west][inner sep=0.75pt]    {$i_{1}$};
\draw (597,74.4) node [anchor=north west][inner sep=0.75pt]    {$i_{2}$};
\draw (131.5,130.9) node [anchor=north west][inner sep=0.75pt]    {$v$};
\draw (372.5,226.9) node [anchor=north west][inner sep=0.75pt]    {$u'_{0}$};
\draw (134.5,304.9) node [anchor=north west][inner sep=0.75pt]    {$u'$};
\draw (373.5,304.9) node [anchor=north west][inner sep=0.75pt]    {$u$};
\draw (132.5,218.9) node [anchor=north west][inner sep=0.75pt]    {$u_{1}$};
\draw (371.5,153.9) node [anchor=north west][inner sep=0.75pt]    {$u_{0}$};
\draw (606.5,289.9) node [anchor=north west][inner sep=0.75pt]    {$x$};

\end{tikzpicture}
        \caption{$P_0=P_{u}\circ(u',u)\circ P_{u'}^{-1}=vu'_0u'uu_1u_0v$}
    \end{subfigure}
    \begin{subfigure}[b]{0.45\textwidth}
        \centering
        \tikzset{every picture/.style={line width=0.75pt}} 

\begin{tikzpicture}[x=0.75pt,y=0.75pt,yscale=-0.45,xscale=0.45]

\draw  [fill={rgb, 255:red, 245; green, 166; blue, 35 }  ,fill opacity=0.5 ][line width=0.75]  (124,122.25) .. controls (124,119.35) and (126.46,117) .. (129.5,117) .. controls (132.54,117) and (135,119.35) .. (135,122.25) .. controls (135,125.15) and (132.54,127.5) .. (129.5,127.5) .. controls (126.46,127.5) and (124,125.15) .. (124,122.25) -- cycle ;
\draw  [fill={rgb, 255:red, 245; green, 166; blue, 35 }  ,fill opacity=0.5 ][line width=0.75]  (364,145.25) .. controls (364,142.35) and (366.46,140) .. (369.5,140) .. controls (372.54,140) and (375,142.35) .. (375,145.25) .. controls (375,148.15) and (372.54,150.5) .. (369.5,150.5) .. controls (366.46,150.5) and (364,148.15) .. (364,145.25) -- cycle ;
\draw  [fill={rgb, 255:red, 245; green, 166; blue, 35 }  ,fill opacity=0.5 ][line width=0.75]  (365,218.25) .. controls (365,215.35) and (367.46,213) .. (370.5,213) .. controls (373.54,213) and (376,215.35) .. (376,218.25) .. controls (376,221.15) and (373.54,223.5) .. (370.5,223.5) .. controls (367.46,223.5) and (365,221.15) .. (365,218.25) -- cycle ;
\draw  [fill={rgb, 255:red, 245; green, 166; blue, 35 }  ,fill opacity=0.5 ][line width=0.75]  (125,210.25) .. controls (125,207.35) and (127.46,205) .. (130.5,205) .. controls (133.54,205) and (136,207.35) .. (136,210.25) .. controls (136,213.15) and (133.54,215.5) .. (130.5,215.5) .. controls (127.46,215.5) and (125,213.15) .. (125,210.25) -- cycle ;
\draw  [fill={rgb, 255:red, 245; green, 166; blue, 35 }  ,fill opacity=0.5 ][line width=0.75]  (127,296.25) .. controls (127,293.35) and (129.46,291) .. (132.5,291) .. controls (135.54,291) and (138,293.35) .. (138,296.25) .. controls (138,299.15) and (135.54,301.5) .. (132.5,301.5) .. controls (129.46,301.5) and (127,299.15) .. (127,296.25) -- cycle ;
\draw  [fill={rgb, 255:red, 245; green, 166; blue, 35 }  ,fill opacity=0.5 ][line width=0.75]  (366,296.25) .. controls (366,293.35) and (368.46,291) .. (371.5,291) .. controls (374.54,291) and (377,293.35) .. (377,296.25) .. controls (377,299.15) and (374.54,301.5) .. (371.5,301.5) .. controls (368.46,301.5) and (366,299.15) .. (366,296.25) -- cycle ;
\draw  [fill={rgb, 255:red, 155; green, 155; blue, 155 }  ,fill opacity=0.5 ][line width=0.75]  (599,281.25) .. controls (599,278.35) and (601.46,276) .. (604.5,276) .. controls (607.54,276) and (610,278.35) .. (610,281.25) .. controls (610,284.15) and (607.54,286.5) .. (604.5,286.5) .. controls (601.46,286.5) and (599,284.15) .. (599,281.25) -- cycle ;
\draw    (129.5,122.25) -- (370.5,218.25) ;
\draw [shift={(254.65,172.1)}, rotate = 201.72] [fill={rgb, 255:red, 0; green, 0; blue, 0 }  ][line width=0.08]  [draw opacity=0] (10.72,-5.15) -- (0,0) -- (10.72,5.15) -- (7.12,0) -- cycle    ;
\draw    (370.5,218.25) -- (132.5,296.25) ;
\draw [shift={(246.75,258.81)}, rotate = 341.85] [fill={rgb, 255:red, 0; green, 0; blue, 0 }  ][line width=0.08]  [draw opacity=0] (10.72,-5.15) -- (0,0) -- (10.72,5.15) -- (7.12,0) -- cycle    ;
\draw    (371.5,296.25) -- (130.5,210.25) ;
\draw [shift={(246.29,251.57)}, rotate = 19.64] [fill={rgb, 255:red, 0; green, 0; blue, 0 }  ][line width=0.08]  [draw opacity=0] (10.72,-5.15) -- (0,0) -- (10.72,5.15) -- (7.12,0) -- cycle    ;
\draw    (130.5,210.25) -- (369.5,145.25) ;
\draw [shift={(254.82,176.44)}, rotate = 164.79] [fill={rgb, 255:red, 0; green, 0; blue, 0 }  ][line width=0.08]  [draw opacity=0] (10.72,-5.15) -- (0,0) -- (10.72,5.15) -- (7.12,0) -- cycle    ;
\draw    (369.5,145.25) -- (129.5,122.25) ;
\draw [shift={(244.52,133.27)}, rotate = 5.47] [fill={rgb, 255:red, 0; green, 0; blue, 0 }  ][line width=0.08]  [draw opacity=0] (10.72,-5.15) -- (0,0) -- (10.72,5.15) -- (7.12,0) -- cycle    ;
\draw  [dash pattern={on 4.5pt off 4.5pt}]  (604.5,281.25) -- (370.5,218.25) ;
\draw    (604.5,281.25) -- (371.5,296.25) ;
\draw [shift={(483.01,289.07)}, rotate = 356.32] [fill={rgb, 255:red, 0; green, 0; blue, 0 }  ][line width=0.08]  [draw opacity=0] (10.72,-5.15) -- (0,0) -- (10.72,5.15) -- (7.12,0) -- cycle    ;
\draw    (132.5,296.25) .. controls (310,369.5) and (564.5,311.25) .. (604.5,281.25) ;
\draw [shift={(375.3,330.79)}, rotate = 177.49] [fill={rgb, 255:red, 0; green, 0; blue, 0 }  ][line width=0.08]  [draw opacity=0] (10.72,-5.15) -- (0,0) -- (10.72,5.15) -- (7.12,0) -- cycle    ;

\draw (124,75.4) node [anchor=north west][inner sep=0.75pt]    {$i_{0}$};
\draw (364,75.4) node [anchor=north west][inner sep=0.75pt]    {$i_{1}$};
\draw (597,74.4) node [anchor=north west][inner sep=0.75pt]    {$i_{2}$};
\draw (131.5,130.9) node [anchor=north west][inner sep=0.75pt]    {$v$};
\draw (372.5,226.9) node [anchor=north west][inner sep=0.75pt]    {$u'_{0}$};
\draw (134.5,304.9) node [anchor=north west][inner sep=0.75pt]    {$u'$};
\draw (373.5,304.9) node [anchor=north west][inner sep=0.75pt]    {$u$};
\draw (132.5,218.9) node [anchor=north west][inner sep=0.75pt]    {$u_{1}$};
\draw (371.5,153.9) node [anchor=north west][inner sep=0.75pt]    {$u_{0}$};
\draw (606.5,289.9) node [anchor=north west][inner sep=0.75pt]    {$x$};

\end{tikzpicture}
        \caption{$P_1=vu'_0u'xuu_1u_0v$}
    \end{subfigure}
    \begin{subfigure}[b]{0.45\textwidth}
        \centering
        \tikzset{every picture/.style={line width=0.75pt}} 

\begin{tikzpicture}[x=0.75pt,y=0.75pt,yscale=-0.45,xscale=0.45]

\draw  [fill={rgb, 255:red, 245; green, 166; blue, 35 }  ,fill opacity=0.5 ][line width=0.75]  (124,122.25) .. controls (124,119.35) and (126.46,117) .. (129.5,117) .. controls (132.54,117) and (135,119.35) .. (135,122.25) .. controls (135,125.15) and (132.54,127.5) .. (129.5,127.5) .. controls (126.46,127.5) and (124,125.15) .. (124,122.25) -- cycle ;
\draw  [fill={rgb, 255:red, 245; green, 166; blue, 35 }  ,fill opacity=0.5 ][line width=0.75]  (364,145.25) .. controls (364,142.35) and (366.46,140) .. (369.5,140) .. controls (372.54,140) and (375,142.35) .. (375,145.25) .. controls (375,148.15) and (372.54,150.5) .. (369.5,150.5) .. controls (366.46,150.5) and (364,148.15) .. (364,145.25) -- cycle ;
\draw  [fill={rgb, 255:red, 245; green, 166; blue, 35 }  ,fill opacity=0.5 ][line width=0.75]  (365,218.25) .. controls (365,215.35) and (367.46,213) .. (370.5,213) .. controls (373.54,213) and (376,215.35) .. (376,218.25) .. controls (376,221.15) and (373.54,223.5) .. (370.5,223.5) .. controls (367.46,223.5) and (365,221.15) .. (365,218.25) -- cycle ;
\draw  [fill={rgb, 255:red, 245; green, 166; blue, 35 }  ,fill opacity=0.5 ][line width=0.75]  (125,210.25) .. controls (125,207.35) and (127.46,205) .. (130.5,205) .. controls (133.54,205) and (136,207.35) .. (136,210.25) .. controls (136,213.15) and (133.54,215.5) .. (130.5,215.5) .. controls (127.46,215.5) and (125,213.15) .. (125,210.25) -- cycle ;
\draw  [fill={rgb, 255:red, 245; green, 166; blue, 35 }  ,fill opacity=0.5 ][line width=0.75]  (127,296.25) .. controls (127,293.35) and (129.46,291) .. (132.5,291) .. controls (135.54,291) and (138,293.35) .. (138,296.25) .. controls (138,299.15) and (135.54,301.5) .. (132.5,301.5) .. controls (129.46,301.5) and (127,299.15) .. (127,296.25) -- cycle ;
\draw  [fill={rgb, 255:red, 245; green, 166; blue, 35 }  ,fill opacity=0.5 ][line width=0.75]  (366,296.25) .. controls (366,293.35) and (368.46,291) .. (371.5,291) .. controls (374.54,291) and (377,293.35) .. (377,296.25) .. controls (377,299.15) and (374.54,301.5) .. (371.5,301.5) .. controls (368.46,301.5) and (366,299.15) .. (366,296.25) -- cycle ;
\draw  [fill={rgb, 255:red, 155; green, 155; blue, 155 }  ,fill opacity=0.5 ][line width=0.75]  (599,281.25) .. controls (599,278.35) and (601.46,276) .. (604.5,276) .. controls (607.54,276) and (610,278.35) .. (610,281.25) .. controls (610,284.15) and (607.54,286.5) .. (604.5,286.5) .. controls (601.46,286.5) and (599,284.15) .. (599,281.25) -- cycle ;
\draw    (129.5,122.25) -- (370.5,218.25) ;
\draw [shift={(254.65,172.1)}, rotate = 201.72] [fill={rgb, 255:red, 0; green, 0; blue, 0 }  ][line width=0.08]  [draw opacity=0] (10.72,-5.15) -- (0,0) -- (10.72,5.15) -- (7.12,0) -- cycle    ;
\draw    (371.5,296.25) -- (130.5,210.25) ;
\draw [shift={(246.29,251.57)}, rotate = 19.64] [fill={rgb, 255:red, 0; green, 0; blue, 0 }  ][line width=0.08]  [draw opacity=0] (10.72,-5.15) -- (0,0) -- (10.72,5.15) -- (7.12,0) -- cycle    ;
\draw    (130.5,210.25) -- (369.5,145.25) ;
\draw [shift={(254.82,176.44)}, rotate = 164.79] [fill={rgb, 255:red, 0; green, 0; blue, 0 }  ][line width=0.08]  [draw opacity=0] (10.72,-5.15) -- (0,0) -- (10.72,5.15) -- (7.12,0) -- cycle    ;
\draw    (369.5,145.25) -- (129.5,122.25) ;
\draw [shift={(244.52,133.27)}, rotate = 5.47] [fill={rgb, 255:red, 0; green, 0; blue, 0 }  ][line width=0.08]  [draw opacity=0] (10.72,-5.15) -- (0,0) -- (10.72,5.15) -- (7.12,0) -- cycle    ;
\draw    (604.5,281.25) -- (370.5,218.25) ;
\draw [shift={(493.78,251.44)}, rotate = 195.07] [fill={rgb, 255:red, 0; green, 0; blue, 0 }  ][line width=0.08]  [draw opacity=0] (10.72,-5.15) -- (0,0) -- (10.72,5.15) -- (7.12,0) -- cycle    ;
\draw    (604.5,281.25) -- (371.5,296.25) ;
\draw [shift={(483.01,289.07)}, rotate = 356.32] [fill={rgb, 255:red, 0; green, 0; blue, 0 }  ][line width=0.08]  [draw opacity=0] (10.72,-5.15) -- (0,0) -- (10.72,5.15) -- (7.12,0) -- cycle    ;

\draw (124,75.4) node [anchor=north west][inner sep=0.75pt]    {$i_{0}$};
\draw (364,75.4) node [anchor=north west][inner sep=0.75pt]    {$i_{1}$};
\draw (597,74.4) node [anchor=north west][inner sep=0.75pt]    {$i_{2}$};
\draw (131.5,130.9) node [anchor=north west][inner sep=0.75pt]    {$v$};
\draw (372.5,226.9) node [anchor=north west][inner sep=0.75pt]    {$u'_{0}$};
\draw (134.5,304.9) node [anchor=north west][inner sep=0.75pt]    {$u'$};
\draw (373.5,304.9) node [anchor=north west][inner sep=0.75pt]    {$u$};
\draw (132.5,218.9) node [anchor=north west][inner sep=0.75pt]    {$u_{1}$};
\draw (371.5,153.9) node [anchor=north west][inner sep=0.75pt]    {$u_{0}$};
\draw (606.5,289.9) node [anchor=north west][inner sep=0.75pt]    {$x$};

\end{tikzpicture}
        \caption{$P_2=vu'_0xuu_1u_0v$}
    \end{subfigure}
    \begin{subfigure}[b]{0.45\textwidth}
        \centering
        \tikzset{every picture/.style={line width=0.75pt}} 

\begin{tikzpicture}[x=0.75pt,y=0.75pt,yscale=-0.45,xscale=0.45]

\draw  [fill={rgb, 255:red, 245; green, 166; blue, 35 }  ,fill opacity=0.5 ][line width=0.75]  (124,122.25) .. controls (124,119.35) and (126.46,117) .. (129.5,117) .. controls (132.54,117) and (135,119.35) .. (135,122.25) .. controls (135,125.15) and (132.54,127.5) .. (129.5,127.5) .. controls (126.46,127.5) and (124,125.15) .. (124,122.25) -- cycle ;
\draw  [fill={rgb, 255:red, 245; green, 166; blue, 35 }  ,fill opacity=0.5 ][line width=0.75]  (364,145.25) .. controls (364,142.35) and (366.46,140) .. (369.5,140) .. controls (372.54,140) and (375,142.35) .. (375,145.25) .. controls (375,148.15) and (372.54,150.5) .. (369.5,150.5) .. controls (366.46,150.5) and (364,148.15) .. (364,145.25) -- cycle ;
\draw  [fill={rgb, 255:red, 245; green, 166; blue, 35 }  ,fill opacity=0.5 ][line width=0.75]  (365,218.25) .. controls (365,215.35) and (367.46,213) .. (370.5,213) .. controls (373.54,213) and (376,215.35) .. (376,218.25) .. controls (376,221.15) and (373.54,223.5) .. (370.5,223.5) .. controls (367.46,223.5) and (365,221.15) .. (365,218.25) -- cycle ;
\draw  [fill={rgb, 255:red, 245; green, 166; blue, 35 }  ,fill opacity=0.5 ][line width=0.75]  (125,210.25) .. controls (125,207.35) and (127.46,205) .. (130.5,205) .. controls (133.54,205) and (136,207.35) .. (136,210.25) .. controls (136,213.15) and (133.54,215.5) .. (130.5,215.5) .. controls (127.46,215.5) and (125,213.15) .. (125,210.25) -- cycle ;
\draw  [fill={rgb, 255:red, 245; green, 166; blue, 35 }  ,fill opacity=0.5 ][line width=0.75]  (366,296.25) .. controls (366,293.35) and (368.46,291) .. (371.5,291) .. controls (374.54,291) and (377,293.35) .. (377,296.25) .. controls (377,299.15) and (374.54,301.5) .. (371.5,301.5) .. controls (368.46,301.5) and (366,299.15) .. (366,296.25) -- cycle ;
\draw  [fill={rgb, 255:red, 155; green, 155; blue, 155 }  ,fill opacity=0.5 ][line width=0.75]  (599,281.25) .. controls (599,278.35) and (601.46,276) .. (604.5,276) .. controls (607.54,276) and (610,278.35) .. (610,281.25) .. controls (610,284.15) and (607.54,286.5) .. (604.5,286.5) .. controls (601.46,286.5) and (599,284.15) .. (599,281.25) -- cycle ;
\draw    (371.5,296.25) -- (130.5,210.25) ;
\draw [shift={(246.29,251.57)}, rotate = 19.64] [fill={rgb, 255:red, 0; green, 0; blue, 0 }  ][line width=0.08]  [draw opacity=0] (10.72,-5.15) -- (0,0) -- (10.72,5.15) -- (7.12,0) -- cycle    ;
\draw    (130.5,210.25) -- (369.5,145.25) ;
\draw [shift={(254.82,176.44)}, rotate = 164.79] [fill={rgb, 255:red, 0; green, 0; blue, 0 }  ][line width=0.08]  [draw opacity=0] (10.72,-5.15) -- (0,0) -- (10.72,5.15) -- (7.12,0) -- cycle    ;
\draw    (369.5,145.25) -- (129.5,122.25) ;
\draw [shift={(244.52,133.27)}, rotate = 5.47] [fill={rgb, 255:red, 0; green, 0; blue, 0 }  ][line width=0.08]  [draw opacity=0] (10.72,-5.15) -- (0,0) -- (10.72,5.15) -- (7.12,0) -- cycle    ;
\draw    (604.5,281.25) -- (129.5,122.25) ;
\draw [shift={(373.16,203.81)}, rotate = 198.51] [fill={rgb, 255:red, 0; green, 0; blue, 0 }  ][line width=0.08]  [draw opacity=0] (10.72,-5.15) -- (0,0) -- (10.72,5.15) -- (7.12,0) -- cycle    ;
\draw    (604.5,281.25) -- (371.5,296.25) ;
\draw [shift={(483.01,289.07)}, rotate = 356.32] [fill={rgb, 255:red, 0; green, 0; blue, 0 }  ][line width=0.08]  [draw opacity=0] (10.72,-5.15) -- (0,0) -- (10.72,5.15) -- (7.12,0) -- cycle    ;

\draw (124,75.4) node [anchor=north west][inner sep=0.75pt]    {$i_{0}$};
\draw (364,75.4) node [anchor=north west][inner sep=0.75pt]    {$i_{1}$};
\draw (597,74.4) node [anchor=north west][inner sep=0.75pt]    {$i_{2}$};
\draw (131.5,130.9) node [anchor=north west][inner sep=0.75pt]    {$v$};
\draw (372.5,226.9) node [anchor=north west][inner sep=0.75pt]    {$u'_{0}$};
\draw (373.5,304.9) node [anchor=north west][inner sep=0.75pt]    {$u$};
\draw (132.5,218.9) node [anchor=north west][inner sep=0.75pt]    {$u_{1}$};
\draw (371.5,153.9) node [anchor=north west][inner sep=0.75pt]    {$u_{0}$};
\draw (606.5,289.9) node [anchor=north west][inner sep=0.75pt]    {$x$};

\end{tikzpicture}
        \caption{$P_3=vxuu_1u_0v$}
    \end{subfigure}
    \begin{subfigure}[b]{0.45\textwidth}
        \centering
        \tikzset{every picture/.style={line width=0.75pt}} 

\begin{tikzpicture}[x=0.75pt,y=0.75pt,yscale=-0.55,xscale=0.45]

\draw  [fill={rgb, 255:red, 245; green, 166; blue, 35 }  ,fill opacity=0.5 ][line width=0.75]  (124,122.25) .. controls (124,119.35) and (126.46,117) .. (129.5,117) .. controls (132.54,117) and (135,119.35) .. (135,122.25) .. controls (135,125.15) and (132.54,127.5) .. (129.5,127.5) .. controls (126.46,127.5) and (124,125.15) .. (124,122.25) -- cycle ;
\draw  [fill={rgb, 255:red, 245; green, 166; blue, 35 }  ,fill opacity=0.5 ][line width=0.75]  (364,145.25) .. controls (364,142.35) and (366.46,140) .. (369.5,140) .. controls (372.54,140) and (375,142.35) .. (375,145.25) .. controls (375,148.15) and (372.54,150.5) .. (369.5,150.5) .. controls (366.46,150.5) and (364,148.15) .. (364,145.25) -- cycle ;
\draw  [fill={rgb, 255:red, 245; green, 166; blue, 35 }  ,fill opacity=0.5 ][line width=0.75]  (125,210.25) .. controls (125,207.35) and (127.46,205) .. (130.5,205) .. controls (133.54,205) and (136,207.35) .. (136,210.25) .. controls (136,213.15) and (133.54,215.5) .. (130.5,215.5) .. controls (127.46,215.5) and (125,213.15) .. (125,210.25) -- cycle ;
\draw  [fill={rgb, 255:red, 245; green, 166; blue, 35 }  ,fill opacity=0.5 ][line width=0.75]  (366,296.25) .. controls (366,293.35) and (368.46,291) .. (371.5,291) .. controls (374.54,291) and (377,293.35) .. (377,296.25) .. controls (377,299.15) and (374.54,301.5) .. (371.5,301.5) .. controls (368.46,301.5) and (366,299.15) .. (366,296.25) -- cycle ;
\draw  [fill={rgb, 255:red, 155; green, 155; blue, 155 }  ,fill opacity=0.5 ][line width=0.75]  (599,281.25) .. controls (599,278.35) and (601.46,276) .. (604.5,276) .. controls (607.54,276) and (610,278.35) .. (610,281.25) .. controls (610,284.15) and (607.54,286.5) .. (604.5,286.5) .. controls (601.46,286.5) and (599,284.15) .. (599,281.25) -- cycle ;
\draw    (604.5,281.25) -- (130.5,210.25) ;
\draw [shift={(362.56,245.01)}, rotate = 8.52] [fill={rgb, 255:red, 0; green, 0; blue, 0 }  ][line width=0.08]  [draw opacity=0] (10.72,-5.15) -- (0,0) -- (10.72,5.15) -- (7.12,0) -- cycle    ;
\draw    (130.5,210.25) -- (369.5,145.25) ;
\draw [shift={(254.82,176.44)}, rotate = 164.79] [fill={rgb, 255:red, 0; green, 0; blue, 0 }  ][line width=0.08]  [draw opacity=0] (10.72,-5.15) -- (0,0) -- (10.72,5.15) -- (7.12,0) -- cycle    ;
\draw    (369.5,145.25) -- (129.5,122.25) ;
\draw [shift={(244.52,133.27)}, rotate = 5.47] [fill={rgb, 255:red, 0; green, 0; blue, 0 }  ][line width=0.08]  [draw opacity=0] (10.72,-5.15) -- (0,0) -- (10.72,5.15) -- (7.12,0) -- cycle    ;
\draw    (604.5,281.25) -- (129.5,122.25) ;
\draw [shift={(373.16,203.81)}, rotate = 198.51] [fill={rgb, 255:red, 0; green, 0; blue, 0 }  ][line width=0.08]  [draw opacity=0] (10.72,-5.15) -- (0,0) -- (10.72,5.15) -- (7.12,0) -- cycle    ;
\draw  [fill={rgb, 255:red, 155; green, 155; blue, 155 }  ,fill opacity=0.5 ][line width=0.75]  (363,224.25) .. controls (363,221.35) and (365.46,219) .. (368.5,219) .. controls (371.54,219) and (374,221.35) .. (374,224.25) .. controls (374,227.15) and (371.54,229.5) .. (368.5,229.5) .. controls (365.46,229.5) and (363,227.15) .. (363,224.25) -- cycle ;
\draw  [dash pattern={on 4.5pt off 4.5pt}]  (129.5,122.25) -- (368.5,224.25) ;
\draw  [dash pattern={on 4.5pt off 4.5pt}]  (130.5,210.25) -- (368.5,224.25) ;
\draw  [dash pattern={on 4.5pt off 4.5pt}]  (368.5,224.25) -- (604.5,281.25) ;
\draw  [fill={rgb, 255:red, 155; green, 155; blue, 155 }  ,fill opacity=0.5 ][line width=0.75]  (594.5,164.25) .. controls (594.5,161.35) and (596.96,159) .. (600,159) .. controls (603.04,159) and (605.5,161.35) .. (605.5,164.25) .. controls (605.5,167.15) and (603.04,169.5) .. (600,169.5) .. controls (596.96,169.5) and (594.5,167.15) .. (594.5,164.25) -- cycle ;
\draw  [dash pattern={on 4.5pt off 4.5pt}]  (369.5,145.25) -- (600,164.25) ;
\draw  [dash pattern={on 4.5pt off 4.5pt}]  (368.5,224.25) -- (600,164.25) ;

\draw (124,75.4) node [anchor=north west][inner sep=0.75pt]    {$i_{0}$};
\draw (364,75.4) node [anchor=north west][inner sep=0.75pt]    {$i_{1}$};
\draw (597,74.4) node [anchor=north west][inner sep=0.75pt]    {$i_{2}$};
\draw (131.5,130.9) node [anchor=north west][inner sep=0.75pt]    {$v$};
\draw (373.5,304.9) node [anchor=north west][inner sep=0.75pt]    {$u$};
\draw (132.5,218.9) node [anchor=north west][inner sep=0.75pt]    {$u_{1}$};
\draw (371.5,153.9) node [anchor=north west][inner sep=0.75pt]    {$u_{0}$};
\draw (606.5,289.9) node [anchor=north west][inner sep=0.75pt]    {$x$};
\draw (376,227.65) node [anchor=north west][inner sep=0.75pt]    {$y$};
\draw (607.5,167.65) node [anchor=north west][inner sep=0.75pt]    {$z$};

\end{tikzpicture}
        \caption{$P_4=vxu_1u_0v$}
    \end{subfigure}
    \begin{subfigure}[b]{0.45\textwidth}
        \centering
        \tikzset{every picture/.style={line width=0.75pt}} 

\begin{tikzpicture}[x=0.75pt,y=0.75pt,yscale=-0.45,xscale=0.45]

\draw  [fill={rgb, 255:red, 245; green, 166; blue, 35 }  ,fill opacity=0.5 ][line width=0.75]  (124,122.25) .. controls (124,119.35) and (126.46,117) .. (129.5,117) .. controls (132.54,117) and (135,119.35) .. (135,122.25) .. controls (135,125.15) and (132.54,127.5) .. (129.5,127.5) .. controls (126.46,127.5) and (124,125.15) .. (124,122.25) -- cycle ;
\draw  [fill={rgb, 255:red, 245; green, 166; blue, 35 }  ,fill opacity=0.5 ][line width=0.75]  (364,145.25) .. controls (364,142.35) and (366.46,140) .. (369.5,140) .. controls (372.54,140) and (375,142.35) .. (375,145.25) .. controls (375,148.15) and (372.54,150.5) .. (369.5,150.5) .. controls (366.46,150.5) and (364,148.15) .. (364,145.25) -- cycle ;
\draw  [fill={rgb, 255:red, 245; green, 166; blue, 35 }  ,fill opacity=0.5 ][line width=0.75]  (125,210.25) .. controls (125,207.35) and (127.46,205) .. (130.5,205) .. controls (133.54,205) and (136,207.35) .. (136,210.25) .. controls (136,213.15) and (133.54,215.5) .. (130.5,215.5) .. controls (127.46,215.5) and (125,213.15) .. (125,210.25) -- cycle ;
\draw  [fill={rgb, 255:red, 155; green, 155; blue, 155 }  ,fill opacity=0.5 ][line width=0.75]  (597,240.25) .. controls (597,237.35) and (599.46,235) .. (602.5,235) .. controls (605.54,235) and (608,237.35) .. (608,240.25) .. controls (608,243.15) and (605.54,245.5) .. (602.5,245.5) .. controls (599.46,245.5) and (597,243.15) .. (597,240.25) -- cycle ;
\draw    (130.5,210.25) -- (369.5,145.25) ;
\draw [shift={(254.82,176.44)}, rotate = 164.79] [fill={rgb, 255:red, 0; green, 0; blue, 0 }  ][line width=0.08]  [draw opacity=0] (10.72,-5.15) -- (0,0) -- (10.72,5.15) -- (7.12,0) -- cycle    ;
\draw    (602.5,240.25) -- (129.5,122.25) ;
\draw [shift={(372.31,182.82)}, rotate = 194.01] [fill={rgb, 255:red, 0; green, 0; blue, 0 }  ][line width=0.08]  [draw opacity=0] (10.72,-5.15) -- (0,0) -- (10.72,5.15) -- (7.12,0) -- cycle    ;
\draw  [fill={rgb, 255:red, 155; green, 155; blue, 155 }  ,fill opacity=0.5 ][line width=0.75]  (366,207.25) .. controls (366,204.35) and (368.46,202) .. (371.5,202) .. controls (374.54,202) and (377,204.35) .. (377,207.25) .. controls (377,210.15) and (374.54,212.5) .. (371.5,212.5) .. controls (368.46,212.5) and (366,210.15) .. (366,207.25) -- cycle ;
\draw  [dash pattern={on 4.5pt off 4.5pt}]  (129.5,122.25) -- (371.5,207.25) ;
\draw  [dash pattern={on 4.5pt off 4.5pt}]  (130.5,210.25) -- (371.5,207.25) ;
\draw  [dash pattern={on 4.5pt off 4.5pt}]  (371.5,207.25) -- (602.5,240.25) ;
\draw  [fill={rgb, 255:red, 155; green, 155; blue, 155 }  ,fill opacity=0.5 ][line width=0.75]  (594.5,164.25) .. controls (594.5,161.35) and (596.96,159) .. (600,159) .. controls (603.04,159) and (605.5,161.35) .. (605.5,164.25) .. controls (605.5,167.15) and (603.04,169.5) .. (600,169.5) .. controls (596.96,169.5) and (594.5,167.15) .. (594.5,164.25) -- cycle ;
\draw  [dash pattern={on 4.5pt off 4.5pt}]  (371.5,207.25) -- (600,164.25) ;
\draw    (369.5,145.25) -- (600,164.25) ;
\draw [shift={(489.73,155.16)}, rotate = 184.71] [fill={rgb, 255:red, 0; green, 0; blue, 0 }  ][line width=0.08]  [draw opacity=0] (10.72,-5.15) -- (0,0) -- (10.72,5.15) -- (7.12,0) -- cycle    ;
\draw    (129.5,122.25) .. controls (169.5,92.25) and (579,125.5) .. (600,164.25) ;
\draw [shift={(361.97,116.65)}, rotate = 4.71] [fill={rgb, 255:red, 0; green, 0; blue, 0 }  ][line width=0.08]  [draw opacity=0] (10.72,-5.15) -- (0,0) -- (10.72,5.15) -- (7.12,0) -- cycle    ;
\draw    (130.5,210.25) .. controls (194,239.5) and (562.5,270.25) .. (602.5,240.25) ;
\draw [shift={(359.54,246.28)}, rotate = 4.44] [fill={rgb, 255:red, 0; green, 0; blue, 0 }  ][line width=0.08]  [draw opacity=0] (10.72,-5.15) -- (0,0) -- (10.72,5.15) -- (7.12,0) -- cycle    ;

\draw (124,75.4) node [anchor=north west][inner sep=0.75pt]    {$i_{0}$};
\draw (364,75.4) node [anchor=north west][inner sep=0.75pt]    {$i_{1}$};
\draw (597,74.4) node [anchor=north west][inner sep=0.75pt]    {$i_{2}$};
\draw (131.5,130.9) node [anchor=north west][inner sep=0.75pt]    {$v$};
\draw (132.5,218.9) node [anchor=north west][inner sep=0.75pt]    {$u_{1}$};
\draw (371.5,153.9) node [anchor=north west][inner sep=0.75pt]    {$u_{0}$};
\draw (606.5,255) node [anchor=north west][inner sep=0.75pt]    {$x$};
\draw (373.5,215.9) node [anchor=north west][inner sep=0.75pt]    {$y$};
\draw (607.5,167.65) node [anchor=north west][inner sep=0.75pt]    {$z$};

\end{tikzpicture}
        \caption{$P_5=vxu_1u_0zv$}
    \end{subfigure}
    \begin{subfigure}[b]{0.45\textwidth}
        \centering
        \tikzset{every picture/.style={line width=0.75pt}} 

\begin{tikzpicture}[x=0.75pt,y=0.75pt,yscale=-0.45,xscale=0.45]

\draw  [fill={rgb, 255:red, 245; green, 166; blue, 35 }  ,fill opacity=0.5 ][line width=0.75]  (124,122.25) .. controls (124,119.35) and (126.46,117) .. (129.5,117) .. controls (132.54,117) and (135,119.35) .. (135,122.25) .. controls (135,125.15) and (132.54,127.5) .. (129.5,127.5) .. controls (126.46,127.5) and (124,125.15) .. (124,122.25) -- cycle ;
\draw  [fill={rgb, 255:red, 245; green, 166; blue, 35 }  ,fill opacity=0.5 ][line width=0.75]  (364,145.25) .. controls (364,142.35) and (366.46,140) .. (369.5,140) .. controls (372.54,140) and (375,142.35) .. (375,145.25) .. controls (375,148.15) and (372.54,150.5) .. (369.5,150.5) .. controls (366.46,150.5) and (364,148.15) .. (364,145.25) -- cycle ;
\draw  [fill={rgb, 255:red, 245; green, 166; blue, 35 }  ,fill opacity=0.5 ][line width=0.75]  (125,210.25) .. controls (125,207.35) and (127.46,205) .. (130.5,205) .. controls (133.54,205) and (136,207.35) .. (136,210.25) .. controls (136,213.15) and (133.54,215.5) .. (130.5,215.5) .. controls (127.46,215.5) and (125,213.15) .. (125,210.25) -- cycle ;
\draw  [fill={rgb, 255:red, 155; green, 155; blue, 155 }  ,fill opacity=0.5 ][line width=0.75]  (597,240.25) .. controls (597,237.35) and (599.46,235) .. (602.5,235) .. controls (605.54,235) and (608,237.35) .. (608,240.25) .. controls (608,243.15) and (605.54,245.5) .. (602.5,245.5) .. controls (599.46,245.5) and (597,243.15) .. (597,240.25) -- cycle ;
\draw    (130.5,210.25) -- (600,164.25) ;
\draw [shift={(370.23,186.76)}, rotate = 174.4] [fill={rgb, 255:red, 0; green, 0; blue, 0 }  ][line width=0.08]  [draw opacity=0] (10.72,-5.15) -- (0,0) -- (10.72,5.15) -- (7.12,0) -- cycle    ;
\draw    (602.5,240.25) -- (129.5,122.25) ;
\draw [shift={(372.31,182.82)}, rotate = 194.01] [fill={rgb, 255:red, 0; green, 0; blue, 0 }  ][line width=0.08]  [draw opacity=0] (10.72,-5.15) -- (0,0) -- (10.72,5.15) -- (7.12,0) -- cycle    ;
\draw  [fill={rgb, 255:red, 155; green, 155; blue, 155 }  ,fill opacity=0.5 ][line width=0.75]  (366,207.25) .. controls (366,204.35) and (368.46,202) .. (371.5,202) .. controls (374.54,202) and (377,204.35) .. (377,207.25) .. controls (377,210.15) and (374.54,212.5) .. (371.5,212.5) .. controls (368.46,212.5) and (366,210.15) .. (366,207.25) -- cycle ;
\draw  [dash pattern={on 4.5pt off 4.5pt}]  (129.5,122.25) -- (371.5,207.25) ;
\draw  [dash pattern={on 4.5pt off 4.5pt}]  (130.5,210.25) -- (371.5,207.25) ;
\draw  [dash pattern={on 4.5pt off 4.5pt}]  (371.5,207.25) -- (602.5,240.25) ;
\draw  [fill={rgb, 255:red, 155; green, 155; blue, 155 }  ,fill opacity=0.5 ][line width=0.75]  (594.5,164.25) .. controls (594.5,161.35) and (596.96,159) .. (600,159) .. controls (603.04,159) and (605.5,161.35) .. (605.5,164.25) .. controls (605.5,167.15) and (603.04,169.5) .. (600,169.5) .. controls (596.96,169.5) and (594.5,167.15) .. (594.5,164.25) -- cycle ;
\draw  [dash pattern={on 4.5pt off 4.5pt}]  (371.5,207.25) -- (600,164.25) ;
\draw    (129.5,122.25) .. controls (169.5,92.25) and (579,125.5) .. (600,164.25) ;
\draw [shift={(361.97,116.65)}, rotate = 4.71] [fill={rgb, 255:red, 0; green, 0; blue, 0 }  ][line width=0.08]  [draw opacity=0] (10.72,-5.15) -- (0,0) -- (10.72,5.15) -- (7.12,0) -- cycle    ;
\draw    (130.5,210.25) .. controls (194,239.5) and (562.5,270.25) .. (602.5,240.25) ;
\draw [shift={(359.54,246.28)}, rotate = 4.44] [fill={rgb, 255:red, 0; green, 0; blue, 0 }  ][line width=0.08]  [draw opacity=0] (10.72,-5.15) -- (0,0) -- (10.72,5.15) -- (7.12,0) -- cycle    ;

\draw (124,75.4) node [anchor=north west][inner sep=0.75pt]    {$i_{0}$};
\draw (364,75.4) node [anchor=north west][inner sep=0.75pt]    {$i_{1}$};
\draw (597,74.4) node [anchor=north west][inner sep=0.75pt]    {$i_{2}$};
\draw (131.5,130.9) node [anchor=north west][inner sep=0.75pt]    {$v$};
\draw (132.5,218.9) node [anchor=north west][inner sep=0.75pt]    {$u_{1}$};
\draw (371.5,153.9) node [anchor=north west][inner sep=0.75pt]    {$u_{0}$};
\draw (606.5,255) node [anchor=north west][inner sep=0.75pt]    {$x$};
\draw (373.5,215.9) node [anchor=north west][inner sep=0.75pt]    {$y$};
\draw (607.5,167.65) node [anchor=north west][inner sep=0.75pt]    {$z$};

\end{tikzpicture}
        \caption{$P_6=vxu_1zv$}
    \end{subfigure}
    \begin{subfigure}[b]{0.45\textwidth}
        \centering
        \tikzset{every picture/.style={line width=0.75pt}} 

\begin{tikzpicture}[x=0.75pt,y=0.75pt,yscale=-0.45,xscale=0.45]

\draw  [fill={rgb, 255:red, 245; green, 166; blue, 35 }  ,fill opacity=0.5 ][line width=0.75]  (124,122.25) .. controls (124,119.35) and (126.46,117) .. (129.5,117) .. controls (132.54,117) and (135,119.35) .. (135,122.25) .. controls (135,125.15) and (132.54,127.5) .. (129.5,127.5) .. controls (126.46,127.5) and (124,125.15) .. (124,122.25) -- cycle ;
\draw  [fill={rgb, 255:red, 245; green, 166; blue, 35 }  ,fill opacity=0.5 ][line width=0.75]  (125,210.25) .. controls (125,207.35) and (127.46,205) .. (130.5,205) .. controls (133.54,205) and (136,207.35) .. (136,210.25) .. controls (136,213.15) and (133.54,215.5) .. (130.5,215.5) .. controls (127.46,215.5) and (125,213.15) .. (125,210.25) -- cycle ;
\draw  [fill={rgb, 255:red, 155; green, 155; blue, 155 }  ,fill opacity=0.5 ][line width=0.75]  (597,240.25) .. controls (597,237.35) and (599.46,235) .. (602.5,235) .. controls (605.54,235) and (608,237.35) .. (608,240.25) .. controls (608,243.15) and (605.54,245.5) .. (602.5,245.5) .. controls (599.46,245.5) and (597,243.15) .. (597,240.25) -- cycle ;
\draw    (602.5,240.25) -- (129.5,122.25) ;
\draw [shift={(372.31,182.82)}, rotate = 194.01] [fill={rgb, 255:red, 0; green, 0; blue, 0 }  ][line width=0.08]  [draw opacity=0] (10.72,-5.15) -- (0,0) -- (10.72,5.15) -- (7.12,0) -- cycle    ;
\draw  [fill={rgb, 255:red, 155; green, 155; blue, 155 }  ,fill opacity=0.5 ][line width=0.75]  (366,207.25) .. controls (366,204.35) and (368.46,202) .. (371.5,202) .. controls (374.54,202) and (377,204.35) .. (377,207.25) .. controls (377,210.15) and (374.54,212.5) .. (371.5,212.5) .. controls (368.46,212.5) and (366,210.15) .. (366,207.25) -- cycle ;
\draw  [dash pattern={on 4.5pt off 4.5pt}]  (129.5,122.25) -- (371.5,207.25) ;
\draw    (130.5,210.25) -- (371.5,207.25) ;
\draw [shift={(256,208.69)}, rotate = 179.29] [fill={rgb, 255:red, 0; green, 0; blue, 0 }  ][line width=0.08]  [draw opacity=0] (10.72,-5.15) -- (0,0) -- (10.72,5.15) -- (7.12,0) -- cycle    ;
\draw  [dash pattern={on 4.5pt off 4.5pt}]  (371.5,207.25) -- (602.5,240.25) ;
\draw  [fill={rgb, 255:red, 155; green, 155; blue, 155 }  ,fill opacity=0.5 ][line width=0.75]  (594.5,164.25) .. controls (594.5,161.35) and (596.96,159) .. (600,159) .. controls (603.04,159) and (605.5,161.35) .. (605.5,164.25) .. controls (605.5,167.15) and (603.04,169.5) .. (600,169.5) .. controls (596.96,169.5) and (594.5,167.15) .. (594.5,164.25) -- cycle ;
\draw    (371.5,207.25) -- (600,164.25) ;
\draw [shift={(490.66,184.83)}, rotate = 169.34] [fill={rgb, 255:red, 0; green, 0; blue, 0 }  ][line width=0.08]  [draw opacity=0] (10.72,-5.15) -- (0,0) -- (10.72,5.15) -- (7.12,0) -- cycle    ;
\draw    (129.5,122.25) .. controls (169.5,92.25) and (579,125.5) .. (600,164.25) ;
\draw [shift={(361.97,116.65)}, rotate = 4.71] [fill={rgb, 255:red, 0; green, 0; blue, 0 }  ][line width=0.08]  [draw opacity=0] (10.72,-5.15) -- (0,0) -- (10.72,5.15) -- (7.12,0) -- cycle    ;
\draw    (130.5,210.25) .. controls (194,239.5) and (562.5,270.25) .. (602.5,240.25) ;
\draw [shift={(359.54,246.28)}, rotate = 4.44] [fill={rgb, 255:red, 0; green, 0; blue, 0 }  ][line width=0.08]  [draw opacity=0] (10.72,-5.15) -- (0,0) -- (10.72,5.15) -- (7.12,0) -- cycle    ;

\draw (124,75.4) node [anchor=north west][inner sep=0.75pt]    {$i_{0}$};
\draw (364,75.4) node [anchor=north west][inner sep=0.75pt]    {$i_{1}$};
\draw (597,74.4) node [anchor=north west][inner sep=0.75pt]    {$i_{2}$};
\draw (131.5,130.9) node [anchor=north west][inner sep=0.75pt]    {$v$};
\draw (132.5,218.9) node [anchor=north west][inner sep=0.75pt]    {$u_{1}$};
\draw (610,243.65) node [anchor=north west][inner sep=0.75pt]    {$x$};
\draw (373.5,215.9) node [anchor=north west][inner sep=0.75pt]    {$y$};
\draw (607.5,167.65) node [anchor=north west][inner sep=0.75pt]    {$z$};

\end{tikzpicture}
        \caption{$P_7=vxu_1yzv$}
    \end{subfigure}
    \begin{subfigure}[b]{0.45\textwidth}
        \centering
        \tikzset{every picture/.style={line width=0.75pt}} 

\begin{tikzpicture}[x=0.75pt,y=0.75pt,yscale=-0.45,xscale=0.45]

\draw  [fill={rgb, 255:red, 245; green, 166; blue, 35 }  ,fill opacity=0.5 ][line width=0.75]  (124,122.25) .. controls (124,119.35) and (126.46,117) .. (129.5,117) .. controls (132.54,117) and (135,119.35) .. (135,122.25) .. controls (135,125.15) and (132.54,127.5) .. (129.5,127.5) .. controls (126.46,127.5) and (124,125.15) .. (124,122.25) -- cycle ;
\draw  [fill={rgb, 255:red, 245; green, 166; blue, 35 }  ,fill opacity=0.5 ][line width=0.75]  (125,210.25) .. controls (125,207.35) and (127.46,205) .. (130.5,205) .. controls (133.54,205) and (136,207.35) .. (136,210.25) .. controls (136,213.15) and (133.54,215.5) .. (130.5,215.5) .. controls (127.46,215.5) and (125,213.15) .. (125,210.25) -- cycle ;
\draw  [fill={rgb, 255:red, 155; green, 155; blue, 155 }  ,fill opacity=0.5 ][line width=0.75]  (599,222.25) .. controls (599,219.35) and (601.46,217) .. (604.5,217) .. controls (607.54,217) and (610,219.35) .. (610,222.25) .. controls (610,225.15) and (607.54,227.5) .. (604.5,227.5) .. controls (601.46,227.5) and (599,225.15) .. (599,222.25) -- cycle ;
\draw  [fill={rgb, 255:red, 155; green, 155; blue, 155 }  ,fill opacity=0.5 ][line width=0.75]  (368.5,176.5) .. controls (368.5,173.6) and (370.96,171.25) .. (374,171.25) .. controls (377.04,171.25) and (379.5,173.6) .. (379.5,176.5) .. controls (379.5,179.4) and (377.04,181.75) .. (374,181.75) .. controls (370.96,181.75) and (368.5,179.4) .. (368.5,176.5) -- cycle ;
\draw    (129.5,122.25) -- (374,176.5) ;
\draw [shift={(256.63,150.46)}, rotate = 192.51] [fill={rgb, 255:red, 0; green, 0; blue, 0 }  ][line width=0.08]  [draw opacity=0] (10.72,-5.15) -- (0,0) -- (10.72,5.15) -- (7.12,0) -- cycle    ;
\draw    (130.5,210.25) -- (374,176.5) ;
\draw [shift={(257.2,192.69)}, rotate = 172.11] [fill={rgb, 255:red, 0; green, 0; blue, 0 }  ][line width=0.08]  [draw opacity=0] (10.72,-5.15) -- (0,0) -- (10.72,5.15) -- (7.12,0) -- cycle    ;
\draw    (374,176.5) -- (604.5,222.25) ;
\draw [shift={(494.15,200.35)}, rotate = 191.23] [fill={rgb, 255:red, 0; green, 0; blue, 0 }  ][line width=0.08]  [draw opacity=0] (10.72,-5.15) -- (0,0) -- (10.72,5.15) -- (7.12,0) -- cycle    ;
\draw  [fill={rgb, 255:red, 155; green, 155; blue, 155 }  ,fill opacity=0.5 ][line width=0.75]  (599.5,132.25) .. controls (599.5,129.35) and (601.96,127) .. (605,127) .. controls (608.04,127) and (610.5,129.35) .. (610.5,132.25) .. controls (610.5,135.15) and (608.04,137.5) .. (605,137.5) .. controls (601.96,137.5) and (599.5,135.15) .. (599.5,132.25) -- cycle ;
\draw    (374,176.5) -- (605,132.25) ;
\draw [shift={(494.41,153.43)}, rotate = 169.16] [fill={rgb, 255:red, 0; green, 0; blue, 0 }  ][line width=0.08]  [draw opacity=0] (10.72,-5.15) -- (0,0) -- (10.72,5.15) -- (7.12,0) -- cycle    ;
\draw    (605,132.25) -- (129.5,122.25) ;
\draw [shift={(362.25,127.14)}, rotate = 1.2] [fill={rgb, 255:red, 0; green, 0; blue, 0 }  ][line width=0.08]  [draw opacity=0] (10.72,-5.15) -- (0,0) -- (10.72,5.15) -- (7.12,0) -- cycle    ;
\draw    (604.5,222.25) -- (130.5,210.25) ;
\draw [shift={(362.5,216.12)}, rotate = 1.45] [fill={rgb, 255:red, 0; green, 0; blue, 0 }  ][line width=0.08]  [draw opacity=0] (10.72,-5.15) -- (0,0) -- (10.72,5.15) -- (7.12,0) -- cycle    ;

\draw (124,75.4) node [anchor=north west][inner sep=0.75pt]    {$i_{0}$};
\draw (364,75.4) node [anchor=north west][inner sep=0.75pt]    {$i_{1}$};
\draw (598,79.4) node [anchor=north west][inner sep=0.75pt]    {$i_{2}$};
\draw (131.5,130.9) node [anchor=north west][inner sep=0.75pt]    {$v$};
\draw (132.5,218.9) node [anchor=north west][inner sep=0.75pt]    {$u_{1}$};
\draw (612,225.65) node [anchor=north west][inner sep=0.75pt]    {$x$};
\draw (381.5,179.9) node [anchor=north west][inner sep=0.75pt]    {$y$};
\draw (612.5,135.65) node [anchor=north west][inner sep=0.75pt]    {$z$};

\end{tikzpicture}
        \caption{$P_8=vyxu_1yzv$}
    \end{subfigure}
    \begin{subfigure}[b]{0.45\textwidth}
        \centering
        \tikzset{every picture/.style={line width=0.75pt}} 

\begin{tikzpicture}[x=0.75pt,y=0.75pt,yscale=-0.45,xscale=0.45]

\draw  [fill={rgb, 255:red, 245; green, 166; blue, 35 }  ,fill opacity=0.5 ][line width=0.75]  (124,122.25) .. controls (124,119.35) and (126.46,117) .. (129.5,117) .. controls (132.54,117) and (135,119.35) .. (135,122.25) .. controls (135,125.15) and (132.54,127.5) .. (129.5,127.5) .. controls (126.46,127.5) and (124,125.15) .. (124,122.25) -- cycle ;
\draw  [fill={rgb, 255:red, 245; green, 166; blue, 35 }  ,fill opacity=0.5 ][line width=0.75]  (125,210.25) .. controls (125,207.35) and (127.46,205) .. (130.5,205) .. controls (133.54,205) and (136,207.35) .. (136,210.25) .. controls (136,213.15) and (133.54,215.5) .. (130.5,215.5) .. controls (127.46,215.5) and (125,213.15) .. (125,210.25) -- cycle ;
\draw  [fill={rgb, 255:red, 155; green, 155; blue, 155 }  ,fill opacity=0.5 ][line width=0.75]  (599,222.25) .. controls (599,219.35) and (601.46,217) .. (604.5,217) .. controls (607.54,217) and (610,219.35) .. (610,222.25) .. controls (610,225.15) and (607.54,227.5) .. (604.5,227.5) .. controls (601.46,227.5) and (599,225.15) .. (599,222.25) -- cycle ;
\draw  [fill={rgb, 255:red, 155; green, 155; blue, 155 }  ,fill opacity=0.5 ][line width=0.75]  (368.5,176.5) .. controls (368.5,173.6) and (370.96,171.25) .. (374,171.25) .. controls (377.04,171.25) and (379.5,173.6) .. (379.5,176.5) .. controls (379.5,179.4) and (377.04,181.75) .. (374,181.75) .. controls (370.96,181.75) and (368.5,179.4) .. (368.5,176.5) -- cycle ;
\draw    (129.5,122.25) -- (374,176.5) ;
\draw [shift={(256.63,150.46)}, rotate = 192.51] [fill={rgb, 255:red, 0; green, 0; blue, 0 }  ][line width=0.08]  [draw opacity=0] (10.72,-5.15) -- (0,0) -- (10.72,5.15) -- (7.12,0) -- cycle    ;
\draw  [fill={rgb, 255:red, 155; green, 155; blue, 155 }  ,fill opacity=0.5 ][line width=0.75]  (599.5,132.25) .. controls (599.5,129.35) and (601.96,127) .. (605,127) .. controls (608.04,127) and (610.5,129.35) .. (610.5,132.25) .. controls (610.5,135.15) and (608.04,137.5) .. (605,137.5) .. controls (601.96,137.5) and (599.5,135.15) .. (599.5,132.25) -- cycle ;
\draw    (374,176.5) -- (605,132.25) ;
\draw [shift={(494.41,153.43)}, rotate = 169.16] [fill={rgb, 255:red, 0; green, 0; blue, 0 }  ][line width=0.08]  [draw opacity=0] (10.72,-5.15) -- (0,0) -- (10.72,5.15) -- (7.12,0) -- cycle    ;
\draw    (605,132.25) -- (129.5,122.25) ;
\draw [shift={(362.25,127.14)}, rotate = 1.2] [fill={rgb, 255:red, 0; green, 0; blue, 0 }  ][line width=0.08]  [draw opacity=0] (10.72,-5.15) -- (0,0) -- (10.72,5.15) -- (7.12,0) -- cycle    ;
\draw    (130.5,210.25) .. controls (170.5,180.25) and (314,170.5) .. (374,176.5) ;
\draw [shift={(254.94,178.8)}, rotate = 174.13] [fill={rgb, 255:red, 0; green, 0; blue, 0 }  ][line width=0.08]  [draw opacity=0] (10.72,-5.15) -- (0,0) -- (10.72,5.15) -- (7.12,0) -- cycle    ;
\draw    (130.5,210.25) .. controls (176,229.5) and (334,206.5) .. (374,176.5) ;
\draw [shift={(249.79,212.47)}, rotate = 352.16] [fill={rgb, 255:red, 0; green, 0; blue, 0 }  ][line width=0.08]  [draw opacity=0] (10.72,-5.15) -- (0,0) -- (10.72,5.15) -- (7.12,0) -- cycle    ;

\draw (124,75.4) node [anchor=north west][inner sep=0.75pt]    {$i_{0}$};
\draw (364,75.4) node [anchor=north west][inner sep=0.75pt]    {$i_{1}$};
\draw (598,79.4) node [anchor=north west][inner sep=0.75pt]    {$i_{2}$};
\draw (131.5,130.9) node [anchor=north west][inner sep=0.75pt]    {$v$};
\draw (132.5,218.9) node [anchor=north west][inner sep=0.75pt]    {$u_{1}$};
\draw (612,225.65) node [anchor=north west][inner sep=0.75pt]    {$x$};
\draw (381.5,179.9) node [anchor=north west][inner sep=0.75pt]    {$y$};
\draw (612.5,135.65) node [anchor=north west][inner sep=0.75pt]    {$z$};

\end{tikzpicture}
        \caption{$P_9=vyu_1yzv$}
    \end{subfigure}
    \begin{subfigure}[b]{0.45\textwidth}
        \centering
        \tikzset{every picture/.style={line width=0.75pt}} 

\begin{tikzpicture}[x=0.75pt,y=0.75pt,yscale=-0.45,xscale=0.45]

\draw  [fill={rgb, 255:red, 245; green, 166; blue, 35 }  ,fill opacity=0.5 ][line width=0.75]  (124,122.25) .. controls (124,119.35) and (126.46,117) .. (129.5,117) .. controls (132.54,117) and (135,119.35) .. (135,122.25) .. controls (135,125.15) and (132.54,127.5) .. (129.5,127.5) .. controls (126.46,127.5) and (124,125.15) .. (124,122.25) -- cycle ;
\draw  [fill={rgb, 255:red, 245; green, 166; blue, 35 }  ,fill opacity=0.5 ][line width=0.75]  (125,210.25) .. controls (125,207.35) and (127.46,205) .. (130.5,205) .. controls (133.54,205) and (136,207.35) .. (136,210.25) .. controls (136,213.15) and (133.54,215.5) .. (130.5,215.5) .. controls (127.46,215.5) and (125,213.15) .. (125,210.25) -- cycle ;
\draw  [fill={rgb, 255:red, 155; green, 155; blue, 155 }  ,fill opacity=0.5 ][line width=0.75]  (368.5,176.5) .. controls (368.5,173.6) and (370.96,171.25) .. (374,171.25) .. controls (377.04,171.25) and (379.5,173.6) .. (379.5,176.5) .. controls (379.5,179.4) and (377.04,181.75) .. (374,181.75) .. controls (370.96,181.75) and (368.5,179.4) .. (368.5,176.5) -- cycle ;
\draw  [fill={rgb, 255:red, 155; green, 155; blue, 155 }  ,fill opacity=0.5 ][line width=0.75]  (599.5,132.25) .. controls (599.5,129.35) and (601.96,127) .. (605,127) .. controls (608.04,127) and (610.5,129.35) .. (610.5,132.25) .. controls (610.5,135.15) and (608.04,137.5) .. (605,137.5) .. controls (601.96,137.5) and (599.5,135.15) .. (599.5,132.25) -- cycle ;
\draw    (130.5,210.25) .. controls (170.5,180.25) and (314,170.5) .. (374,176.5) ;
\draw [shift={(254.94,178.8)}, rotate = 174.13] [fill={rgb, 255:red, 0; green, 0; blue, 0 }  ][line width=0.08]  [draw opacity=0] (10.72,-5.15) -- (0,0) -- (10.72,5.15) -- (7.12,0) -- cycle    ;
\draw    (130.5,210.25) .. controls (176,229.5) and (334,206.5) .. (374,176.5) ;
\draw [shift={(249.79,212.47)}, rotate = 352.16] [fill={rgb, 255:red, 0; green, 0; blue, 0 }  ][line width=0.08]  [draw opacity=0] (10.72,-5.15) -- (0,0) -- (10.72,5.15) -- (7.12,0) -- cycle    ;
\draw    (129.5,122.25) .. controls (162,147.5) and (319,167.5) .. (368.5,176.5) ;
\draw [shift={(253.04,157.39)}, rotate = 190.36] [fill={rgb, 255:red, 0; green, 0; blue, 0 }  ][line width=0.08]  [draw opacity=0] (10.72,-5.15) -- (0,0) -- (10.72,5.15) -- (7.12,0) -- cycle    ;
\draw    (129.5,122.25) .. controls (169.5,92.25) and (359,146.5) .. (374,176.5) ;
\draw [shift={(251.33,124.52)}, rotate = 13.29] [fill={rgb, 255:red, 0; green, 0; blue, 0 }  ][line width=0.08]  [draw opacity=0] (10.72,-5.15) -- (0,0) -- (10.72,5.15) -- (7.12,0) -- cycle    ;

\draw (124,75.4) node [anchor=north west][inner sep=0.75pt]    {$i_{0}$};
\draw (364,75.4) node [anchor=north west][inner sep=0.75pt]    {$i_{1}$};
\draw (598,79.4) node [anchor=north west][inner sep=0.75pt]    {$i_{2}$};
\draw (131.5,130.9) node [anchor=north west][inner sep=0.75pt]    {$v$};
\draw (132.5,218.9) node [anchor=north west][inner sep=0.75pt]    {$u_{1}$};
\draw (381.5,179.9) node [anchor=north west][inner sep=0.75pt]    {$y$};
\draw (612.5,135.65) node [anchor=north west][inner sep=0.75pt]    {$z$};

\end{tikzpicture}
        \caption{$P_{10}=vyu_1yv$}
    \end{subfigure}
    \caption{contraction $T_{u'u}$ when $u\in \buildings^{\set{i_1}}$ and $u'\in \buildings_s^{\set{i_0}}$}
\end{figure}

\begin{figure}
    \centering
    \begin{subfigure}[b]{0.45\textwidth}
        \centering
        \tikzset{every picture/.style={line width=0.75pt}} 

\begin{tikzpicture}[x=0.75pt,y=0.75pt,yscale=-0.75,xscale=0.55]

\draw  [fill={rgb, 255:red, 245; green, 166; blue, 35 }  ,fill opacity=0.5 ][line width=0.75]  (124,122.25) .. controls (124,119.35) and (126.46,117) .. (129.5,117) .. controls (132.54,117) and (135,119.35) .. (135,122.25) .. controls (135,125.15) and (132.54,127.5) .. (129.5,127.5) .. controls (126.46,127.5) and (124,125.15) .. (124,122.25) -- cycle ;
\draw  [fill={rgb, 255:red, 245; green, 166; blue, 35 }  ,fill opacity=0.5 ][line width=0.75]  (364,145.25) .. controls (364,142.35) and (366.46,140) .. (369.5,140) .. controls (372.54,140) and (375,142.35) .. (375,145.25) .. controls (375,148.15) and (372.54,150.5) .. (369.5,150.5) .. controls (366.46,150.5) and (364,148.15) .. (364,145.25) -- cycle ;
\draw  [fill={rgb, 255:red, 245; green, 166; blue, 35 }  ,fill opacity=0.5 ][line width=0.75]  (365,218.25) .. controls (365,215.35) and (367.46,213) .. (370.5,213) .. controls (373.54,213) and (376,215.35) .. (376,218.25) .. controls (376,221.15) and (373.54,223.5) .. (370.5,223.5) .. controls (367.46,223.5) and (365,221.15) .. (365,218.25) -- cycle ;
\draw  [fill={rgb, 255:red, 245; green, 166; blue, 35 }  ,fill opacity=0.5 ][line width=0.75]  (125,210.25) .. controls (125,207.35) and (127.46,205) .. (130.5,205) .. controls (133.54,205) and (136,207.35) .. (136,210.25) .. controls (136,213.15) and (133.54,215.5) .. (130.5,215.5) .. controls (127.46,215.5) and (125,213.15) .. (125,210.25) -- cycle ;
\draw  [fill={rgb, 255:red, 245; green, 166; blue, 35 }  ,fill opacity=0.5 ][line width=0.75]  (127,296.25) .. controls (127,293.35) and (129.46,291) .. (132.5,291) .. controls (135.54,291) and (138,293.35) .. (138,296.25) .. controls (138,299.15) and (135.54,301.5) .. (132.5,301.5) .. controls (129.46,301.5) and (127,299.15) .. (127,296.25) -- cycle ;
\draw  [fill={rgb, 255:red, 245; green, 166; blue, 35 }  ,fill opacity=0.5 ][line width=0.75]  (598,295.25) .. controls (598,292.35) and (600.46,290) .. (603.5,290) .. controls (606.54,290) and (609,292.35) .. (609,295.25) .. controls (609,298.15) and (606.54,300.5) .. (603.5,300.5) .. controls (600.46,300.5) and (598,298.15) .. (598,295.25) -- cycle ;
\draw  [fill={rgb, 255:red, 155; green, 155; blue, 155 }  ,fill opacity=0.5 ][line width=0.75]  (364,271.25) .. controls (364,268.35) and (366.46,266) .. (369.5,266) .. controls (372.54,266) and (375,268.35) .. (375,271.25) .. controls (375,274.15) and (372.54,276.5) .. (369.5,276.5) .. controls (366.46,276.5) and (364,274.15) .. (364,271.25) -- cycle ;
\draw    (129.5,122.25) -- (370.5,218.25) ;
\draw [shift={(254.65,172.1)}, rotate = 201.72] [fill={rgb, 255:red, 0; green, 0; blue, 0 }  ][line width=0.08]  [draw opacity=0] (10.72,-5.15) -- (0,0) -- (10.72,5.15) -- (7.12,0) -- cycle    ;
\draw    (370.5,218.25) -- (132.5,296.25) ;
\draw [shift={(246.75,258.81)}, rotate = 341.85] [fill={rgb, 255:red, 0; green, 0; blue, 0 }  ][line width=0.08]  [draw opacity=0] (10.72,-5.15) -- (0,0) -- (10.72,5.15) -- (7.12,0) -- cycle    ;
\draw    (132.5,296.25) -- (603.5,295.25) ;
\draw [shift={(373,295.74)}, rotate = 179.88] [fill={rgb, 255:red, 0; green, 0; blue, 0 }  ][line width=0.08]  [draw opacity=0] (10.72,-5.15) -- (0,0) -- (10.72,5.15) -- (7.12,0) -- cycle    ;
\draw    (603.5,295.25) -- (130.5,210.25) ;
\draw [shift={(362.08,251.87)}, rotate = 10.19] [fill={rgb, 255:red, 0; green, 0; blue, 0 }  ][line width=0.08]  [draw opacity=0] (10.72,-5.15) -- (0,0) -- (10.72,5.15) -- (7.12,0) -- cycle    ;
\draw    (130.5,210.25) -- (369.5,145.25) ;
\draw [shift={(254.82,176.44)}, rotate = 164.79] [fill={rgb, 255:red, 0; green, 0; blue, 0 }  ][line width=0.08]  [draw opacity=0] (10.72,-5.15) -- (0,0) -- (10.72,5.15) -- (7.12,0) -- cycle    ;
\draw    (369.5,145.25) -- (129.5,122.25) ;
\draw [shift={(244.52,133.27)}, rotate = 5.47] [fill={rgb, 255:red, 0; green, 0; blue, 0 }  ][line width=0.08]  [draw opacity=0] (10.72,-5.15) -- (0,0) -- (10.72,5.15) -- (7.12,0) -- cycle    ;
\draw  [dash pattern={on 4.5pt off 4.5pt}]  (369.5,271.25) -- (130.5,210.25) ;
\draw  [dash pattern={on 4.5pt off 4.5pt}]  (132.5,296.25) -- (369.5,271.25) ;
\draw  [dash pattern={on 4.5pt off 4.5pt}]  (369.5,271.25) -- (603.5,295.25) ;

\draw (124,75.4) node [anchor=north west][inner sep=0.75pt]    {$i_{0}$};
\draw (364,75.4) node [anchor=north west][inner sep=0.75pt]    {$i_{1}$};
\draw (597,74.4) node [anchor=north west][inner sep=0.75pt]    {$i_{2}$};
\draw (131.5,130.9) node [anchor=north west][inner sep=0.75pt]    {$v$};
\draw (372.5,226.9) node [anchor=north west][inner sep=0.75pt]    {$u'_{0}$};
\draw (134.5,304.9) node [anchor=north west][inner sep=0.75pt]    {$u'$};
\draw (605.5,303.9) node [anchor=north west][inner sep=0.75pt]    {$u$};
\draw (132.5,218.9) node [anchor=north west][inner sep=0.75pt]    {$u_{1}$};
\draw (371.5,153.9) node [anchor=north west][inner sep=0.75pt]    {$u_{0}$};
\draw (377,274.65) node [anchor=north west][inner sep=0.75pt]    {$w$};

\end{tikzpicture}
        \caption{$P_0=P_{u}\circ(u',u)\circ P_{u'}^{-1}=vu'_0u'uu_1u_0v$}
    \end{subfigure}
    \begin{subfigure}[b]{0.45\textwidth}
        \centering
        \tikzset{every picture/.style={line width=0.75pt}} 

\begin{tikzpicture}[x=0.75pt,y=0.75pt,yscale=-0.75,xscale=0.55]

\draw  [fill={rgb, 255:red, 245; green, 166; blue, 35 }  ,fill opacity=0.5 ][line width=0.75]  (124,122.25) .. controls (124,119.35) and (126.46,117) .. (129.5,117) .. controls (132.54,117) and (135,119.35) .. (135,122.25) .. controls (135,125.15) and (132.54,127.5) .. (129.5,127.5) .. controls (126.46,127.5) and (124,125.15) .. (124,122.25) -- cycle ;
\draw  [fill={rgb, 255:red, 245; green, 166; blue, 35 }  ,fill opacity=0.5 ][line width=0.75]  (364,145.25) .. controls (364,142.35) and (366.46,140) .. (369.5,140) .. controls (372.54,140) and (375,142.35) .. (375,145.25) .. controls (375,148.15) and (372.54,150.5) .. (369.5,150.5) .. controls (366.46,150.5) and (364,148.15) .. (364,145.25) -- cycle ;
\draw  [fill={rgb, 255:red, 245; green, 166; blue, 35 }  ,fill opacity=0.5 ][line width=0.75]  (365,218.25) .. controls (365,215.35) and (367.46,213) .. (370.5,213) .. controls (373.54,213) and (376,215.35) .. (376,218.25) .. controls (376,221.15) and (373.54,223.5) .. (370.5,223.5) .. controls (367.46,223.5) and (365,221.15) .. (365,218.25) -- cycle ;
\draw  [fill={rgb, 255:red, 245; green, 166; blue, 35 }  ,fill opacity=0.5 ][line width=0.75]  (125,210.25) .. controls (125,207.35) and (127.46,205) .. (130.5,205) .. controls (133.54,205) and (136,207.35) .. (136,210.25) .. controls (136,213.15) and (133.54,215.5) .. (130.5,215.5) .. controls (127.46,215.5) and (125,213.15) .. (125,210.25) -- cycle ;
\draw  [fill={rgb, 255:red, 245; green, 166; blue, 35 }  ,fill opacity=0.5 ][line width=0.75]  (127,296.25) .. controls (127,293.35) and (129.46,291) .. (132.5,291) .. controls (135.54,291) and (138,293.35) .. (138,296.25) .. controls (138,299.15) and (135.54,301.5) .. (132.5,301.5) .. controls (129.46,301.5) and (127,299.15) .. (127,296.25) -- cycle ;
\draw  [fill={rgb, 255:red, 245; green, 166; blue, 35 }  ,fill opacity=0.5 ][line width=0.75]  (598,295.25) .. controls (598,292.35) and (600.46,290) .. (603.5,290) .. controls (606.54,290) and (609,292.35) .. (609,295.25) .. controls (609,298.15) and (606.54,300.5) .. (603.5,300.5) .. controls (600.46,300.5) and (598,298.15) .. (598,295.25) -- cycle ;
\draw  [fill={rgb, 255:red, 155; green, 155; blue, 155 }  ,fill opacity=0.5 ][line width=0.75]  (364,271.25) .. controls (364,268.35) and (366.46,266) .. (369.5,266) .. controls (372.54,266) and (375,268.35) .. (375,271.25) .. controls (375,274.15) and (372.54,276.5) .. (369.5,276.5) .. controls (366.46,276.5) and (364,274.15) .. (364,271.25) -- cycle ;
\draw    (129.5,122.25) -- (370.5,218.25) ;
\draw [shift={(254.65,172.1)}, rotate = 201.72] [fill={rgb, 255:red, 0; green, 0; blue, 0 }  ][line width=0.08]  [draw opacity=0] (10.72,-5.15) -- (0,0) -- (10.72,5.15) -- (7.12,0) -- cycle    ;
\draw    (370.5,218.25) -- (132.5,296.25) ;
\draw [shift={(246.75,258.81)}, rotate = 341.85] [fill={rgb, 255:red, 0; green, 0; blue, 0 }  ][line width=0.08]  [draw opacity=0] (10.72,-5.15) -- (0,0) -- (10.72,5.15) -- (7.12,0) -- cycle    ;
\draw    (132.5,296.25) -- (603.5,295.25) ;
\draw [shift={(373,295.74)}, rotate = 179.88] [fill={rgb, 255:red, 0; green, 0; blue, 0 }  ][line width=0.08]  [draw opacity=0] (10.72,-5.15) -- (0,0) -- (10.72,5.15) -- (7.12,0) -- cycle    ;
\draw    (130.5,210.25) -- (369.5,145.25) ;
\draw [shift={(254.82,176.44)}, rotate = 164.79] [fill={rgb, 255:red, 0; green, 0; blue, 0 }  ][line width=0.08]  [draw opacity=0] (10.72,-5.15) -- (0,0) -- (10.72,5.15) -- (7.12,0) -- cycle    ;
\draw    (369.5,145.25) -- (129.5,122.25) ;
\draw [shift={(244.52,133.27)}, rotate = 5.47] [fill={rgb, 255:red, 0; green, 0; blue, 0 }  ][line width=0.08]  [draw opacity=0] (10.72,-5.15) -- (0,0) -- (10.72,5.15) -- (7.12,0) -- cycle    ;
\draw    (369.5,271.25) -- (130.5,210.25) ;
\draw [shift={(245.16,239.51)}, rotate = 14.32] [fill={rgb, 255:red, 0; green, 0; blue, 0 }  ][line width=0.08]  [draw opacity=0] (10.72,-5.15) -- (0,0) -- (10.72,5.15) -- (7.12,0) -- cycle    ;
\draw  [dash pattern={on 4.5pt off 4.5pt}]  (132.5,296.25) -- (369.5,271.25) ;
\draw    (369.5,271.25) -- (603.5,295.25) ;
\draw [shift={(480.03,282.59)}, rotate = 5.86] [fill={rgb, 255:red, 0; green, 0; blue, 0 }  ][line width=0.08]  [draw opacity=0] (10.72,-5.15) -- (0,0) -- (10.72,5.15) -- (7.12,0) -- cycle    ;

\draw (124,75.4) node [anchor=north west][inner sep=0.75pt]    {$i_{0}$};
\draw (364,75.4) node [anchor=north west][inner sep=0.75pt]    {$i_{1}$};
\draw (597,74.4) node [anchor=north west][inner sep=0.75pt]    {$i_{2}$};
\draw (131.5,130.9) node [anchor=north west][inner sep=0.75pt]    {$v$};
\draw (372.5,226.9) node [anchor=north west][inner sep=0.75pt]    {$u'_{0}$};
\draw (134.5,304.9) node [anchor=north west][inner sep=0.75pt]    {$u'$};
\draw (605.5,303.9) node [anchor=north west][inner sep=0.75pt]    {$u$};
\draw (132.5,218.9) node [anchor=north west][inner sep=0.75pt]    {$u_{1}$};
\draw (371.5,153.9) node [anchor=north west][inner sep=0.75pt]    {$u_{0}$};
\draw (377,274.65) node [anchor=north west][inner sep=0.75pt]    {$w$};

\end{tikzpicture}
        \caption{$P_1=vu'_0u'uwu_1u_0v$}
    \end{subfigure}
    \begin{subfigure}[b]{0.45\textwidth}
        \centering
        \tikzset{every picture/.style={line width=0.75pt}} 

\begin{tikzpicture}[x=0.75pt,y=0.75pt,yscale=-0.75,xscale=0.55]

\draw  [fill={rgb, 255:red, 245; green, 166; blue, 35 }  ,fill opacity=0.5 ][line width=0.75]  (124,122.25) .. controls (124,119.35) and (126.46,117) .. (129.5,117) .. controls (132.54,117) and (135,119.35) .. (135,122.25) .. controls (135,125.15) and (132.54,127.5) .. (129.5,127.5) .. controls (126.46,127.5) and (124,125.15) .. (124,122.25) -- cycle ;
\draw  [fill={rgb, 255:red, 245; green, 166; blue, 35 }  ,fill opacity=0.5 ][line width=0.75]  (364,145.25) .. controls (364,142.35) and (366.46,140) .. (369.5,140) .. controls (372.54,140) and (375,142.35) .. (375,145.25) .. controls (375,148.15) and (372.54,150.5) .. (369.5,150.5) .. controls (366.46,150.5) and (364,148.15) .. (364,145.25) -- cycle ;
\draw  [fill={rgb, 255:red, 245; green, 166; blue, 35 }  ,fill opacity=0.5 ][line width=0.75]  (365,218.25) .. controls (365,215.35) and (367.46,213) .. (370.5,213) .. controls (373.54,213) and (376,215.35) .. (376,218.25) .. controls (376,221.15) and (373.54,223.5) .. (370.5,223.5) .. controls (367.46,223.5) and (365,221.15) .. (365,218.25) -- cycle ;
\draw  [fill={rgb, 255:red, 245; green, 166; blue, 35 }  ,fill opacity=0.5 ][line width=0.75]  (125,210.25) .. controls (125,207.35) and (127.46,205) .. (130.5,205) .. controls (133.54,205) and (136,207.35) .. (136,210.25) .. controls (136,213.15) and (133.54,215.5) .. (130.5,215.5) .. controls (127.46,215.5) and (125,213.15) .. (125,210.25) -- cycle ;
\draw  [fill={rgb, 255:red, 245; green, 166; blue, 35 }  ,fill opacity=0.5 ][line width=0.75]  (127,296.25) .. controls (127,293.35) and (129.46,291) .. (132.5,291) .. controls (135.54,291) and (138,293.35) .. (138,296.25) .. controls (138,299.15) and (135.54,301.5) .. (132.5,301.5) .. controls (129.46,301.5) and (127,299.15) .. (127,296.25) -- cycle ;
\draw  [fill={rgb, 255:red, 245; green, 166; blue, 35 }  ,fill opacity=0.5 ][line width=0.75]  (598,295.25) .. controls (598,292.35) and (600.46,290) .. (603.5,290) .. controls (606.54,290) and (609,292.35) .. (609,295.25) .. controls (609,298.15) and (606.54,300.5) .. (603.5,300.5) .. controls (600.46,300.5) and (598,298.15) .. (598,295.25) -- cycle ;
\draw  [fill={rgb, 255:red, 155; green, 155; blue, 155 }  ,fill opacity=0.5 ][line width=0.75]  (364,271.25) .. controls (364,268.35) and (366.46,266) .. (369.5,266) .. controls (372.54,266) and (375,268.35) .. (375,271.25) .. controls (375,274.15) and (372.54,276.5) .. (369.5,276.5) .. controls (366.46,276.5) and (364,274.15) .. (364,271.25) -- cycle ;
\draw    (129.5,122.25) -- (370.5,218.25) ;
\draw [shift={(254.65,172.1)}, rotate = 201.72] [fill={rgb, 255:red, 0; green, 0; blue, 0 }  ][line width=0.08]  [draw opacity=0] (10.72,-5.15) -- (0,0) -- (10.72,5.15) -- (7.12,0) -- cycle    ;
\draw    (370.5,218.25) -- (132.5,296.25) ;
\draw [shift={(246.75,258.81)}, rotate = 341.85] [fill={rgb, 255:red, 0; green, 0; blue, 0 }  ][line width=0.08]  [draw opacity=0] (10.72,-5.15) -- (0,0) -- (10.72,5.15) -- (7.12,0) -- cycle    ;
\draw    (130.5,210.25) -- (369.5,145.25) ;
\draw [shift={(254.82,176.44)}, rotate = 164.79] [fill={rgb, 255:red, 0; green, 0; blue, 0 }  ][line width=0.08]  [draw opacity=0] (10.72,-5.15) -- (0,0) -- (10.72,5.15) -- (7.12,0) -- cycle    ;
\draw    (369.5,145.25) -- (129.5,122.25) ;
\draw [shift={(244.52,133.27)}, rotate = 5.47] [fill={rgb, 255:red, 0; green, 0; blue, 0 }  ][line width=0.08]  [draw opacity=0] (10.72,-5.15) -- (0,0) -- (10.72,5.15) -- (7.12,0) -- cycle    ;
\draw    (369.5,271.25) -- (130.5,210.25) ;
\draw [shift={(245.16,239.51)}, rotate = 14.32] [fill={rgb, 255:red, 0; green, 0; blue, 0 }  ][line width=0.08]  [draw opacity=0] (10.72,-5.15) -- (0,0) -- (10.72,5.15) -- (7.12,0) -- cycle    ;
\draw    (132.5,296.25) -- (369.5,271.25) ;
\draw [shift={(255.97,283.23)}, rotate = 173.98] [fill={rgb, 255:red, 0; green, 0; blue, 0 }  ][line width=0.08]  [draw opacity=0] (10.72,-5.15) -- (0,0) -- (10.72,5.15) -- (7.12,0) -- cycle    ;

\draw (124,75.4) node [anchor=north west][inner sep=0.75pt]    {$i_{0}$};
\draw (364,75.4) node [anchor=north west][inner sep=0.75pt]    {$i_{1}$};
\draw (597,74.4) node [anchor=north west][inner sep=0.75pt]    {$i_{2}$};
\draw (131.5,130.9) node [anchor=north west][inner sep=0.75pt]    {$v$};
\draw (372.5,226.9) node [anchor=north west][inner sep=0.75pt]    {$u'_{0}$};
\draw (134.5,304.9) node [anchor=north west][inner sep=0.75pt]    {$u'$};
\draw (605.5,303.9) node [anchor=north west][inner sep=0.75pt]    {$u$};
\draw (132.5,218.9) node [anchor=north west][inner sep=0.75pt]    {$u_{1}$};
\draw (371.5,153.9) node [anchor=north west][inner sep=0.75pt]    {$u_{0}$};
\draw (377,274.65) node [anchor=north west][inner sep=0.75pt]    {$w$};

\end{tikzpicture}
        \caption{$P_2=vu'_0u'wu_1u_0v$}
    \end{subfigure}
    \caption{contraction $T_{u'u}$ when $u\in \buildings^{\set{i_2}}$ and $u'\in \buildings_s^{\set{i_0}}$ (reduction to $u\in \buildings^{\set{i_1}}$ and $u'\in \buildings_s^{\set{i_0}}$)}
\end{figure}

\begin{figure}
    \centering
    \begin{subfigure}[b]{0.45\textwidth}
        \centering
        \tikzset{every picture/.style={line width=0.75pt}} 

\begin{tikzpicture}[x=0.75pt,y=0.75pt,yscale=-0.75,xscale=0.55]

\draw  [fill={rgb, 255:red, 245; green, 166; blue, 35 }  ,fill opacity=0.5 ][line width=0.75]  (124,122.25) .. controls (124,119.35) and (126.46,117) .. (129.5,117) .. controls (132.54,117) and (135,119.35) .. (135,122.25) .. controls (135,125.15) and (132.54,127.5) .. (129.5,127.5) .. controls (126.46,127.5) and (124,125.15) .. (124,122.25) -- cycle ;
\draw  [fill={rgb, 255:red, 245; green, 166; blue, 35 }  ,fill opacity=0.5 ][line width=0.75]  (364,145.25) .. controls (364,142.35) and (366.46,140) .. (369.5,140) .. controls (372.54,140) and (375,142.35) .. (375,145.25) .. controls (375,148.15) and (372.54,150.5) .. (369.5,150.5) .. controls (366.46,150.5) and (364,148.15) .. (364,145.25) -- cycle ;
\draw  [fill={rgb, 255:red, 245; green, 166; blue, 35 }  ,fill opacity=0.5 ][line width=0.75]  (365,218.25) .. controls (365,215.35) and (367.46,213) .. (370.5,213) .. controls (373.54,213) and (376,215.35) .. (376,218.25) .. controls (376,221.15) and (373.54,223.5) .. (370.5,223.5) .. controls (367.46,223.5) and (365,221.15) .. (365,218.25) -- cycle ;
\draw  [fill={rgb, 255:red, 245; green, 166; blue, 35 }  ,fill opacity=0.5 ][line width=0.75]  (125,210.25) .. controls (125,207.35) and (127.46,205) .. (130.5,205) .. controls (133.54,205) and (136,207.35) .. (136,210.25) .. controls (136,213.15) and (133.54,215.5) .. (130.5,215.5) .. controls (127.46,215.5) and (125,213.15) .. (125,210.25) -- cycle ;
\draw  [fill={rgb, 255:red, 245; green, 166; blue, 35 }  ,fill opacity=0.5 ][line width=0.75]  (127,296.25) .. controls (127,293.35) and (129.46,291) .. (132.5,291) .. controls (135.54,291) and (138,293.35) .. (138,296.25) .. controls (138,299.15) and (135.54,301.5) .. (132.5,301.5) .. controls (129.46,301.5) and (127,299.15) .. (127,296.25) -- cycle ;
\draw  [fill={rgb, 255:red, 245; green, 166; blue, 35 }  ,fill opacity=0.5 ][line width=0.75]  (598,295.25) .. controls (598,292.35) and (600.46,290) .. (603.5,290) .. controls (606.54,290) and (609,292.35) .. (609,295.25) .. controls (609,298.15) and (606.54,300.5) .. (603.5,300.5) .. controls (600.46,300.5) and (598,298.15) .. (598,295.25) -- cycle ;
\draw    (129.5,122.25) -- (370.5,218.25) ;
\draw [shift={(254.65,172.1)}, rotate = 201.72] [fill={rgb, 255:red, 0; green, 0; blue, 0 }  ][line width=0.08]  [draw opacity=0] (10.72,-5.15) -- (0,0) -- (10.72,5.15) -- (7.12,0) -- cycle    ;
\draw    (370.5,218.25) -- (132.5,296.25) ;
\draw [shift={(246.75,258.81)}, rotate = 341.85] [fill={rgb, 255:red, 0; green, 0; blue, 0 }  ][line width=0.08]  [draw opacity=0] (10.72,-5.15) -- (0,0) -- (10.72,5.15) -- (7.12,0) -- cycle    ;
\draw    (603.5,295.25) -- (130.5,210.25) ;
\draw [shift={(362.08,251.87)}, rotate = 10.19] [fill={rgb, 255:red, 0; green, 0; blue, 0 }  ][line width=0.08]  [draw opacity=0] (10.72,-5.15) -- (0,0) -- (10.72,5.15) -- (7.12,0) -- cycle    ;
\draw    (130.5,210.25) -- (369.5,145.25) ;
\draw [shift={(254.82,176.44)}, rotate = 164.79] [fill={rgb, 255:red, 0; green, 0; blue, 0 }  ][line width=0.08]  [draw opacity=0] (10.72,-5.15) -- (0,0) -- (10.72,5.15) -- (7.12,0) -- cycle    ;
\draw    (369.5,145.25) -- (129.5,122.25) ;
\draw [shift={(244.52,133.27)}, rotate = 5.47] [fill={rgb, 255:red, 0; green, 0; blue, 0 }  ][line width=0.08]  [draw opacity=0] (10.72,-5.15) -- (0,0) -- (10.72,5.15) -- (7.12,0) -- cycle    ;
\draw  [fill={rgb, 255:red, 245; green, 166; blue, 35 }  ,fill opacity=0.5 ][line width=0.75]  (365,296.25) .. controls (365,293.35) and (367.46,291) .. (370.5,291) .. controls (373.54,291) and (376,293.35) .. (376,296.25) .. controls (376,299.15) and (373.54,301.5) .. (370.5,301.5) .. controls (367.46,301.5) and (365,299.15) .. (365,296.25) -- cycle ;
\draw    (132.5,296.25) -- (370.5,296.25) ;
\draw [shift={(256.5,296.25)}, rotate = 180] [fill={rgb, 255:red, 0; green, 0; blue, 0 }  ][line width=0.08]  [draw opacity=0] (10.72,-5.15) -- (0,0) -- (10.72,5.15) -- (7.12,0) -- cycle    ;
\draw    (370.5,296.25) -- (603.5,295.25) ;
\draw [shift={(492,295.73)}, rotate = 179.75] [fill={rgb, 255:red, 0; green, 0; blue, 0 }  ][line width=0.08]  [draw opacity=0] (10.72,-5.15) -- (0,0) -- (10.72,5.15) -- (7.12,0) -- cycle    ;

\draw (124,75.4) node [anchor=north west][inner sep=0.75pt]    {$i_{0}$};
\draw (364,75.4) node [anchor=north west][inner sep=0.75pt]    {$i_{1}$};
\draw (597,74.4) node [anchor=north west][inner sep=0.75pt]    {$i_{2}$};
\draw (131.5,130.9) node [anchor=north west][inner sep=0.75pt]    {$v$};
\draw (372.5,226.9) node [anchor=north west][inner sep=0.75pt]    {$u'_{0}$};
\draw (134.5,304.9) node [anchor=north west][inner sep=0.75pt]    {$u '_{1}$};
\draw (605.5,303.9) node [anchor=north west][inner sep=0.75pt]    {$u$};
\draw (132.5,218.9) node [anchor=north west][inner sep=0.75pt]    {$u_{1}$};
\draw (371.5,153.9) node [anchor=north west][inner sep=0.75pt]    {$u_{0}$};
\draw (372.5,304.9) node [anchor=north west][inner sep=0.75pt]    {$u'$};

\end{tikzpicture}
        \caption{$P_0=P_{u}\circ(u',u)\circ P_{u'}^{-1}=vu'_0u'_1u'uu_1u_0v$}
    \end{subfigure}
    \begin{subfigure}[b]{0.45\textwidth}
        \centering
        \tikzset{every picture/.style={line width=0.75pt}} 

\begin{tikzpicture}[x=0.75pt,y=0.75pt,yscale=-0.75,xscale=0.55]

\draw  [fill={rgb, 255:red, 245; green, 166; blue, 35 }  ,fill opacity=0.5 ][line width=0.75]  (124,122.25) .. controls (124,119.35) and (126.46,117) .. (129.5,117) .. controls (132.54,117) and (135,119.35) .. (135,122.25) .. controls (135,125.15) and (132.54,127.5) .. (129.5,127.5) .. controls (126.46,127.5) and (124,125.15) .. (124,122.25) -- cycle ;
\draw  [fill={rgb, 255:red, 245; green, 166; blue, 35 }  ,fill opacity=0.5 ][line width=0.75]  (364,145.25) .. controls (364,142.35) and (366.46,140) .. (369.5,140) .. controls (372.54,140) and (375,142.35) .. (375,145.25) .. controls (375,148.15) and (372.54,150.5) .. (369.5,150.5) .. controls (366.46,150.5) and (364,148.15) .. (364,145.25) -- cycle ;
\draw  [fill={rgb, 255:red, 245; green, 166; blue, 35 }  ,fill opacity=0.5 ][line width=0.75]  (365,218.25) .. controls (365,215.35) and (367.46,213) .. (370.5,213) .. controls (373.54,213) and (376,215.35) .. (376,218.25) .. controls (376,221.15) and (373.54,223.5) .. (370.5,223.5) .. controls (367.46,223.5) and (365,221.15) .. (365,218.25) -- cycle ;
\draw  [fill={rgb, 255:red, 245; green, 166; blue, 35 }  ,fill opacity=0.5 ][line width=0.75]  (125,210.25) .. controls (125,207.35) and (127.46,205) .. (130.5,205) .. controls (133.54,205) and (136,207.35) .. (136,210.25) .. controls (136,213.15) and (133.54,215.5) .. (130.5,215.5) .. controls (127.46,215.5) and (125,213.15) .. (125,210.25) -- cycle ;
\draw  [fill={rgb, 255:red, 245; green, 166; blue, 35 }  ,fill opacity=0.5 ][line width=0.75]  (127,296.25) .. controls (127,293.35) and (129.46,291) .. (132.5,291) .. controls (135.54,291) and (138,293.35) .. (138,296.25) .. controls (138,299.15) and (135.54,301.5) .. (132.5,301.5) .. controls (129.46,301.5) and (127,299.15) .. (127,296.25) -- cycle ;
\draw  [fill={rgb, 255:red, 245; green, 166; blue, 35 }  ,fill opacity=0.5 ][line width=0.75]  (598,295.25) .. controls (598,292.35) and (600.46,290) .. (603.5,290) .. controls (606.54,290) and (609,292.35) .. (609,295.25) .. controls (609,298.15) and (606.54,300.5) .. (603.5,300.5) .. controls (600.46,300.5) and (598,298.15) .. (598,295.25) -- cycle ;
\draw    (129.5,122.25) -- (370.5,218.25) ;
\draw [shift={(254.65,172.1)}, rotate = 201.72] [fill={rgb, 255:red, 0; green, 0; blue, 0 }  ][line width=0.08]  [draw opacity=0] (10.72,-5.15) -- (0,0) -- (10.72,5.15) -- (7.12,0) -- cycle    ;
\draw    (370.5,218.25) -- (132.5,296.25) ;
\draw [shift={(246.75,258.81)}, rotate = 341.85] [fill={rgb, 255:red, 0; green, 0; blue, 0 }  ][line width=0.08]  [draw opacity=0] (10.72,-5.15) -- (0,0) -- (10.72,5.15) -- (7.12,0) -- cycle    ;
\draw    (603.5,295.25) -- (130.5,210.25) ;
\draw [shift={(362.08,251.87)}, rotate = 10.19] [fill={rgb, 255:red, 0; green, 0; blue, 0 }  ][line width=0.08]  [draw opacity=0] (10.72,-5.15) -- (0,0) -- (10.72,5.15) -- (7.12,0) -- cycle    ;
\draw    (130.5,210.25) -- (369.5,145.25) ;
\draw [shift={(254.82,176.44)}, rotate = 164.79] [fill={rgb, 255:red, 0; green, 0; blue, 0 }  ][line width=0.08]  [draw opacity=0] (10.72,-5.15) -- (0,0) -- (10.72,5.15) -- (7.12,0) -- cycle    ;
\draw    (369.5,145.25) -- (129.5,122.25) ;
\draw [shift={(244.52,133.27)}, rotate = 5.47] [fill={rgb, 255:red, 0; green, 0; blue, 0 }  ][line width=0.08]  [draw opacity=0] (10.72,-5.15) -- (0,0) -- (10.72,5.15) -- (7.12,0) -- cycle    ;
\draw  [fill={rgb, 255:red, 245; green, 166; blue, 35 }  ,fill opacity=0.5 ][line width=0.75]  (365,296.25) .. controls (365,293.35) and (367.46,291) .. (370.5,291) .. controls (373.54,291) and (376,293.35) .. (376,296.25) .. controls (376,299.15) and (373.54,301.5) .. (370.5,301.5) .. controls (367.46,301.5) and (365,299.15) .. (365,296.25) -- cycle ;
\draw    (132.5,296.25) .. controls (186,347.5) and (543,339.5) .. (603.5,295.25) ;
\draw [shift={(374.69,331.43)}, rotate = 179.11] [fill={rgb, 255:red, 0; green, 0; blue, 0 }  ][line width=0.08]  [draw opacity=0] (10.72,-5.15) -- (0,0) -- (10.72,5.15) -- (7.12,0) -- cycle    ;

\draw (124,75.4) node [anchor=north west][inner sep=0.75pt]    {$i_{0}$};
\draw (364,75.4) node [anchor=north west][inner sep=0.75pt]    {$i_{1}$};
\draw (597,74.4) node [anchor=north west][inner sep=0.75pt]    {$i_{2}$};
\draw (131.5,130.9) node [anchor=north west][inner sep=0.75pt]    {$v$};
\draw (372.5,226.9) node [anchor=north west][inner sep=0.75pt]    {$u'_{0}$};
\draw (134.5,304.9) node [anchor=north west][inner sep=0.75pt]    {$u '_{1}$};
\draw (605.5,303.9) node [anchor=north west][inner sep=0.75pt]    {$u$};
\draw (132.5,218.9) node [anchor=north west][inner sep=0.75pt]    {$u_{1}$};
\draw (371.5,153.9) node [anchor=north west][inner sep=0.75pt]    {$u_{0}$};
\draw (372.5,304.9) node [anchor=north west][inner sep=0.75pt]    {$u'$};

\end{tikzpicture}
        \caption{$P_1=vu'_0u'_1uwu_1u_0v$}
    \end{subfigure}
    \caption{contraction $T_{u'u}$ when $u\in \buildings^{\set{i_2}}$ and $u'\in \buildings_s^{\set{i_1}}$ (reduction to $u\in \buildings^{\set{i_2}}$ and $u'\in \buildings_s^{\set{i_0}}$)}
\end{figure}

To take advantage of \Cref{lem:color-restriction}, we start by showing any projection of the flags complex onto a ``well-separated'' set of three colors has constant coboundary expansion by directly constructing a family of cones.

\begin{lemma}
\label{lem:three-projection-coboundary}
Fix any integer $d\geq 3$, any prime power $q$, and $\Gamma$ any group.
Let $I=\set{i_0<i_1<i_2}\subseteq [d-1]$. 
Let $s \in \buildings_{q}^{d,[d-1]\setminus I}$.
Assume that 
\begin{enumerate}
    \item $2i_0\leq i_1$ or there is a subspace $V\in s$ with $i_0<\col(V)<i_1$, and
    \item either $2i_1-i_0\leq i_2$, or $2i_1-\dim(V)\leq i_2$ where $V\in s$ is a subspace with $i_0<\col(V)<i_1$, or there is a subspace $W\in s$ with $i_1<\col(W)<i_2$.
\end{enumerate}
Then, the complex $\buildings_{q,s}^{d,I}$ has 1-coboundary expansion at least
\begin{align}
    h^1\left(\buildings_{q,s}^{d,I}, \Gamma\right)\geq \frac{1}{13}.
\end{align}
\end{lemma}
\begin{proof}
Let $\buildings=\buildings^d_q$. It is elementary to see the group of automorphisms $\mathrm{Aut}(\buildings_s^I)$ acts transitively on the triangles of $\buildings_s^I$ so by \Cref{lem:cone-argument} it is sufficient to construct a cone of diameter 13.\footnote{Namely taking an appropriate basis for the vector spaces in any two such flags, we can transform one to the other with the corresponding basis exchange matrix.}

Fix an arbitrary $v\in \buildings_s^{\set{i_0}}$.
For any vertex $u\in \buildings_s^{\set{i_0}}\setminus \set{v}$, we take $P_u=v u_0 u$, where $u_0$ is an arbitrary vertex in $\buildings_s^{\set{i_1}}$ such that $u_0\supseteq v+u$.
For any vertex $u\in \buildings_s^{\set{i_1,i_2}}$, take $P_u=P_{u_1}\circ (u_1,u)$, where $u_1$ is an arbitrary vertex in $\buildings_s^{\set{i_0}}\setminus \set{v}$ such that $u\supseteq u_1$.

We now describe the contraction $T_{u'u}$ for each edge $u'u$.
First, suppose $u\in \buildings^{\set{i_1}}$ and $u'\in \buildings_s^{\set{i_0}}$.
Then the first walk in the contraction is given by 
\[
P_0=P_{u}\circ(u',u)\circ P_{u'}^{-1}=vu'_0u'uu_1u_0v.
\]
As either $u'_0\cap u\supseteq u'$ and $i_2\geq 2i_1-i_0$, or $u'_0\cap u\supseteq V$ and $i_2\geq 2i_1-\dim(V)$, or $W\supseteq u'_0, u$ there exists a vertex $x\in \buildings_s^{\set{i_2}}$ such that $x\supseteq u+u'_0$.
Now we get the following set of $\sim_1$ relations.
\begin{enumerate}
    \item $P_0\sim_1 P_1=vu'_0u'xuu_1u_0v$ using the loop $uxu$ and then the triangle $u'xu$.
    \item $P_1\sim_1 P_2=vu'_0xuu_1u_0v$ using the triangle $u'_0u'x$.
    \item $P_2\sim_1 P_3=vxuu_1u_0v$ using the triangle $vu'_0x$.
    \item $P_3\sim_1 P_4=vxu_1u_0v$ using the triangle $xuu_1$.
\end{enumerate}
Observe that there exists $y\in \buildings_s^{\set{i_1}}$ such that
$v+u_1\subseteq y \subseteq x$ as $2i_0\leq i_1$ or $v+u_1\subseteq V$.
Also there is a $z\in \buildings_s^{\set{i_2}}$ such that $z\supseteq y+u_0$ since $y\cap u_0 \supseteq v$ or $y+u_0 \subseteq W$.
This gives us another sequence of $\sim_1$ relations.
\begin{enumerate}
    \setcounter{enumi}{4}
    \item $P_4\sim_1 P_5=vxu_1u_0zv$ using the triangle $u_0zv$.
    \item $P_5\sim_1 P_6=vxu_1zv$ using the triangle $u_1u_0z$.
    \item $P_6\sim_1 P_7=vxu_1yzv$ using the triangle $u_1yz$.
    \item $P_7\sim_1 P_8=vyxu_1yzv$ using the triangle $vyx$.
    \item $P_8\sim_1 P_9=vyu_1yzv$ using the triangle $yxu_1$.
    \item $P_9\sim_1 P_{10}=vyu_1yv$ using the triangle $yzv$.
\end{enumerate}
This completes the contraction $T_{u'u}$ as $P_{10}$ is just two loops.
Now, we reduce the task of finding contraction in the other two cases to the case above.
Suppose $u\in \buildings_s^{\set{i_2}}$ and $u'\in \buildings_s^{\set{i_0}}$.
Then, we start with $P_0=vu'_0u'uu_1u_0v$.
Note that there exists a $w \in \buildings_s^{\set{i_1}}$ such that $u\supseteq w \supseteq u'+u_1$.
Observe that $P_0\sim_1 P_1=vu'_0u'uwu_1u_0v$ using the triangle $uwu_1$, and $P_1\sim_1 P_2=vu'_0u'wu_1u_0v$ using the triangle $u'wu$ and then the loop $wuw$.
This is precisely the walk at the start of the first case (except for the labeling of vertices).
Finally, assume $u\in \buildings^{\set{i_2}}$ and $u'\in \buildings_s^{\set{i_1}}$.
Then, we start with $P_0=vu'_0u'_1u'uu_1u_0v$.
We reduce it to the second case by observing $P_0\sim_1 P_1=vu'_0u'_1uu_1u_0v$ using the triangle $u'_1u'u$.
\end{proof}

While we cannot use \Cref{lem:three-projection-coboundary} directly to show the projected flags complex is a coboundary expander (since \emph{no} 3-color projection there-in is well-separated), we \emph{can} use it to show the complex's \emph{links} are good coboundary expanders. We may then combine this with the local-to-global lemma \Cref{lem:localcoboundary-to-coboundary} to complete the proof.

\begin{lemma}
\label{lem:coboundary-link-spherical}
Fix integers $d> k \geq 2$, and a prime power $q\geq 25$.
Let 
\[
S_{k}=\set{d-k, \dots, d-1, d, d+1, \dots, d+k}
\]
and let $s\in \buildings^{2d,S_{k}}_{q}$ be a non-empty face in the projected building $\buildings^{2d,S_{k}}_{q}$ with co-dimension at least 3.
Then, the link $\buildings^{2d,S_{k}}_{q,s}$ has 1-coboundary expansion 
\[
h^1(\buildings^{2d,S_{k}}_{q,s},\Gamma)=\Omega(1).
\]
\end{lemma}
\begin{proof}

Let $c^* = \col (s) \subseteq S_{k}$ and write $c^* = \set{j_1 , \dots , j_b }$.
Let $S'=S_{k}\setminus \col(s)$.
Now, we want to show that for a constant fraction of $I=\set{i_0<i_1<i_2}\in  \binom{S'}{3}$, $\buildings_{q,s}^{2d,I}$ has 1-coboundary expansion at least $\beta$ for some constant $\beta>0$.
The proof is then complete using \Cref{lem:color-restriction}, as by \Cref{thm:spectral-projected} we already know that $\buildings^{2d,I'}_{q,s}$ is $\frac{1}{2}$-local spectral expander for any $I'$ for sufficiently large $q$.
Pick a uniformly random set $I=\set{i_0<i_1<i_2}\in  \binom{S'}{3}$.

Define $b + 1$ \emph{bins} \[ B_i = \set{j \in S_{k} | j_i < j < j_{i+1} }\] for $0 < i < b$
and also $B_0 = \set{j \in S_{k} |j < j_1 }$ and $B_b = \set{j \in S_{k} |j > j_b }$. We divide the proof into two cases based on the size of the $s$

\paragraph{Case 1: Small $s$.}
Assume $|s|\leq 2k+1-80$. In this case, we may assume there is a bin  $B_{i}$ with $|B_{i}|>\frac{2k+1-|s|}{4} \geq 20$, as otherwise a union bound implies the probability our three random indices lie in separate bins is at least
\[
1-\frac{(2k+1-|s|)/4}{2k+1-|s|}-\frac{2(2k+1-|s|)/4}{2k+1-|s|}\geq \frac{1}{4},
\]
so $h^1(\buildings^{2d,I}_{q,s},\Gamma)\geq \frac{1}{13}$ by \Cref{lem:three-projection-coboundary} and $h^1(\buildings^{2d,S_{k}}_{q,s},\Gamma)\geq \frac{1}{1872}$.
We may also assume $j_i\ne 0$; otherwise, observe that any projected building is isomorphic to the projected building formed by replacing each subspace with its orthogonal complement, where $j_i=b$ instead.
Since $|B_{i}|>\frac{2k+1-|s|}{4}$, with probability at least
\[
\frac{|B_{i}|/20}{2k+1-|s|} \cdot \frac{|B_{i}|/5}{2k+1-|s|} \cdot \frac{|B_{i}|/4}{2k+1-|s|}\geq\frac{1}{25600}
\]
$i_0$ is in the second twentieth of the indices $B_{i}$, $i_1$ is in the second fifth of the indices $B_{i}$,
and $i_2$ is in the final quarter of the indices $B_{i}$.
In this case, the link is isomorphic to $\buildings_q^{j_{i+1}-j_i,\{i_0-j_i,i_1-j_i,i_2-j_i\}}$ where $\{i_0-j_i,i_1-j_i,i_2-j_i\}$ satisfy the spreadness conditions of \Cref{lem:three-projection-coboundary}. Thus $h^1(\buildings^{2d,I}_{q,s},\Gamma)\geq \frac{1}{13}$ with probability at least $\frac{1}{25600}$ and $h^1(\buildings^{2d,S_{k}}_{q,s},\Gamma)\geq 10^{-8}$ as desired.

\paragraph{Case 2: Large $s$.}
When $|s|>2k+1-80$, it suffices to show there exists a single choice of $I$ with constant coboundary expansion, since each valid $I$ appears with probability at least $\frac{1}{\binom{2k+1-|s|}{3}}=\Omega(1)$.
We now show such a choice of $I$ always exists. We break into cases based on maximum bin size
\begin{enumerate}
    \item Suppose we have a bin $B_{i}$ with at least three colors. As before, by taking orthogonal complements if necessary, we may assume $i \neq 0$ (similarly for the cases below). We may then take $I=\set{j_i+1,j_i+2,j_i+3}$.
    As $\buildings^{2d,I}_{q,s}$ isomorphic to $\buildings_q^{j_{i+1}-j_i,\{1,2,3\}}$, \Cref{lem:three-projection-coboundary} implies $h^1(\buildings^{2d,I}_{q,s},\Gamma)\geq \frac{1}{13}$.
    \item Suppose we have a bin $B_{i}$ with two colors and there is a non-empty bin $B_{i'}$ where $i'>i>0$.
    Then, taking $I=\set{j_i+1,j_i+2,j_{i'}+1}$ we get that $\buildings^{2d,I}_{q,s}$ is isomorphic to $\buildings_{q, s'}^{j_{i'+1}-j_i,\{1,2,j_{i'}-j_i+1\}}$ for the appropriate $s' \subset s$ between these indices. In particular, $s'$ contains a subspace of dimension $j_{i'}-j_i$, so $\buildings^{2d,I}_{q,s}$ has constant coboundary expansion by \Cref{lem:three-projection-coboundary}. If we instead have $B_{i}$ with one color and a bin $B_{i'}$ with two colors where $i'>i$, the proof is analogous.

    \item Finally, we may have three non-empty bins $B_{i}, B_{i'}, B_{i''}$, where $i''>i'>i$. Taking $I=\set{j_i+1,j_{i'}+1,j_{i''}+1}$ we get that $\buildings^{2d,I}_{q,s}$ is isomorphic to $\buildings_{q, s'}^{j_{i''+1}-j_i,\{1,j_{i'}-j_i+1,j_{i''}-j_i+1\}}$ where $s' \subset s$ contains subspaces of dimensions $j_{i'}-j_{i}$ and $j_{i''}-j_{i}$, and constant coboundary expansion again follows from \Cref{lem:three-projection-coboundary}.
\end{enumerate}

\end{proof}

The proof of \Cref{thm:coboundary-spherical} now follows directly from \Cref{lem:localcoboundary-to-coboundary}: if all links of a local spectral expander with vanishing cohomology are coboundary expanders, then so is the complex itself.
\begin{proof}[Proof of \autoref{thm:coboundary-spherical}]
By \Cref{lem:coboundary-link-spherical}, it is enough to handle just the case of the full complex $h^{1}(\buildings^{d,S_{k}}_q, \Gamma)$. Toward this end, let $q_0$ be large enough so that we have $\frac{2e}{\sqrt{q_0}-1}\leq \frac{\beta^2}{16}$.
Let $V, W\in \buildings^{2d,S_{k}}_{q}(1)$ be any edge.
By \Cref{thm:spectral-projected}, we get that $\buildings^{d,S_{k}}_{q,\set{W,V}}$ is a $1/2$-local spectral expander and
by \Cref{lem:coboundary-link-spherical} we have $h^{1}(\buildings^{d,S_{k}}_{q,V}, \Gamma) = \Omega(1)$.
So, there is a constant $\beta>0$ such that $h^0(\buildings^{d,S_{k}}_{q,\set{W,V}}, \Gamma) \geq \beta$ and $h^{1}(\buildings^{d,S_{k}}_{q,V}, \Gamma) \geq \beta$.
As $\frac{2e}{\sqrt{q_0}-1}\leq \frac{\beta^2}{16}$ by \Cref{lem:localcoboundary-to-coboundary} we get $h^{1}(\buildings^{d,S_{k}}_q, \Gamma)\geq \beta^2/16$.
\end{proof}

\section{The Faces Complexes}
In this section, we analyze the expansion properties of the faces complexes of the projected flags complex.
\subsection{Local Spectral Expansion}
We first bound the local spectral expansion of the faces complexes.
The following lemma, which relates the local spectral expansion of the faces complex on disjoint colors to the local spectral expansion of the partite base complex, is standard but we include it for completeness.
\begin{lemma}
\label{lem:spectral-projected-faces}
Let $X$ be a $k$-partite $\lambda$-product.
Let $2r+2\leq k$.
Let $J\subseteq \binom{[k]}{r+1}$ be a subset of disjoint colors.
Then $\mathrm F^JX$ is a $\frac{(r+1)\lambda}{1-(k-r)\lambda}$-one-sided local spectral expander.
\end{lemma}

\begin{proof}
Let $s\in \mathrm F^JX$ with co-dimension exactly 2.
Observe that the random walk matrix of $(F^JX)_s$ is
\[
\mathsf W_s= \begin{bmatrix}
0 & \mathsf S_{c_1,c_2}\\
\mathsf S_{c_2,c_1} & 0
\end{bmatrix},
\] where ${c_1,c_2}\sqcup \col(s)=J$. 
By~\Cref{lem:color-swap-expansion}, we get that $\lambda(W_s)\leq (r+1)\cdot\lambda$. Similarly, links of co-dimension $i$ are $i$-partite graphs where each bi-partite component is a corresponding swap walk in the original complex as above, and are therefore connected. The Trickle-Down Theorem (\Cref{thm:trickle-down}) then gives the desired statement.
\end{proof}

Now, the following lemma is an easy consequence of \Cref{thm:spectral-projected} and \Cref{lem:spectral-projected-faces}.

\begin{lemma}
\label{lem:spectral-faces-complex}
Fix any $d\geq 3$, any prime power $q$, and a set $S\subseteq [d-1]$ of size at least 2.
Let $2r+2\leq |S|$ and $J\subseteq \binom{S}{r+1}$ be disjoint color sets.
Then $\F^J (\buildings^{d,S}_q)$ is a $\frac{2(r+1)}{\sqrt{q}-1-2(|S|-r)}$-local-spectral expander for any prime power $q$.
\end{lemma}
\subsection{Coboundary Expansion}
We now prove the faces complexes associated with the projected flags complex is good 1-coboundary expander. 

\begin{theorem}[Swap coboundary expansion of the projected flags complex]
\label{thm:swap-coboundary-projected-flags}
Let $d,k, r$ be integers such that $d > k \geq r^5$.
There is some $q_0 = q_0 (k,r)$ such that the following holds.
Let $q \geq q_0$ be any prime power, $\buildings$ be the $\buildings^{2d}_q$-flags complex, and 
\[
S_{k}=\set{d-k, \dots, d-1, d, d+1, \dots, d+k}.
\]
Let $\faces = \mathrm F^r (\buildings^{S_{k}})$ be its faces complex.
Then $\faces$ is a coboundary expander and $h^1(\mathcal Z) \geq \exp(-O(\sqrt r \log^2 r))$.
\end{theorem}

As in \cite[Lemma 8.2]{DD24buildings}, we start by showing it is enough to argue that a constant fraction of projections of the faces complex are themselves co-boundary expanders, a variant of \Cref{lem:color-restriction} from \cite{DD24coboundary} for the faces complex.
In particular, we start with the $r$-faces complex of a $k$-partite complex and consider the projections onto $m$ disjoint $r$-sets of colors from $[k]$, where each such $r$-set acts as a color for the faces complex.
If
\begin{inparaenum}[(i)]
    \item all the 1-skeletons of vertex-links of all such projections are spectral expanders, and
    \item a large fraction of such projections are 1-coboundary expanders,
\end{inparaenum}
then (abstracting \cite[Lemma 8.2]{DD24buildings}) this implies coboundary expansion of the faces complex. 

\begin{restatable}{lemma}{disjointcolor}
\label{lem:disjoint-color-restriction}
Let $k,r$ be integers such that $k\geq r^5$.
Let $X$ be a $k$-partite simplicial complex and $\mathcal Y = \F^r X$.
Let $m = \sqrt{r+1}$.
Assume that for each $I\in \mathrm F\Delta_k(m-1)$ and for each $v\in \mathcal Y^I$, the underlying graph of $\mathcal Y^I_v$ is a $\Theta (1)$-edge expander.
If $\pr{I\sim \mathrm F^r\Delta_k(m-1)}{\mathcal Y^I \text{ is a 1-coboundary expander with }h^1(\mathcal Y^I)\geq \beta}=p$
where $I$ is a uniformly chosen set of $m= \sqrt{r + 1}$ pairwise disjoint colors from $\fcolors$, then $\mathcal Y$ is a 1-coboundary expander with $h^1(\mathcal Y)=\Omega(\beta p^2)$.
\end{restatable}

We defer the proof of \Cref{lem:disjoint-color-restriction} to \Cref{app:omitted} as it is essentially identical to the proof of \cite[Lemma 8.2]{DD24buildings} up to changes in nomenclature.
We will also rely on \cite[Proposition 7.1.1]{DD24buildings} to relate the coboundary expansion of the faces complex to the base complex.

\begin{proposition}[{\citep[Proposition 7.1.1]{DD24buildings}}]
\label{prop:base-to-face}
Let $X$ be an $k$-partite complex that is a $\lambda$-local spectral expander for $\lambda \leq \frac{1}{2r^2}$.
Let $\ell \geq 5$
and let $J = \set{c_1 , c_2 , \dots , c_\ell }$ be a set of mutually disjoint colors $c_j \subseteq [k]$, with $|c_j| \leq r$. 
Denote by $R = \sum^\ell_{j =1} |c_j |$.
Let $\beta > 0$ and assume that for every $I = \set{i_1 , i_2 , \dots , i_\ell }$ such that $i_j \in c_j$
and every $w \in X^{\cup J\setminus I}$ , we have $h^1 (X^I_w)\geq \beta$.
Then $h^1(\mathrm F^J X) \geq \beta_1^R$ for $\beta_1 = \Omega_\ell (\beta)$.
\end{proposition}

Before we begin proving \Cref{thm:swap-coboundary-projected-flags}, it will be helpful to establish some notation and definitions.
\paragraph{Notation}
\begin{itemize}

    \item Let $\buildings = \buildings^{2d}_q$. Note that $\buildings$ is $(2d-1)$-partite. 
    \item Let $\faces = \mathrm F^r (\buildings^{S_{k}})$ and $\tilde \faces = \mathrm F (\buildings^{S_{k}})$.
    \item Let $\fcolors = \binom{S_{k}}{r+1}$.
    \item Let $m=\sqrt{r+1}$
\end{itemize}

Given $w \in \buildings^{S_{k}}$, let $c^* = \col (w) \subseteq S_{k}$ and write $c^* = \set{j_1 , \dots , j_b }$.
Define $b + 1$ $c^*$-\emph{bins} \[ B_i = \set{j \in S_{k} | j_i < j < j_{i+1} }\] for $0 < i < b$
and also $B_0 = \set{j \in S_{k} |j < j_1 }$ and $B_b = \set{j \in S_{k} |j > j_b }$.
Note that the link $\buildings^{S_{k}}_w$ is isomorphic to the join of $\set{(\buildings^{j_{i+1}-j_{i}}_q)^{\tilde B_i}}_{i=0}^b$ where
we have taken $j_0=0$, $j_{b+1} = 2d$ and $\tilde B_i = \set{\tilde j|\tilde  j=j- j_{i} \text{ where }j\in B_i }$ for $i=1, 2,\dots, b$.
We will refer to a $c^*$-bin as just a bin when the $c^*$ is clear from context.

Now, recall that for $J',J \subseteq \binom{S_{k}}{\leq r+1}$ we write $J' \le J$, if $J' = \set{c'_1 , \dots , c'_m }$ where $c'_j \subseteq c_j$.
Let $c^i_j=c_j\cap B_i$.

\begin{definition}[Classification of bins]
Let $J = \set{c_1,c_2,\dots,c_\ell} \subseteq \fcolors$ be a set of colors of $\faces$.
A bin $B_i$ is $J$-\emph{crowded} if $J_i = (c^i_1 , \dots , c^i_\ell)$ has more than one coordinate for which $c^i_j \ne \emptyset$.
It is $J$-\emph{lonely} if there is exactly one such coordinate. Otherwise, it is $J$-\emph{empty}.
When the set $J$ is clear from the context we simply say crowded, lonely, and empty for $J$-crowded, $J$-lonely, and $J$-empty, respectively.

\end{definition}

\begin{definition}[Well-spread colors]
\label{def:well-spread}
Let $J$ be a set of $m$ colors in $\fcolors$.
We say that $J$ is \emph{well-spread} if the following properties hold.
\begin{enumerate}
    \item \label{item:disjointness} $J$ consists of pairwise disjoint colors.
    \item \label{item:global-spread} For every $\ell_1 , \ell_2 \in (\cup J) \cup \set{d-k, d+k}$
    it holds that
    $|\ell_1 - \ell_2 | \geq \frac{2k}{(m(r+1))^3}$. 
    \item \label{item:proj-bin-crowding} For every $J'\subseteq J$ of size $|J'| = 5$ and $\overline{J'} = J \setminus J'$ the following hold.
    \begin{enumerate}
        \item \label{item:small-bin} For every $\ell \in \cup \overline{J'} \cup \set{d-k}$,
        there exists some $\ell' \in \cup \overline{J'} \cup \set{d+k}$ 
        such that $1 < \ell' - \ell \leq \frac{200 k \log (r+1)}{(r+1)m}$. 
        \item \label{item:crowding-bound} For every $c \in J$ , the number of colors $i \in c$ that are in $J'$-crowded $\cup \overline {J'}$-bins is at most 
        $\frac{100(r+1) \log(r+1)}{m \log m}=\frac{200(r+1)}{m}$. 
        
    \end{enumerate}
\end{enumerate}
\end{definition}

\begin{remark}
Our definition of well-spread colors does not include an analogue of \cite[Item 3(c)]{DD24buildings}.
The bin size bounds in Item 3(c) lets them obtain a stronger bound for the flags complex.
Here, we have adapted the proof used for weaker swap-expansion parameter where we don't need the bin size bounds.
\end{remark}

\begin{proof}[Proof of \autoref{thm:swap-coboundary-projected-flags}]

Our proof strategy follows the overall framework of \cite{DD24buildings} for the full flags complex.
The idea is to break $\faces$ into progressively simpler complexes (projections, then deep links) until we have a complex whose coboundary expansion we can bound directly via our bounds for $\buildings^{S_{k}}_w$.

\paragraph{Moving to projections of $\faces$ to well-spread colors:}
First, we move from bounding the coboundary expansion of $\faces$ to coboundary expansion of the projections of $\faces$ on well-spread color-sets $J$ (see \Cref{def:well-spread}).
We do this in two (sub)steps:
\begin{inparaenum}[(i)]
\item first, we show most color-sets are well-spread (see \Cref{claim:good-color-density}),
\item then, we show that if most color-sets are well-spread and all well-spread color sets are good co-boundary expanders, then $\mathcal Z$ itself is a good co-boundary expander (see \Cref{claim:faces-color-restriction}).
\end{inparaenum}

More concretely, denote by $\mathcal J$ the set of well-spread $J$’s as per \Cref{def:well-spread}. We claim the measure of $\mathcal{J}$ tends to $1$ as $r$ grows large, i.e., most projections are well-spread as follows.
\begin{claim}
\label{claim:good-color-density}
Let $k, r$ be such that $k \geq r^5$. Let $6 \leq m \leq (r + 1)$.
The probability that $m$ colors chosen uniformly from $S_{k}$ are well-spread tends to 1 as $r \to \infty$.
\end{claim}
\Cref{claim:good-color-density} is an elementary combinatorial calculation and follows directly from an analogous computation in \cite[Proposition 8.4.1]{DD24buildings} for the correct setting of parameters,\footnote{To be precise, the dimension parameter in \cite[Proposition 8.4.1]{DD24buildings} should be set to $2k+1$ and the arity parameter to $r$.
After this, observe that the bijection from $[2k+1]$ to $S_{k}$ that maps $x\mapsto d-(k+1)+x$ has the property that (i) a well-spread set of color according the definition of \cite{DD24buildings} maps to a well-spread set of colors according to our definition, and (ii) a uniform distribution on $[2k+1]$ induces a uniform distribution on $S_{k}$.} so we omit the proof.

By \Cref{claim:good-color-density}, for large enough $r$ at least half of the sets $J$ are in $\mathcal J$. Thus, applying, \Cref{lem:disjoint-color-restriction} to $\buildings$, we immediately get the coboundary expansion of $\mathcal{Z}$ is lower bounded by the worst-case expansion over spread projections as stated in \Cref{claim:faces-color-restriction}.

\begin{claim}
\label{claim:faces-color-restriction}
    $h^1 (\faces, \Gamma) \geq \Omega(1) \cdot \min_{J \in \mathcal{J}} h^1 (\faces^J, \Gamma).$
\end{claim}

\begin{remark}
We note that
\begin{align*}
\faces^J = \F^r (\buildings^{S_{k}})^J = \F^r(\buildings)^J.
\end{align*}
Comparing this to the proof of flags coboundary expansion in \cite{DD24buildings} we find that they also end up with the projection of the $r$-faces complex of the flags complex to well-spread colors.
As a result, we will be able to proceed with the proof in exactly the same manner from hereon with the changes due to changes in parameters of the well-spreadness definition being addressed when they come up.
\end{remark}

\paragraph{Trickling down the links:}

In the previous step, we reduced the task of showing the coboundary expansion of $\faces$ to coboundary expansion of $\faces^J$ for spread color sets. Our next step will be to reduce the analysis of $h^1(\faces^J)$ to the coboundary expansion of its links. In particular, in \Cref{claim:faces-local-to-global} below, we fix $J \in \mathcal J$ and use a `local to global' argument to move to the links of $\faces^J$.
We note that we need the local spectral expansion to be constant for \Cref{claim:faces-local-to-global} to hold and indeed there is a $q_0$ such that we have local spectral expansion $\faces^J$ to be at least $1/2$ for any $q\geq q_0$ due to \Cref{claim:faces-local-to-global}.
\begin{claim}[{\citep[Claim~8.5.1]{DD24buildings}}]
\label{claim:faces-local-to-global}
Let $J$ 
be a set of $m$ pairwise disjoint colors for $\faces$,  and let $i= -1, 0, \dots, m - 6$, then
\[
h^1(\faces^J, \Gamma) \geq \exp(-O(i)) \cdot \min_{s\in \faces^J (i)} h^1(\faces_s^J, \Gamma),
\]
where we understand $\faces_s^J$ to mean $\faces^{J'}_s$ for $J'=J\setminus \col(s)$.
\end{claim}
We note that, as stated in \cite[Claim~8.5.1]{DD24buildings} the set $J$ should be well-spread, but it is easy to check the proof only relies on disjointness. Thus, by \Cref{claim:faces-local-to-global},
\[
h^1(\faces^J, \Gamma) \geq \exp(-O(m))\cdot \min_{s\in \faces^J (m-6)} h^1 (\faces_s^J, \Gamma).
\]

\paragraph{Reducing to Crowded Colors:}

Fix any $s \in \faces^J(m - 6)$. Our goal is now to reduce from analyzing coboundary expansion of $\faces^J_s$ to $\faces^{\tilde{J}}_s$
where $\tilde J$ are the crowded colors in the link.
In particular, since by definition of spreadness there are few crowded colors,
this will allow us to apply \Cref{prop:base-to-face} with much smaller 
$R=\sum\limits_{j=1}^5 |\tilde c_i|$ to achieve the final desired coboundary expansion.
Toward this end, we'd like to use the following result of \cite{DD24buildings}.

\begin{corollary}[{\citep[Corollary 8.7]{DD24buildings}}]
\label{cor:tensor-corollary}
Let $w\in \buildings$.
There is some constant $\beta > 0$ such that
\[
h^1 (\tilde \faces^J_w, \Gamma) \geq \beta \cdot h^1 (\tilde \faces_w^{\tilde J}, \Gamma),
\]
where $\tilde J = \set{\tilde c_1 , \tilde c_2 , \dots , \tilde c_\ell }$
and $\tilde c_j = \set{ i \in c_j | i \text{ is not in a lonely or empty bin}}$.
\end{corollary}
As stated, the only problem is the above bounds $\tilde \faces^J_w$ for a vertex $w \in \tilde \faces$ (equivalently a face in $\buildings$), while we need to bound $\faces^J_s$. However, by \Cref{obs:face-link} we have $ \faces_s^{J} \cong \mathrm F^{J} \buildings_{\cup s}$, and since $J$ consists entirely of $(r+1)$-sets we indeed have $\mathrm F^{J} \buildings_{\cup s} \cong \tilde \faces_{\cup s}^{J}$.
Thus, applying \Cref{cor:tensor-corollary}, we have
\[
h^1 (\faces_s^J, \Gamma ) = h^1 (\mathrm{F}^J\buildings_{\cup s}, \Gamma ) = h^1(\tilde \faces_{\cup s}^{J}, \Gamma)\geq \Omega( h^1 (\tilde \faces_{\cup s}^{\tilde J}, \Gamma)).
\]

\paragraph{Reduction to the Base Complex:} 

Next, denote by $\beta = \min_{w,I} h^1 (\buildings^{\cup I}_{\cup s\sqcup w}, \Gamma)$ where the minimum is taken over sets $I$
consisting of five singletons, i.e., $I = \set{\set{i_0 }, \dots , \set{i_4 }}$ such that $I \cup \col (s) \le \tilde J$, and $w \in \buildings_{\cup s}$
such that $\col(w) \subseteq \bigcup \tilde J$ and $\col(w)\cap I = \emptyset$.
Since $\tilde Z_{s}^{\tilde J} \cong \F^{\tilde J}\buildings_{\cup s}$, by \Cref{prop:base-to-face},
\[
h^1 (\tilde Z_{s}^{\tilde J}, \Gamma) \geq \mathrm{const} \cdot \beta_1^R,
\]
where $\beta_1 = \Omega(\beta)$ and $R = \sum_j |\tilde c_j |$.
By \Cref{item:crowding-bound} of \Cref{def:well-spread}, for every $c_j$,
\[
|\tilde c_j|=\text{the number of indices from }c_j \text{ in crowded bins}=O\paren {\frac{r \log r}{m \log m}},
\]
so, in total $R = O \paren{\frac{r \log r}{m \log m}}= O\paren{\sqrt{r}}$ 
\begin{equation}
\label{eq:intermediate-bound}
h^1 (\faces, \Gamma)\geq \mathrm{const} \cdot \exp(-O(\sqrt r)) \cdot \beta^{\sqrt r}.
\end{equation}
\paragraph{Using bounds for the Flags complex:} It is left to bound $\beta$, the coboundary expansion of deep spread out links in the projected complex.

\begin{restatable}{claim}{buildingsexpansion}
\label{claim:buildings-expansion}
Let $I = \set{\set{i_0 }, \dots , \set{i_4 }}$, and let $s \in \faces$ be such that $|s| = m - 5$
and there is a well-spread set of colors $J \in \mathcal J$  such that $I \cup \col (s) \le J$.
Let $w'\in \buildings_{\cup s}^{\cup J}$ be such that $\col(w')\cap I = \emptyset$.
Then
\[
h^1(\buildings^{\cup I}_{\cup s\sqcup w'}, \Gamma) \geq \exp(-O (\log^2 r)).
\]
\end{restatable}
\Cref{claim:buildings-expansion} is a variant of \citep[Lemma 8.8]{DD24buildings}, which makes the same statement for the full flags complex under a modified notion of spreadness dependent on the ambient dimension $d$ ($n$ in their definition), while our definition of spreadness critically depends only on $k$, as $d$ grows with the number of vertices and none of the projections we consider would be spread under \cite{DD24buildings}'s definition. Nevertheless, because we are inside the link of $s$, `morally' we may treat the ambient dimension as $k$, in which case our definition reduces to theirs.

In somewhat more detail, the proof of \Cref{claim:buildings-expansion} essentially proceeds by
\begin{inparaenum}[(1)]
    \item first, using \Cref{lem:color-restriction} to move to projections to 4 indices,
    \item relating the coboundary expansion of the complexes to the diameter, and
    \item finally, obtaining a bound on the diameter as result of well-spreadness of $J$.
\end{inparaenum}
Because we are inside a link, the diameter bounds depend on the ratio of the parameters in \Cref{item:global-spread} and parameters in \Cref{item:proj-bin-crowding} which remains unchanged from the definition of \cite{DD24buildings}. For reader's convenience we reproduce relevant parts of the proof of \Cref{claim:buildings-expansion} in \Cref{app:buildings-expansion}.

We can now complete the proof. Namely, by \Cref{claim:buildings-expansion}, we have
\[
\beta = \min_{w,I} h^1 (\buildings_{\cup s\sqcup w}^{\cup I}, \Gamma) \geq \exp(-O (\log^2 r)).
\]
Plugging this back into \Cref{eq:intermediate-bound} we get
\[
h^1 (\faces, \Gamma) \geq \mathrm{const}\cdot \exp(-O(\sqrt r \log^2r))
\]
as needed.
\end{proof}

\section{Applications}
\subsection{Background}
\subsubsection{Agreement testing}
\paragraph{Agreement Testing.} Let $\mathcal{F} \subseteq {\binom{[n]}{k}}$ be a family of $k$ element subsets of $[n]$
and let $\set{f_s : s\to \Sigma | s\in \mathcal F}$ be an ensemble of local functions, each defined over a subset $s\in \mathcal F$.
An agreement test is a randomized property tester for the question:
\begin{center}
\emph{Is there a global function $G: [n] \to \Sigma$ such that $f_s = G|_s$ for all $s\in \mathcal F$?}
\end{center}
Towards answering this question, we will work with the following basic restricted type of 2-query test.
\begin{definition}(2-query agreement test)
Let $\Sigma$ be an arbitrary finite set which we will call the \emph{alphabet}.
Given a family $\mathcal F$ of $k$ element subsets of $[n]$ and an ensemble of functions $\mathcal E = \set{f_s : s\to \Sigma | s\in \mathcal F}$ from $[n]$ to $\Sigma$ which we will call \emph{local functions}, a 2-query agreement test is a distribution $\mathcal D$ over a pair of subsets $s_1,s_2\in \mathcal F$, with the agreement of the ensemble being
\begin{align*}
\agr_{\mathcal D}(\mathcal E)=\pr{s_1,s_2\sim \mathcal D}{f_{s_1}\vert_{s_1\cap s_2}=f_{s_2}\vert_{s_1\cap s_2}}.
\end{align*}
\end{definition}

\begin{remark}
In the literature, sometimes more general agreement tests are considered where the distribution $\mathcal D$ is over $\ell$ subsets in $\mathcal D$.
\end{remark}

We note that for any property testing question we would like the test to be complete, i.e., we would want an object having the property in consideration to pass the test.
By design, agreement tests (as described above) have completeness as one can clearly see restriction of a global function to sets in $\mathcal F$ would pass any 2-query test of the above form.

We also want agreement tests to be sound, i.e., whenever the tests passes there is a global function whose restrictions agree with most functions in the ensemble.
There are two regimes of soundness of agreement tests that are typically considered:
\begin{description}
    \item[The 99\% regime or the high acceptance regime.] In this regime, we want to claim that there is a global function $G:[n]\to \Sigma$ such that $G\vert_s=f\vert_s$ for most $s\in \mathcal F$ whenever $f_{s}\vert_{s\cap s'}=f_{s'}\vert_{s\cap s'}$ for most $s,s'\in \mathcal F$.
    More formally, we want to show
    \[
        \agr_{\mathcal D}(\set{f_s}_{s\in \mathcal F})\geq 1-\eps \quad \Longrightarrow \quad \exists G:[n]\to \Sigma,\ \Pr{G\vert_s=f_s}\geq 1-O(\eps).
    \]
    \item[The 1\% regime or the low acceptance regime.] In this regime, if there is at least some, say $\eps$, fraction of $s,s'\in \mathcal F$ so that $f_{s}\vert_{s\cap s'}=f_{s'}\vert_{s\cap s'}$, we want to claim that there is a global function $G:[n]\to \Sigma$ such that $G\vert_s \approx f\vert_s$ for some fraction of $s\in \mathcal F$, say $\poly(\eps)$.
    More formally, we want to show there there is some polynomial $\poly$ such that for any $\eps\geq \eps_0$,
    \[
        \agr_{\mathcal D}(\set{f_s}_{s\in \mathcal F})\geq \eps \quad \Longrightarrow \quad \exists G:[n]\to \Sigma,\ \Pr{\dist(G\vert_s, f_s)\leq \eta}\geq \poly(\eps),
    \]
    where $\eta$ is some `closeness' parameter.
    Assuming (for simplicity) that the polynomial is of form $\poly(x)= x^c$, we call $(\eps_0,\eta, c)$ the \emph{soundness parameters} and call a test $\mathcal D$ with these parameters a $(\eps_0,\eta, c)$-\emph{sound test}.
    More consisely, we may leave $\eps_0,c$ unspecified say a test is $\eta$-sound, i.e., a $\eta$-sound test is $(\eps_0,\eta,c)$-sound for some $\eps_0,c>0$.
\end{description}
In the low soundness regime, agreement testing is quite non-trivial even over $\mathcal F=\binom{[n]}{k}$, and there is a long line of works which settled this~\cite{GS00,DG08,DR06,DS14,IJKW10,DN23}.

Here, we will study agreement testing on sparser families of sets, in particular, we will take $\mathcal F=X(k-1)$ for some simplicial complex $X$ of dimension $d>k-1$
and define an agreement test based on its weights.
In particular, we will consider the following agreement test.
\paragraph{$(k,\sqrt k)$-test.} The $(k,\sqrt k)$-test $\mathcal D$ is a 2-query test given by the following process.
    \begin{enumerate}
        \item Sample $t\sim \Pi_d$.
        \item Sample $s\sim \Pi_{\sqrt{k}-1}$, conditioned on $s\subseteq t$.
        \item Sample $s_1,s_2\sim \Pi_{k-1}$, conditioned on $s\subseteq s_1,s_2 \subseteq t$.
    \end{enumerate}

Finally, we remark that for any function $G:[n]\to \Sigma$ the ensemble $\set{G\vert_s}_{s\in \mathcal F}$ can be thought of as an encoding of $G$ that is locally testable using the agreement test.
This code is often referred to as a derandomized direct product code and agreement testing is referred to as direct product testing in this context.

\subsubsection{Probabilistically Checkable Proofs (PCPs)}
For the most part, we will forego a discussion of the `proofs' view of probabilistically checkable proofs and take a more combinatorial view of PCPs using the language of the Label Cover problem.
\begin{definition} 
An instance of \emph{Label Cover} $\Psi = (G = (L\cup R, E), \Sigma_\ell, \Sigma_r, \Phi = \set{\phi_e}_{e \in E} )$ 
consists of a bipartite graph $G$, alphabets $\Sigma_\ell$, $\Sigma_r$ and constraints 
$\phi_e : \Sigma_\ell \to \Sigma_r$, for each edge $e\in E$.
\end{definition}

Given a label cover instance $\Phi$, the goal is to find assignments $A_\ell : L \to \Sigma_\ell$ and $A_r : R \to \Sigma_r$ that satisfy as many of the constraints as possible,
namely that maximize the quantity
\[
\val_{\Phi}(A_\ell, A_r) = \frac{1}{|E|}|\set{e = (u, v) \in E |\phi_e(A_\ell(u))= A_r(v)}|.
\]
This maximum value is referred to as the value of the instance $\Phi$ and denoted by
\[
\val(\Phi) = \max_{A_\ell,A_r} \val_\Phi(A_\ell, A_r).
\]
The label cover problem is a special case of the 2-ary constraint satisfaction problem where each constraint is a relation on the alphabet instead of being a function.
\begin{definition}
An instance $\Psi = (G = (V, E), \Sigma, \set{\Phi_e}_{e\in E})$ of
a 2-CSP consists of a weighted digraph $G$,
alphabet $\Sigma$, and constraints $\Phi_e \subseteq \Sigma \times \Sigma$, one for each edge.
\end{definition}
Let $\gapLC{\alpha}{\beta}$ be the promise problem wherein
the input is an instance $\Phi$ of label cover promised to either have $\val(\Phi)\ge \alpha$
or else $\val(\Phi) \le \beta$, and the goal is to distinguish between these two cases.

\begin{definition} 
An instance of \emph{Generalized Label Cover} $\Psi = (G = (L\cup R, E), \Sigma_\ell, \Sigma_r, \Phi = \set{\phi_e}_{e \in E} )$ 
consists of a bipartite graph $G$, alphabets $\Sigma_\ell$, $\Sigma_r$, $\set{\Sigma_\ell(u)}_{u\in L}$ with $\Sigma_\ell(u)\subseteq \Sigma_\ell$, for each $u\in L$, and constraints 
$\phi_e : \Sigma_\ell(u) \to \Sigma_r$, for each edge $e=(u,v)\in E$.
\end{definition}

\subsection{New Agreement Testers}
\cite{DD24a, BM24, BMVY25} showed the following agreement testing result for complexes with strong expansion properties.
We refer the reader to \cite[Remark 4.2]{BM24} for the specific parameters.
\begin{theorem}[\cite{DD24a}]
\label{thm:expander-to-agrtest}
There is a $c > 0$ such that for all $\delta > 0$, there is $r \in \N$ such that for any $\ell\geq \exp(\poly(1/\delta))$,
any $k\geq \poly (\ell)$ and there is a $\gamma_0(k)$ such that for any $\gamma\leq \gamma_0$ the following holds.

If a $k$-dimensional clique complex $X$ is a $\gamma$-local spectral expander and $(c2^{-o(r)}, r)$-swap coboundary expander,
then the $(\ell, \sqrt \ell)$-agreement test over $X(\ell)$ with respect to every alphabet $\Sigma$ has soundness $\delta$. 
Namely, if $F : X(\ell-1) \to \Sigma^\ell$ passes the $(\ell, \sqrt \ell)$-agreement test with respect to $X$ with probability at least $\delta$,
then there is $f : X(0) \to \Sigma$ such that
\[
\pr{A\sim \mu_\ell}{\dist(F [A], f |A) \leq \delta} \geq \poly(\delta).
\]
\qed
\end{theorem}
As an immediate corollary we get that the $(\ell,\sqrt \ell)$-test on the projected flags complex has a low soundness test.
\flagsAgrtest*
\begin{proof}
Let $c > 0, r, \gamma_0(k)$ be as in \Cref{thm:expander-to-agrtest}.
First note that by taking $q_0$ the flags complex $\buildings_q^{2d,S_{k}}$ has $\gamma\leq \gamma_0$ by \Cref{thm:spectral-projected}.
Now, by \Cref{thm:swap-coboundary-projected-flags} we have the the swap coboundary expansion required by \Cref{thm:expander-to-agrtest}.
Hence, by \Cref{thm:expander-to-agrtest}, the $(\ell, \sqrt \ell)$-agreement test over $\buildings^{S_{k}}(\ell)$ with respect to every alphabet $\Sigma$ has soundness $\delta$ for any $\ell\geq \exp(\poly(1/\delta))$.
\end{proof}

\subsection{A Near-linear Size PCP}
Finally, we show the projected flags complex can be used to construct PCPs using techniques in~\cite{BMVY25}.
\pcpTheorem*

We first record the set of properties of the projected flags complex in \Cref{thm:good-hdx} that are required for the PCP construction.
The statement of \Cref{thm:good-hdx} is an analog of Theorem 2.13 in~\cite{BMVY25}, albeit with slightly different parameters, as we replace the Chapman-Lubotzky complexes~\cite{CL25} with the projected buildings.
\begin{definition}[Symmetrically-Connected Complex]
\label{def:sym-bound-main}
Let $L$ be a $k$-partite complex.
We call $L$ \emph{symmetrically-connected} if
\begin{enumerate}
    \item For any $i\ne j\in [k]$, the diameter of $L^{ij}$ is bounded by $O(\frac{k}{|i-j|})$.
    \item $\symmetric_{L}$ acts transitively on its top faces.
\end{enumerate}
\end{definition}
\begin{theorem}
\label{thm:good-hdx}
For all $\delta \in (0, 1)$, and for all $C>0$ the following holds.
For large enough $\ell, k \in \N$, and for large enough $n \in \N$, there is a prime power $q=\Theta(\log^C n)$ and $d=\Theta\paren{\sqrt{\frac{\log(n)}{\log \log (n)}}}$,
one can construct in time $\poly(n)$
a $2k$-dimensional complex $X$ for which $n \leq |X(0)| \leq  n q^{O(d)}$ such that the following holds.
\begin{enumerate}
\item \label{item:partite} The complex $X$ is $(2k+1)$-partite.
\item \label{item:clique} The complex $X$ is a clique complex.
\item \label{item:degree-bound} $\Delta_{max}(X) \leq q^{O(kd)}$.
\item \label{item:sym-connected} Each vertex-link $L$ of $X$ is symmetrically-connected.
In other words, we have the following:
\begin{enumerate}
    \item \label{item:transitive-action} For each vertex link $L$ of $X$ there is a group $\symmetric(L)$ that acts transitively on the $(k-2)$-faces $L(k-2)$.
    \item \label{item:diameter-bound} For each vertex link $L$ of $X$ and any colors $i \ne j$, the bipartite graph $(L_i(0), L_j (0))$ is uniformly weighted (over edges) and has diameter $O(\frac{\max\set{i,j}}{|i-j|})$.
\end{enumerate}
\item  \label{item:link-eigenvalue-bound} $X$ is a $\frac{1}{\sqrt q}$-product.
\item \label{item:eigenvalue-bound} The 1-skeleton $G = (X(0), X(1))$ satisfies $\lambda(G) \leq \frac{1}{k(\sqrt q -1)}$ and $|\lambda|(G) \leq \frac{1}{2k}$.
\item \label{item:product-test} The $(\ell, \sqrt \ell)$-direct product test on $X$ has soundness $\delta$.
\end{enumerate}
\end{theorem}
\begin{proof}
Given $\delta>0$, $C>0$, and $n$,
take $q=2^{\ceil{\log_2 \log^C_2 n}}$ and $d=\ceil{\sqrt{\log_q n}}$.
If $q$ and $k$ are large enough, then the complex $X=\buildings^{2d, S_{k}}_{q}$ has $(\ell, \sqrt \ell)$-agreement test with soundness $\delta$ for large enough $\ell$ by \Cref{thm:flags-agrtest} (this shows \Cref{item:product-test}).
We get \Cref{item:link-eigenvalue-bound,item:eigenvalue-bound} from \Cref{thm:spectral-projected}, \Cref{fact:grassmann_expansion}, and the Trickle-Down theorem (\Cref{thm:trickle-down}).
It is elementary to see the group of automorphisms $\mathrm{Aut}(L)$ acts transitively on the top faces of $L$ where $L$ is a vertex link (this gives us \Cref{item:transitive-action}).\footnote{As in the proof of \Cref{lem:three-projection-coboundary}, taking an appropriate basis for the vector spaces in any two such flags, we can transform one to the other with the corresponding basis exchange matrix.}
\Cref{item:diameter-bound} follows directly from Claim 8.7.2 in \cite{DD24buildings} and hence we get \Cref{item:sym-connected}.
\Cref{item:degree-bound} is a straight-forward counting which we omit for the sake of brevity.
\Cref{item:partite,item:clique} also follow directly from the definition.
\end{proof}

\paragraph{Overview of PCP construction in \cite{BMVY25}.}
The first step of \cite{BMVY25}'s PCP involves reducing an arbitrary 2-query PCP, viewed as a 2-CSP, to a 2-CSP whose constraint graph underlies an HDX.
In particular, they show how to embed an arbitrary 2-CSP on a graph $G$ while maintaining the soundness provided that $G$ has a fault-tolerant routing protocol, and show any sufficiently local-spectral HDX (see~\Cref{item:link-eigenvalue-bound,item:eigenvalue-bound}) with dense symmetric links (see~\Cref{item:transitive-action,item:diameter-bound}) indeed satisfies such a protocol.
The second step uses a size-efficient direct product tester from HDX.
This allows derandomized parallel repetition for 2-CSPs on an HDX, and reduce their soundness to a small
constant close while incurring a poly-logarithmic blow-up in size.
Combining the first and second steps they get a quasi-linear PCP with small soundness, albeit with a large alphabet.

Since the projected flags complex satisfies essentially the same properties as the Chapman-Lubotzky complex albeit with different parameters (compare \Cref{thm:good-hdx} with \cite[Theorem 2.13]{BMVY25}), the same construction also works even if we use projected flags complex.
\Cref{thm:general-pcp} states a general version of this theorem.
In \Cref{app:pcp}, we provide a sketch of the proof to highlight the changes in parameters. 

\begin{restatable}[{\cite{BMVY25}}]{theorem}{generalpcp}
\label{thm:general-pcp}
There exist $\alpha, \eps_0, d_0 > 0$, and $r,C\in \N$ such that for large enough $n \in \N$ the following holds.
Let $X$ be $d$-partite $\lambda$-product such that
\begin{inparaenum}[(i)]
    \item its underlying graph $G$ has $|\lambda|(G)\leq \alpha/4$;
    \item its vertex links are symmetrically-bounded;
    \item its max degree $\Delta(X)$ is bounded by $\Delta$;
    \item the dimension $d\geq d_0$ and its size $|X(0)|= n$.
\end{inparaenum}
Also, assume that the $(\ell, \sqrt \ell)$-direct product tester on $X$ has soundness $\delta^2\leq \eps_0^2/16$.
Then there is a $\poly(n)$-time procedure mapping any 2-CSP $\Psi'$ with $n$ vertices on a $r$-regular graph $H$ and alphabet $\Sigma$, to an instance of Label Cover $\Phi$ over the
weighted inclusion graph $(X(\sqrt \ell), X(\ell))$,  over left alphabet $\Sigma^{\poly(\Delta \log (n)  \log |\Sigma|) \ell}$ and right alphabet $\Sigma^{\poly(\Delta \log (n)  \log |\Sigma|)\sqrt{\ell}}$ satisfying the following properties.
\begin{enumerate}[(i)]
    \item \textbf{Completeness:} If $\val(\Psi') = 1$ then $\val(\Psi) = 1$.
    \item \textbf{Soundness:} If $\val(\Psi')\leq 1-\poly(d)r^2\lambda^2\log^C n - \frac{1}{\log^C n}$, then $\val(\Psi) \leq \delta$.
    
\end{enumerate}
\end{restatable}

We now start with a size efficient construction of PCP due to \cite{Meir16} stated in a more convenient formulation in terms of hardness of 2-CSPs.\footnote{We need a regularization step \citep[Lemma~4.2]{Din07} on top of the original statement to obtain this form.}
\begin{theorem}[\cite{Meir16}]
\label{thm:small-soundness-pcp}
There exist constants $c > 0, \delta \in (0,1)$ such that there is a polynomial-time reduction mapping a
3-SAT instance $\phi$ of size $n$ to a 2-CSP instance $\Psi$ over the regular graph $G = (V, E)$ and alphabet $\Sigma$ where
\begin{itemize}
    \item $\size(G) \leq n(\log n)^{c \log \log n}$ and $|\Sigma| = O(1)$.
    \item If $\phi$ is satisfiable, then $\val(\Psi) = 1$.
    \item If $\phi$ is not satisfiable, then $\val(\Psi) \leq 1 - \delta$.
\end{itemize}
\end{theorem}

\begin{remark}
The construction of PCP in \cite{Meir16} is almost combinatorial. The only algebraic component is the multiplicative codes used in the paper.
\end{remark}

Applying this to our complex, i.e., combining \Cref{thm:good-hdx,thm:general-pcp,thm:small-soundness-pcp}, we obtain \Cref{thm:specific-pcp}.

\begin{theorem}
\label{thm:specific-pcp}
For all $\delta \in (0, 1)$, there exist constants $C, C' > 0$ such that for all sufficiently large integers $d$ and $n$, there exist $k, \ell$ with $k\geq \poly(\ell)$ and $\ell \geq \exp(\poly(1/\delta))$ so that the following holds. 
Let $\set{X_{n'}}_{n'\in \N}$ be the infinite sequence of complexes from \Cref{thm:good-hdx},
where every $X_{n'}$ is a $(2k+1)$-dimensional complex on $n'$ vertices with parameters
$q = \Theta(\log^{C'} n')$ and $\delta$ that is constructible in time $\poly(n')$.
There is a polynomial time reduction mapping a 3SAT instance $\phi$ of size $n$ to a label cover instance $\Psi$ over the weighted inclusion graph $(X_{n'}(\ell), X_{n'} (\sqrt \ell))$ for some $n'\leq  n \log^{C\log \log n} n$, such that
\begin{enumerate}
    \item If $\phi$ is satisfiable, then $\Psi$ is satisfiable.
    \item If $\phi$ is unsatisfiable, then $\val(\Psi)\leq \delta$.
    \item The left alphabet of $\Psi$ is $\Sigma^\ell$ and right alphabet is $\Sigma^{\sqrt \ell}$ for some alphabet $\Sigma$ with $\log |\Sigma| \leq q^{Ckd}$.
\end{enumerate}
\end{theorem}

As an immediate corollary, we get the version of the PCP theorem as stated in \Cref{thm:pcp-theorem}.

\section*{Acknowledgements}
We thank Yotam Dikstein and Roy Meshulam for helpful discussions.
We are grateful to the International Centre for Theoretical Sciences (ICTS) for organizing the ``ICTS Workshop on HDXs and Codes'' (code: ICTS/HDXandCodes2025/04)
which greatly facilitated our discussions with Roy Meshulam.
We thank the anonymous reviewers of CCC 2026 for their comments and suggestions.
\bibliographystyle{alpha}
\bibliography{ref.bib}

\appendix
\section{Omitted Proofs}
\label{app:omitted}
\subsection{A Small Proposition on Expansion of Multipartite Graphs}
In this subsection, we prove the following proposition which roughly states that in a $k$-partite graph if for every two parts the bipartite graph induced by them is an expander, the entire $k$-partite graph is also an expander.
\kpartExpansion*
\begin{proof}[Proof of \Cref{prop:kpartite-expansion}]
Let $\mathsf M:\R^V\to \R^V$ be the transition matrix of the underlying walk on $G$, and similarly $\mathsf M^{ij}:\R^{P_j}\to \R^{P_i}$ for the bipartite walk from $P_i$ to $P_j$.

Let $f:V\to \R$ be a unit vector orthogonal to the all ones vector, i.e., we have $\ex{\pi_{1}}{f}=0$ and $\Ex{f^2}=1$.
Now, we bound $\inprod{\mathsf Mf,f}$ as follows. Write $f_i$ to denote the restriction of $f$ to part $P_i$ and observe
\begin{align*}
\inprod{\mathsf Mf,f}
&=\ex{ij\sim \kappa_2}{\inprod{\mathsf M^{ij}f_{j},f_{i}}_i},
\end{align*}
where the inner product is taken over the stationary distribution of $M^{ij}$ (on part $P_i$), since we may draw an edge by first selecting the parts of its vertices $i,j$, then conditionally on the selected parts.

Now, let $f_i^{\parallel}=\ex{P_i}{f_i}\one_{P_i}$ and $f_i^\perp = f_i - f_i^\parallel$. We have
\begin{align*}
\inprod{\mathsf Mf,f}&=\ex{ij\sim \kappa_2}{\inprod{\mathsf M^{ij}f_{j},f_{i}}_i}\\
&=\ex{ij\sim \kappa_2}{\inprod{\mathsf M^{ij}\paren{f^\parallel_{j}+f^\perp_j}, f^\parallel_i+f^\perp_i}_i}\\
&=\ex{ij\sim \kappa_2}{\inprod{\mathsf M^{ij}f^\parallel_{j},f^\parallel_{i}}_i} + \ex{ij\sim \kappa_2}{\inprod{\mathsf M^{ij}f^\perp_{j},f^\perp_{i}}_i}.
\end{align*}
We shall now bound the orthogonal and parallel components separately.
\paragraph{Bounding the parallel term.} Write $\kappa_1$ to denote the stationary distribution on our weighted complete graph $K$ over $[k]$, and define the function $g:[k] \to [0,1]$ by $g(i)=\ex{P_i}{f}$.
Observe that $g\perp\one_{[k]}$ as $\ex{i\sim \kappa_1}{g(i)}=\Ex{f}=0$
and is at most unit norm
\begin{align*}
\norm{g}_{\kappa_1}^2=\ex{i\sim \kappa_1}{g(i)^2}=\ex{i\sim \kappa_1}{\ex{P_i}{f}^2}\leq \ex{i\sim \kappa_1}{\ex{P_i}{f^2}}=\ex{\pi_1}{f^2}=1.
\end{align*}
So, the parallel component can be bound as
\begin{align*}
\ex{ij\sim \kappa_2}{\inprod{\mathsf M^{ij}f_j^\parallel, f_i^\parallel}_i}&=\ex{ij\sim \kappa_2}{\ex{P_j}{f}\ex{P_i}{f}\inprod{\mathsf M^{ij}\one_{P_j}, \one_{P_i}}_i}\\
&=\ex{ij\sim \kappa_2}{g(j)g(i)}\\
&\leq \lambda_2(K)\ex{i\sim \kappa_1}{g(i)^2}&g\perp_{\kappa_1} \one_{[k]}\\
&\leq \lambda_2(K).& \paren{\norm{g}_{\kappa_1}\leq 1}
\end{align*}
\paragraph{Bounding the orthogonal term.}
Note that $\ex{ij\sim \kappa_2}{\inprod{\mathsf M^{ij}f_j^\perp, f_i^\perp}_i}\leq \ex{ij\sim \kappa_2}{\lambda_{ij}\norm{f_i^\perp}\norm{f_j^\perp}}$ as $f_j^\perp\perp \one_{[k]}$.
Then by AM-GM we have
\begin{align*}
\ex{ij\sim \kappa_2}{\lambda_{ij}\norm{f_i^\perp}\norm{f_j^\perp}}
&\leq \ex{ij\sim \kappa_2}{
    \lambda_{ij}\cdot
    \paren{
        \frac{
            \norm{f_i^\perp}^2+\norm{f_j^\perp}^2
        }{2}
    }
}&\paren{\text{AM-GM Inequality}}\\
&= \ex{ij\sim \kappa_2}{\lambda_{ij}\cdot \norm{f_i^\perp}^2}\\
&= \ex{i\sim \kappa_1}{\ex{j\sim \kappa_2[i]}{\lambda_{ij}\norm{f_i^\perp}^2}}\\
&\leq \max_{i\in [k]}\ex{j\sim \kappa_2[i]}{\lambda_{ij}}\ex{ij\sim \kappa_2}{\norm{f_i^\perp}^2}&\paren{\ex{i\sim \kappa_1}{\norm{f_i^\perp}^2}\leq \ex{i\sim \kappa_1}{\norm{f_i}^2}=1}\\
&\leq \max_{i\in [k]}\ex{j\sim \kappa_2[i]}{\lambda_{ij}}
\end{align*}
as desired.
\end{proof}

\subsection{Color Restriction Lemma for the Faces Complex}

We now prove \Cref{lem:disjoint-color-restriction} which we recall is an abstraction of \cite[Lemma 8.2]{DD24buildings} that is specified to the full flags complex. We restate the lemma for convenience.

\disjointcolor*

We will need the following notion of distributional closeness.
\begin{definition} 
Let $(P, Q)$ be an (ordered) pair of probability distributions supported on a set $\Omega$.
We say that $(P, Q)$ are \emph{$(A, \alpha)$-smooth} if for every $v \in A$ it holds that $\alpha \pr{P}{v} \leq \pr{Q}{v}$.
We say that $(P, Q)$ are \emph{$\alpha$-smooth} if they are $(\Omega,\alpha)$-smooth.
\end{definition}
The following property is easy to verify from the definitions. We omit its proof.
\begin{claim}
Let $(P, Q)$ be $(A, \alpha)$-smooth. Then for every $B \subseteq A$ it holds that $\alpha \pr{P}{B} \leq \pr{Q}{B}$.\qed
\end{claim}

We now move to proving \Cref{lem:disjoint-color-restriction}.
The proof is essentially verbatim the proof of \cite[Lemma 8.2]{DD24buildings} (i.e., the abstraction simply follows from observing the authors do not need any properties of the building beyond those given in the statement of \Cref{lem:disjoint-color-restriction}), and is included only for completeness.

\begin{proof}[Proof of~\Cref{lem:disjoint-color-restriction}]

Following \cite{DD24buildings}, we first claim we may assume a good fraction of larger projections onto $m^4$ colors are co-boundary expanders. This ensures that for any color $c$, the fraction of colors $c'\in I$ that intersect $c$ is a negligible fraction of the colors in $I$.
\begin{claim}[Claim 8.3.1 in \cite{DD24buildings}]
Let $I\sim \mathrm F^r \Delta_n((r+1)^2-1)$ of colors.
Then with probability $p/2$, $\mathcal Y^I$ has coboundary expansion $\Omega(\beta p)$.
\end{claim}
The claim is a simple Markov argument combined with \Cref{lem:color-restriction} (see~\cite{DD24buildings} for details).

For the rest of the proof, we fix $f: \mathcal Y(1) \to \Gamma$.
We need to show that there is some $g : \mathcal Y(0) \to \Gamma$ such that $ \beta' dist(f, \delta g) \leq \wt(\delta f)$, where $\beta'= \Omega(\beta p^{2})$.
Note that for a $I$ such that $\mathcal Y^I$ has coboundary expansion $\Omega(\beta p)$ it is easy to define $g$ inside $\mathcal Y^J$ (so long as $\delta f$ is not much larger than expected in $I$).
The challenge is to extend $g$ to the entire $\mathcal Y$.
To that end, for a fixed $I$ we define the following two distributions over triangles.
\begin{itemize}
\item $T_{IIn}$ is the distribution over $uvw \in \mathcal Y(2)$ given that $u, v \in \mathcal Y^I(0)$ and $w \not \in \mathcal Y^J(0)$.
\item  $T_{nnI}$ is the distribution over $uvw \in \mathcal Y(2)$ given that $u, v \not \in \mathcal Y^I(0)$ and $w \in \mathcal Y^I(0)$.
\end{itemize}

A second averaging argument implies the existence of a `good' $I$ such that $\mathcal{Y}^I$ is a coboundary expander, and for which $|\delta f|$ is not much larger when restricted to $I$, triangles with one edge in $I$, or triangles with one vertex in $I$, which allows us to define $g$ close to a coboundary in $\mathcal{Y}^I$ then extend it.

\begin{claim}[{\cite[Claim 8.3.2]{DD24buildings}}]
There exists some set $I\in \mathrm F^r \Delta((r+1)^2-1)$ colors such that
\begin{enumerate}[(i)]
    \item All colors in $I$ are mutually disjoint and $\mathcal Y^I$ has coboundary expansion $\Omega(\beta p)$.
    \item $\wt_{\mathcal Y^I} (\delta f) \leq O(p^{-1})\wt(\delta f)$.
    \item $\wt_{T_{IIn}} (\delta f) = \pr{uvw\sim T_{IIn}} {\delta f(uvw) \ne Id} \leq O(p^{-1})\wt(\delta f)$.
    \item $\wt_{T_{nnI}} (\delta f) = \pr{uvw\sim T_{nnI}} {\delta f(uvw) \ne Id} \leq O(p^{-1})\wt(\delta f)$.
\end{enumerate}
\end{claim}
Choose a $I$ as in the above claim.
Now, we define $g:\mathcal Y(0)\to \Gamma$ as follows.
\begin{description}
\item[Inside $\mathcal Y^I(0)$:]
Since $\mathcal Y^I$ has coboundary expansion $\Omega(\beta p)$, there exists some $g$ such that
\[
\Omega(\beta p) \dist_{\mathcal Y^I} (f, \delta g) \leq \wt_{\mathcal Y^I} (\delta f) \leq O(p^{-1})\wt(\delta f).
\]
We take $g$ to be one such function. Now, we extend it to rest to the rest of the complex as follows.
\item[Outside $\mathcal Y^I(0)$:]
For $v\in \mathcal Y(0) \setminus \mathcal Y^I(0)$, let $g(v) = \maj_{u\in \mathcal Y^I_v(0)} \set{f(vu)g(u)}$ to be the most popular value (ties
broken arbitrarily).
We note that $\mathcal Y^I(0)\ne \emptyset$ because the color of $v$ can intersect at most $r+1$ colors in
$I$ because they are mutually disjoint, and this leaves at least $(r+1)^2 - (r+ 1)$ remaining colors
to choose from.
\end{description}
Let
\begin{align*}
E_j&= \set{\set{u, v}\in \mathcal Y(1)| \left| \mathcal Y^I(0) \cap \set{u, v}\right |= j},& j= 0, 1, 2.
\end{align*}
Let $\dist_{E_i}$ denote the distance between two functions from $C^1$ as per the distribution over $\mathcal Y(1)$ conditioned on the edge being in $E_i$.
We already know that
\[
\dist_{E_2} (f, \delta g) = \dist_{\mathcal Y^I}(f, \delta g) \leq O\paren{\frac{1}{\beta p^2}}\wt(\delta f)
\]
since this follows from the coboundary expansion of $\mathcal Y^I$ and the fact that $g$ was chosen to minimize this
distance.
It is now sufficient to show that $\dist_{E_1} (f, \delta g) = O( 1/\beta p^{2} ) \wt(\delta f)$, and that $\dist_{E_0} (f, \delta g) = O( 1/\beta p^{2} )\wt(\delta f)$ as $\dist( f, \delta g)$ is a convex combination of $\dist_{E_i} ( f, \delta g)$.

First, we consider the edges partly outside $\mathcal Y^I$, i.e., the edges in $E_1$. 
\begin{claim}[Claim 8.3.3 in \cite{DD24buildings}]
\label{claim:d-smoothness}
Let $I$ be a set of $(r + 1)^2$ mutually disjoint colors.
Consider the following distributions over $E_1$.
\begin{description}
    \item[$D_0$:]The distribution where one samples an edge $uv \in \mathcal Y(1)$ conditioned on being in $E_1$.
    \item[$D_2$:]The distribution where one samples an edge by first sampling a triangle $uvw \sim T_{IIn}$ such that
        $v, w \in \mathcal Y^I(0)$ and $u\not \in \mathcal Y^I(0)$, and then outputting $uv$.
\end{description}
Then $(D_0 , D_2)$ are $\frac{1}{2}$-smooth.
\end{claim}

By definition $\dist_{E_1} (f, \delta g) = \pr{uv\sim D_0} {f(uv) \ne \delta g(uv)}$.
By \Cref{claim:d-smoothness} we can replace $D_0$ with $D_2$ up to a constant factor loss and just show
\[
\pr{uv\sim D_2}{f(uv) \ne \delta g(uv)} \leq O\left( \frac{1}{\beta p^2}\right)\wt(\delta f)
\]
Towards that end, fix some $v\in \mathcal Y (0) \setminus \mathcal Y^I(0)$,
and let 
\[
\eps_v = \pr{u\in \mathcal Y^I_v (0)} {f(uv) \ne \delta g(uv)} = \pr{u\in \mathcal Y^I_v (0)} {g(v)^{-1} \ne g(u)^{-1}f(uv)}.
\]
The underlying graph of $\mathcal Y^I_v$ is a $\Omega(1)$-edge expander by \Cref{lem:spectral-to-edge} and \Cref{thm:spectral-projected}.
Therefore, by \Cref{lem:expansion-majority}, it holds that
\[
\pr{u\in \mathcal Y^J_v(0)} {g(v)^{-1} \ne g(u)^{-1}f(uv)} \leq O\paren{ \pr{uw\in \mathcal Y^J_v (1)} {g(w)^{-1}f(wv)\ne g(u)^{-1}f(uv)}}.
\]
In the following claim we should that the error probability on the right hand side of the inequality above can be bounded by the distance of $f$ from the coboundary restricted to $\mathcal Y^I$ and weight of $\delta f$ with respect to $T_{IIn}$ .

\begin{claim*}
If $g(w)^{-1}f(wv) \ne g(u)^{-1}f(uv)$, then either $f(wu) \ne \delta g(wu)$ or $\delta f(vwu) \ne Id$.
\end{claim*}
\begin{proof}
Suppose $f(wu) = \delta g(wu)$ and $f(vwu) = Id$. 
Then we get
\begin{align*}
&g(w)^{-1}f(wv) = g(u)^{-1}f(uv)\\
\Leftrightarrow & f(vw)g(w)g(u)^{-1}f(uv) = Id\\
\Leftrightarrow & f(vw)\delta g(wu) f(uv) = Id\\
\Leftrightarrow & f(vw)f(wu)f(uv) = Id &(\text{by assumption, }\delta g(wu)=f(wu))\\
\Leftrightarrow & \delta f(vwu) = Id
\end{align*}
So, we get the contrapositive of the statement to be true, i.e., if $f(wu) = \delta g(vw)$ and $\delta f(vwu) = Id$, then $g(w)^{-1}f(wv) = g(u)^{-1}f(uv)$.
\end{proof}

From the above claim we get that 
\begin{align*}
\pr{uv\sim D_2}{f(uv) \ne \delta g(uv)} &\leq \ex{v\not \in \mathcal Y^I(0)}{\eps_v}\\
&= O \paren{ \ex{v\not \in \mathcal Y^I(0)}{ \pr{u \in \mathcal Y^I_v(0)} {g(v)^{-1} \ne g(u)^{-1}f(uv)}} }\\
&= O \paren{ \ex{v\not \in \mathcal Y^I(0)}{ \pr{uw \in \mathcal Y^I_v(1)} {\delta f(vwu)\ne Id}}  + \ex{v\not \in \mathcal Y^I(0)}{ \pr{uw \in \mathcal Y^I_v(1)} {f(wu) \ne \delta g(wu)}}}\\
&= O(\wt_{T_{IIn}} (\delta f) + \dist_{E_2} (f, \delta g))\\
&\leq O\paren{\frac{1}{\beta p^2}}\wt(\delta f)
\end{align*}
where the second-to-last line follows from the fact that choosing $v$ and then an edge $uw$ as above is equidistributed with $T_{IIn}$.

We can now move to showing that
$\dist_{E_0} (f, \delta g) = \pr{uv\sim P_0} {f(uv) \ne \delta g(uv)} = O( 1/ \beta p^2 )\wt(\delta f)$.
The proof is similar to the $E_1$ case. 
\begin{claim}[Claim 8.3.4 in \cite{DD24buildings}]
\label{claim:p-smoothness}
Let $J$ be a set of $(r + 1)^2$ mutually disjoint colors. Consider the following distributions over $E_0$ :
\begin{description}
    \item[$P_0$:]The distribution where one samples an edge $uv \in \mathcal Y(1)$ conditioned on $u,v \not \in \mathcal Y^I(0)$.
    \item[$P_1$:]The distribution where one samples an edge $uv \in \mathcal Y(1)$ by first sampling a triangle $uvw \sim T_{nnI}$ such that $u,v \not \in \mathcal Y^I(0)$ and $w\in \mathcal Y^I(0)$, and then outputting $uv$.
\end{description}
Then $(P_0 , P_1)$ are $\frac{1}{2}$-smooth.
\end{claim}

\begin{claim}[Claim 8.3.3 in \cite{DD24buildings}]
\label{claim:d-smoothness-two}
Let $J$ be a set of $(r + 1)^2$ mutually disjoint colors.
Consider the following distributions over $E_1$.
\begin{description}
    \item[$D_0$:]The distribution where one samples an edge $uv \in \mathcal Y(1)$ conditioned on being in $E_1$.
    \item[$D_1$:]The distribution where one samples an edge by first sampling a triangle $uvw \sim T_{nnI}$ such that
        $u, w \not \in \mathcal Y^I(0)$ and $v\in \mathcal Y^I(0)$, and then outputting $uv$.

\end{description}
Then $(D_1, D_0)$ are $\frac{1}{2}$-smooth.
\end{claim}

By \Cref{claim:p-smoothness} it is enough to show that
\[
\pr{uv\sim P_1}{f(uv) \ne \delta g(uv)} = \pr{uvw \sim T_{nnI} ,u,v\not \in \mathcal Y^J(0)} {f(uv) \ne  \delta g(uv)} \leq O( 1/ \beta p^2 )\wt(\delta f).
\]
Now, similar to before we have:
\begin{claim*}
If $f(uv)=\delta g(uv)$, then either $\delta g(vw) \ne f(vw)$ or $\delta g(uv) \ne f(uw)$ or $\delta f(uvw)\ne Id$.
\end{claim*}
\begin{proof}
Suppose $\delta g(vw) = f(vw)$ and $\delta g(uv) = f(uw)$.
Then we get
\begin{align*}
\delta f(uvw)&=Id\\
\Leftrightarrow f(uv)f(vw)f(wu)&=Id\\
\Leftrightarrow f(uv) &= f(uw)f(wv)\\
&=\delta g(uw) \delta g(wv)\\
&= \delta g(u)g(v)^{-1}=\delta g(uv).
\end{align*}
\end{proof}

From the above claims we get
\begin{align*}
\pr{uv\sim P_1}{f(uv) \ne \delta g(uv)} &\leq \pr{uvw \sim T_{nnJ}}{\delta f(uvw) \ne Id} + 2 \pr{uw \sim D_1}{f(uw)\ne \delta g(uw)}\\
&\leq \pr{uvw \sim T_{nnJ}}{\delta f(uvw) \ne Id} + 4 \pr{uw \sim D_0}{f(uw)\ne \delta g(uw)}&\text{(\Cref{claim:d-smoothness-two})}\\
&\leq \wt_{T_{nnJ}}(\delta f) + 4\dist_{E_1}(f, \delta g)=O\paren{\frac{1}{\beta p^2}}\wt(\delta f).
\end{align*}
Thus, we have $\dist_{E_i} (f, \delta g)=O\paren{\frac{1}{\beta p^2}}\wt(\delta f)$ for each $i=0,1,2$ which completes the proof.
\end{proof}

\subsection{Coboundary Expansion of Spread Projections}
\label{app:buildings-expansion}

Recall our goal is to prove the following claim lower bounding the co-boundary expansion of spread 5-partite projections of deep links of the projected flags complex.

\buildingsexpansion*

As discussed, the proof closely follows \cite[Lemma 8.8]{DD24buildings}, which shows the same statement with a stronger (ambient-dimension dependent) notion of spread colors. As in \cite{DD24buildings}, we will rely directly on the following lemma lower bounding the coboundary expansion of slightly separated 4-partite projections of the full flags complex $\buildings=\buildings^{2d}_q$.

\begin{lemma}[{\cite[Lemma 8.9]{DD24buildings}}]
\label{lem:4projection-bound}
Let $\Gamma$ be any group.
Let $I' = \{i_0 < i_1 < i_2 < i_3\}$ such that $i_3 > 21$ and such that $i_j - i_{j-1} \geqslant 3$. Then
\[ h^1(\buildings^{I'}, \Gamma) \geqslant \exp\left(-O\left(\log\frac{i_3}{i_1 - i_0} \cdot \log\frac{i_3}{i_1}\right)\right). \]
\end{lemma}

\begin{proof}[Proof of \Cref{claim:buildings-expansion}]
Let us denote by $w = \cup s \sqcup {w'}$. Fix some $I$, and let 
\[ I' = \{i_0 < i_1 < i_2 < i_3\} \]
be any four indexes inside $I$.
By \Cref{thm:spectral-projected} and \Cref{lem:color-restriction}, if we show that $\buildings^{I'}_w$ has coboundary expansion at least $\beta$, then it holds that $h^1(\buildings^{I'}_w, \Gamma) \geqslant \Omega(\beta)$. 

The coboundary expansion $h^1(\buildings^{I'}_w, \Gamma)$ depends on the choice of $I'$ and $w$. We split into two cases based on whether or not elements in $I'$ all lie in the same $\col(w)$-bin.

\paragraph{Different Bins}

We begin by considering the case where the indices of $I'$ are in more than one $\col(w)$-bin. We will rely on the following claim of \cite{DD24buildings}.
\begin{claim}[{\cite[Claim 8.7.1]{DD24buildings}}]
\label{claim:coboundary-diameter-bound}
Let $I' = \{i_0 < i_1 < i_2 < i_3\}$, let $w \in \buildings$ be such that $\col(w) \cap I = \emptyset$ and let $v \in w$ be $i_0 \leq \dim(v) \leq i_3$.
\begin{enumerate}
\item If $i_0 \leq \dim(v) \leqslant i_1$, then $h^1(\buildings^{I'}_w, \Gamma) \geq \Omega\paren{\frac{1}{\diam(\buildings^{\{i_1,i_2,i_3\}}_w)}}$,
\item If $i_1 < \dim(v) < i_2$, then $h^1(\buildings^{I'}_v, \Gamma) \geq \Omega\left(\frac{1}{\min\set{\diam(\buildings^{\set{i_0,i_1}}_v),\diam(\buildings^{\set{i_2,i_3}}_v)}}\right)$,
\item If $i_2 \leq \dim(v) \leq i_3$, then $h^1(\buildings^{I'}_v, \Gamma) \geq \Omega\paren{\frac{1}{\diam(\buildings^{\{i_0,i_1,i_2\}}_v)}}$.
\end{enumerate}
\end{claim}

Now, assume, for example, that $i_0$ is separate from $i_1,i_2,i_3$ (the rest of the cases follow from the same argument). In this case, by \Cref{claim:coboundary-diameter-bound} we will have that 
\[
    h^1(\buildings^{I'}_w, \Gamma) \geqslant \Omega\left(\frac{1}{\diam(\buildings^{\set{i_1,i_2,i_3}}_w)}\right).
\] 
If $i_1,i_2,i_3$ are not all in the same bin, then the diameter is constant (since by definition every dimension $i_1$ subspace in the link is contained in every dimension $i_3$ subspace in the link).
Otherwise, they are in the same bin.
Let $[k_0,k_1]$ be the $\col(w)$-bin that contains $i_1,i_2,i_3$.
Let $\tilde{i}_j = i_j - k_0$. Let's consider explicitly the subspaces in $\buildings^{\{i_1,i_2,i_3\}}_{w}$.
Let $v_0,v_1 \in w$ be the subspaces of dimension $k_0$ and $k_1$, respectively.
Then $\buildings^{\{i_1,i_2,i_3\}}_w$ contains all spaces that contain $v_0$ and are contained in $v_1$ of dimensions $I'$.
It is easy to see that $\buildings^{I'}_{w}$ is isomorphic to a 3-partite colored complex whose ambient space is $\mathbb{F}_q^{k_1-k_0}$, with parts corresponding to dimensions $\tilde{i}_1,\tilde{i}_2,\tilde{i}_3$. 
\begin{claim}[{\cite[Claim 8.7.2]{DD24buildings}}]
\label{claim:projected-flags-diameter-bound}
Let $I' \subseteq [n]$, $|I'| \geq 2$.
Then $\diam(\buildings^I_w) = O(\frac{\max I'}{\max I' - \min I'})$.
\end{claim}
By \Cref{claim:projected-flags-diameter-bound} and the above characterization of $\buildings_w^{I'}$ we therefore have
\[
\diam(\buildings^{\{i_1,i_2,i_3\}}_{w}) = \diam(\buildings^{k_1-k_0,\{\tilde{i}_1,\tilde{i}_2,\tilde{i}_3\}}) \leq O\left(\frac{\tilde{i}_3}{\tilde{i}_3 - \tilde{i}_1}\right).
\] 
We have that $\tilde{i}_3 \leqslant k_1 - k_0 \leqslant 200 k \frac{\log (r+1)}{(r+1)m}$ by well-spreadness.
To state this explicitly, the fact that $s \in \faces(m-6)$ implies that any $\col(\cup s)$-bin has length $\leq O(k \frac{\log (r+1)}{(r+1)m})$ by \Cref{item:small-bin} in \Cref{def:well-spread}.
The $\col(w)$-bins cannot be longer since $\cup s \subseteq w$.
Moreover, by \Cref{item:global-spread} in \Cref{def:well-spread}, the distance between every two colors in $J$ is at least $\Omega(\frac{k}{(m(r+1))^3})$, so in particular $\tilde i_3 - \tilde i_1= i_3 - i_1 \geq \Omega(\frac{k}{(m(r+1))^3})$.
Thus
\[
    \diam(\buildings^{\set{i_1,i_2,i_3}}_{w}) \leq O(\poly(r))
\]
and in turn
\[
    h^1(\buildings^{I'}_{w}) = \Omega(1/\mathrm{poly}(r)) = \exp(-O(\log r)).
\]

\paragraph{Same Bin} Next, consider the case where all indexes of $I'$ are in the same bin.
Similar to before, let us denote by $[k_0,k_1]$ be the $\col(w)$-bin that contains $I'$.
As before let $\tilde{I} = \{\tilde{i}_j\}_{j=0}^3$ where $\tilde{i}_j = i_j - k_0$ and as before $\buildings^{I'}_w$ is isomorphic to the four-partite complex $\buildings^{\tilde{I}}$ whose ambient space is $\mathbb{F}_q^{k_1-k_0}$, with parts corresponding to dimensions $\tilde{I}$.

The proof of this case follows from \Cref{lem:4projection-bound}, which shows that 
\[
    h^1(\buildings^{I'}_w, \Gamma) = h^1(\buildings^{k_1-k_0,\tilde{I}}, \Gamma) \geq \exp\paren{-O\paren{\log\paren{\frac{\tilde{i}_3}{\tilde{i}_1 - \tilde{i}_0}} \log\paren{\frac{\tilde{i}_3}{\tilde{i}_0}}}},
\] 
along with similar bounds on $\tilde{i_3}$ and $\tilde{i}_1-\tilde{i}_0$ by well-spreadness. In particular, as in the prior case, by \Cref{item:small-bin} in \Cref{def:well-spread}, every bin has size $O(k\frac{\log (r+1)}{(r+1)m})$ so $\tilde{i}_3 \leqslant O(k \frac{\log (r+1)}{(r+1)m})$, while by \Cref{item:global-spread} implies the distance between every two colors in $J$ (and therefore in $\tilde{I}$) is at least $\Omega(\frac{k}{(m(r+1))^3})$. Thus we have both $\tilde{i}_0 = i_0 - k_0 \geq \Omega(\frac{k}{(m(r+1))^3}$) and similarly $\tilde{i}_1 - \tilde{i}_0 = i_1 - i_0 \geq \Omega(\frac{k}{(m(r+1))^3})$, so
\[
    h^1(\buildings^{I'}_w, \Gamma) \geq \exp(-O(\log^2(r)))
\]
as desired.
\end{proof}
\newcommand{\zprod}{\textrm{\textcircled{z}}}
\newcommand{\rprod}{\textrm{\textcircled{r}}}
\newcommand{\inbound}{\mathtt{IN}}

\section{Generalized Construction of a PCP on an HDX}
\label{app:pcp}
In this section, we outline for completeness the following abstraction of \cite{BMVY25}'s PCP construction (sans alphabet reduction). We emphasize this section contains no new technical material over \cite{BMVY25}, and follows immediately from examining their proof. We give the sketch for completeness, since a result of this general form is not directly stated in \cite{BMVY25}.

\generalpcp*

\begin{remark*}
\cite{BMVY25} states \Cref{thm:general-pcp} for Generalized Label Cover instead of Label Cover.
We note that the result still holds for the vanilla label cover.
Consider an arbitrary set of maps $\set{f_u:\Sigma\to \Sigma(u)}_{u\in V}$ such that $f_u(a)=a$ for all $a\in \Sigma(u)$.
We can form a label cover instance $\Phi$ with alphabet $\Sigma$ from a generalized label cover instance $\Phi_0$ on $\Sigma, \set{\Sigma(u)}$ (on the same graph) with the same value using these maps by having the constraints as follows.
A constraint on $(u,v)$ as per $\Phi$ will be satisfied by $(a,b)$ whenever $(f_u(a),f_v(b))$ satisfies the constraint on $(u,v)$ as per $\Phi_0$.
\end{remark*}

The full proof follows immediately from a sequence of main claims presented in order below.
We omit exact details for the sake of avoiding redundancy and refer the reader back to \cite{BMVY25}. While we largely keep to the original notations, definitions, etc.,\ from \cite{BMVY25}, we will introduce a few additional convenient definitions for clarity. We stress these definitions serve only to streamline the statements and arguments, and do not otherwise change the material in any substantive way. 

Overall, there are two main components to the proof: 
\begin{inparaenum}[(a)]
    \item showing a `fault tolerant, link-to-link routing protocol' on any nice enough HDX, and
    \item showing that one can embed a PCP onto the underlying graph of any complex with such a protocol.
\end{inparaenum}
At the end, there is also a standard gap amplification step accomplished through the $1\%$-regime agreement test on the complex.
\subsection{Routing on an HDX}
In this section, we outline the link-to-link routing in \cite{BMVY25}, abstracting what requirements are needed in their argument. Before proceeding it will be helpful to setup some notation.
\paragraph{Notation.}
For $\nu \in [0,1]$, let $\maj_\nu (\sigma_1, \dots, \sigma_k)$ denote the \emph{threshold majority function} on $(\sigma_1, \dots , \sigma_k)$. 
If for some $\sigma$, we have $\sigma_i = \sigma$ for at least $\nu$ fraction
of indices $i$, then we take $\maj_\eta (\sigma_1, \dots , \sigma_k)=\sigma$; otherwise $\maj_\nu (\sigma_1, \dots , \sigma_k)=\bot$ (signifying `invalid').
We take the \emph{threshold} $\nu$ to be at least $0.5$ and it is typically close to 1 (e.g., $\nu=0.99$).


The routing problem of interest to us is the \emph{almost everywhere reliable transmission problem} from \cite{DPPU88},
where the goal is to design a sparse graph $G = (V, E)$ that allows transference of messages in a fault tolerant way.
To elaborate, at the start of the routing, each vertex $v$ has a message from $\Sigma$
and destination given by $\pi(v)$ where $\pi : V \to V$ is an arbitrary permutation. The goal is to design a $T$-round protocol to send the message in $v$ to $\pi(v)$. In each round, the protocol works by sending symbols in $\Sigma$ from each of the nodes to a subset of its neighbors in $G$. At the end of the protocol, we'd like to ensure most $v$ indeed transmit their message to $\pi(v)$, even if a small fraction of the vertices or edges behave maliciously.
The parameters of interest in this problem are as follows.
\begin{itemize}
    \item \textbf{Work Complexity:} the work complexity of a protocol is the maximum computational
    complexity any node in the graph $G$ incurs throughout the protocol.
    \item \textbf{Degree:} this is the (maximum) degree of $G$. We would like the degree to be as small as possible.
    \item \textbf{Tolerance:} a protocol is considered $(\eps(n), \nu(n))$-edge tolerant if, when an adversary
    corrupts up to $\eps(n)$-fraction of edges of the graph (allowing them to deviate arbitrarily from the protocol),
    at most $\nu(n)$-fraction of the transmissions from $u \to \pi(u)$ are disrupted.
\end{itemize}

In the following definition we want to define the notion of $T$ rounds of communication.
Any communication will be based on a graph $G$ and will use an alphabet $\Sigma_0\supseteq \set{\bot}$
with $\bot$ understood to mean `invalid'.
The idea is: each vertex starts with some label from $\Sigma_0$.
At each round of the communication each vertex will look at the messages received from its neighbors and its label
and decide the messages to be sent to each of its neighbors along with its label according to some function.
The messages received by $u$ from $v$ in the next round will be a function of the message sent by $v$ and whether the channel between $u$ and $v$ worked correctly or not (quantified by a ``channel response'').
At round 0, the received messages will be assumed to be $\bot$.

\begin{definition}[Communication Protocol]
\label{def:communication}
Let $G$ be a graph.
A \emph{communication protocol} $\mathcal P$ on $G$ using an alphabet $\Sigma_0\supseteq \set{\bot}$ is
described by a set of routing functions
\begin{align*}
\set{r_v:\Sigma_0^{N(v)\cup \set{v}}\to \Sigma_0^{N(v)\cup \set{v}}}_{v\in G}.
\end{align*}
Let $E_{\text{conn}}=\overrightarrow E \cup  \set{(v,v)|v\in V(G)}$.
Given any set of initial labels on vertices $m_{v}\in \Sigma_0$ and number of rounds $T\in \N$,
for any set of \emph{channel responses} $\set{c_{i,(u,v)}:\Sigma_0\to \Sigma_0}_{i\in [T], (u,v)\in \overrightarrow E}$
we obtain a set of messages $\set{\inbound_{i}:E_{\text{conn}}\to \Sigma_0}_{i\in [T]}$
that obey the following relations
\begin{align*}
\inbound_{0}(v,v)&=m_v && \forall v\in V(G),\\
\inbound_{0}(v,u)&=\bot && \forall v\in V(G), \forall u\in N(v),\\
\inbound_{i}(v,u)&=c_{i,(u,v)}(r_v(\inbound_{i-1}(v, u)))&&\forall v\in V(G), \forall i\in [T],\forall u\in N(v),\\
\inbound_{i}(v, v)&=r_v(\inbound_{i-1}(v, v))&&\forall v\in V(G), \forall i\in [T].
\end{align*}
We will say a vertex $u$ \emph{received} the message $\inbound_i(u, v)$ from the vertex $v$ in the round $i$.
For a fixed $v\in V(G)$ we call $(\inbound_{i}[u,\cdot])_{i\in [T]\cup \set{0}}$ the \emph{transcript} of $\mathcal P$ at $u$.
The \emph{work complexity} of $\mathcal P$ is the total computation that any vertex $v$ performs for computing $r_v$ throughout the protocol,
as measured by the circuit size for the equivalent Boolean functions that the vertex computes.\qed
\end{definition}
\begin{remark*}
\Cref{def:communication} is a memoryless model of a communication protocol, meant to formalize only what is needed for the PCP application in \cite{BMVY25}.
We will depart slightly from the convention in \cite{BMVY25} have labels on top of each vertex at each round for notational convenience. 
\end{remark*}

We now move to formally defining edge-tolerant routing protocols which need to account for possibly corrupted edges.
\begin{definition}[Edge Tolerant Routing]
Let $\mathcal R$ described by $\set{r_v:\Sigma_0^{N(v)\cup \set{v}}\to \Sigma_0^{N(v)\cup \set{v}}}_{v\in G}$ be a communication protocol on a graph $G$ using an alphabet $\Sigma_0$.
We say $\mathcal R$ is a $T$-round $(\eps(n), \nu(n))$-\emph{edge-tolerant routing protocol} if the following holds.
Call an edge $(u, v) \in G$ \emph{uncorrupted} if $v$ receives any message transferred from $u$ across $(u,v)$, i.e., for any round $i\in [T]$ and for any possible messages $m\in\Sigma_0$ sent from $u$ we have $c_{i,(u,v)}(m)=m$; otherwise, we say the edge $(u, v)$ is \emph{corrupted}.
Let $\pi : [n] \to [n]$ be any permutation of $V(G)$ and $\Sigma=\Sigma_0\setminus \set{\bot}$.
Then for all functions $f:V(G)\to \Sigma$, after running the protocol $\mathcal R$ for $T$ rounds with each vertex $v$ being labeled by $f(v)$, we get a final labeling $g: V(G)\to \Sigma_0$ with the guarantee that for any possible subset of corrupted edges of measure at most $\varepsilon(n)$ we have
\[
\pr{u\in V(G)}{f(u)\ne g(\pi(u))}\leq \nu(n).
\]
Finally, we say that the function $g$ was \emph{computed} by this protocol.\qed
\end{definition}

The main result of this subsection (\Cref{lem:link-to-link}) is a slight variation on edge-tolerant routing over simplicial complexes. Instead of starting with an assignemnt over vertices, we will start with an assignment to the vertex \textit{links} of the complex. Then, as in typical edge-tolerant routing, given any permutation on the vertices, the protocol `views' the starting majority value on $v$'s link as it's initial message, and aims to compute an assignment such that on almost all links, the majority value is the permutation of this initial majority assignment.

To formalize this, we will define the notion of link-to-link routing.
For technical reasons, it would have been convenient for the underlying graph to be regular.
Since this isn't the case for our complex, as in \cite{BMVY25} we will move to the `zig-zag product'~\cite{RVW02}\footnote{\cite{BMVY25} use a slight generalization of the zig-zag product.} of the complex's underlying graph with the following family of expanders~\cite{Mar73} (see also \cite{HLW06}). 
\begin{lemma}[\cite{Mar73}]
\label{lem:expander-family}
For all $\lambda > 0$, there exist $r, m_0 \in \N$ and a family of $r$-regular graphs $\mathcal H = \set{H_m}_{m\geq m_0}$
which is polynomial-time constructible,
where for each $m \geq m_0$ that graph $H_m$ has $m$ vertices and $|\lambda|(H_m) \leq \lambda$.
\end{lemma}

\begin{definition}[Zig-Zag Product]
Let $G$ be a graph and $\mathcal H= \set{H_k}_{k\in I}$ be a family of graphs indexed by $I$ such that $|V(H_k)|=k$.
Further assume that $\set{\deg(u)}_{u\in V(G)}\subseteq I$.
The \emph{zig-zag product} between $G$ and $\mathcal H$, denoted $G\zprod \mathcal H$, has the following set of vertices and edges.
\begin{description}
    \item[Vertices:] Replace a vertex $u$ of $G$ with a copy of $H_{\deg(u)}$, that is the graph from $H$ on $\deg (u)$ many vertices (henceforth called a \emph{cloud}).
    For $u \in V (G)$, $c \in V (H_{\deg(u)})$, let $(u, c)$ denote the $c$-th vertex in the cloud of $u$.
    For a vertex $v=(u,c)\in V(Z)$, we will use $v_1$ to denote $u$ and $v_2$ to denote $c$.
    \item[Edges:] $((u, c_1), (v, c_4 ))$ is an edge if there exist $c_2$ and $c_3$ such that $(v, c_4)$ can be reached from $(u, c_1)$ by the following sequence of steps:
    \begin{enumerate}[i.]
        \item take a step in the cloud of $u$ to go to $(u, c_2)$, where an edge $c_1\sim c_2$ must be in $H_{\deg(u)}$;
        \item take a step between the clouds of $u$ and $v$ to go to $(v, c_3)$, where $u$ must be the $c_3$-th neighbor of $v$ and $v$ must be the $c_2$-th neighbor of $u$;
        \item finally, take a step in the cloud of $v$ to reach $(v, c_4)$, where an edge $c_3\sim c_4$ must be in $H_{\deg(u)}$. \qed
    \end{enumerate}
\end{description}
For a vertex $z\in G\zprod \mathcal H$, $z_1$ and $z_2$ denote the components of $z$ from $G$ and $\mathcal H$, respectively.
\end{definition}

The key property of the zig-zag product is that taking the product of an arbitrary expander $G$ with a family of regular expander graphs $\mathcal H$
yields a regular expander graph with structure closely related to $G$. 
\begin{lemma}
If $G$ is a graph on $n$ vertices with $|\lambda|(G) \leq \alpha$,
and $\mathcal H = \set{H_m}_{m\geq m_0}$ is a family of $r$-regular graphs with $|\lambda|(H_m) \leq \beta$ for all $m \geq m_0$,
then $G \zprod \mathcal H$ is a $r^2$-regular graph with $|\lambda|(G \zprod \mathcal H)$ at most $\alpha + \beta + \beta^2$.
\end{lemma}

Finally, we can now formally define the notion of link-to-link routing. 

\begin{definition}[Link-to-Link Routing]
Let $\mathcal H$ be a family of $k$-regular graphs.
We say that a simplicial complex $X$ has an \emph{$\mathcal H$-regularized $(\eps(n), \eta(n), \nu(n))$-link-to-link routing protocol} with round complexity $T(n)$ and work complexity $w(n)$ if the following holds.
Let $G = (X(0), X(1))$.
Let $Z = G \zprod \mathcal H$ and $\mathcal S = \set{(v, u) : v=(v_1,v_2) \in V (Z), u \in X_{v_1}(0)}$.
Let $\pi : V (Z) \to V (Z)$ be any permutation on $V(Z)$.
Let $\Sigma$ be some alphabet.
There is (i) a communication protocol on $G$ using some alphabet $\Sigma_0\supseteq \set{\bot}$ with round complexity $T(n)$ and work complexity $w(n)$,
and (ii) an polynomial-time computable (injective) encoding function $\mathsf{Enc}:\Sigma^\mathcal S \to \Sigma_0^{V(G)}$
such that given a function $A_0 : \mathcal S \to \Sigma$ satisfying 
\[
\pr{v=(v_1,v_2)\sim V(Z)}{\maj_{0.99}(A_0(v, u) | u \sim X_{v_1}(0)) \ne \bot} \geq 1 - \eta,
\]
if we start the protocol with $\mathsf{Enc}(A_0)$ on the vertices of $G$, 
then for all possible corruptions of at most an $\eps$-fraction of edges, at the end of $T$ rounds vertices of $G$ is labeled with $\mathsf{Enc}(A_T)$,
where $A_T : \mathcal S \to \Sigma$ is a function satisfying
\[
\pr{v=(v_1,v_2)\sim V(Z)}{\maj_{0.99}(A_T(\pi(v), w)|w\sim X_{\pi(v)_1}(0))=\maj_{0.99}(A_0(v,u)|u\sim X_{v_1}(0))}\geq 1-\nu(n).
\]
In this case, we say $A_T$ was \emph{computed} by this protocol.\qed
\end{definition}

To construct a link-to-link routing protocol over a complex, we will require that the vertex-links of the complex have the following strong connectivity and symmetry properties.

\begin{definition}[Symmetrically-Connected Complex, Restatement of \Cref{def:sym-bound-main}]
\label{def:sym-bound}
Let $L$ be a $k$-partite complex.
We call $L$ \emph{symmetrically-connected} if
\begin{enumerate}
    \item For any $i\ne j\in [k]$, the diameter of $L^{ij}$ is bounded by $O(\frac{k}{|i-j|})$.
    \item $\symmetric_{L}$ acts transitively on its top faces.
\end{enumerate}
\end{definition}

\begin{remark*}
The definition of link-to-link routing and symmetrically connected complexes were implicit in various statements of \cite{BMVY25}.
\end{remark*}

The following lemma, an abstraction of \cite[Section 4]{BMVY25}, states that any local spectral expander with symmetrically-connected vertex links has a link-to-link routing.
\begin{lemma}
\label{lem:link-to-link}
There exist $\eps_0, \alpha, d_0(\alpha) > 0$, and $r,C\in \N$ such that for any large enough $n \in \N$ the following holds.
Let $X$ be a $d$-partite $\lambda$-product such that
\begin{inparaenum}[(i)]
    \item its underlying graph $G$ has $|\lambda|(G)\leq \alpha/4$;
    \item its vertex links are symmetrically-connected;
    \item its max degree $\Delta(X)$ is bounded by $\Delta=\Delta(n)$;
    \item the dimension $d\geq d_0$ and its size $|X(0)|= n$.
\end{inparaenum}
Then, there exists a family of $r$-regular graphs $\mathcal H$ such that
$X$ has an $\mathcal H$-regularized $(\eps_0, \eta, \eta+\poly(d)r^2\lambda^2\log^C n+ \frac{1}{\log^C n})$-link-to-link routing protocol
that computes the function from $\mathcal S$ to an alphabet $\Sigma$
with round complexity $T = O(\log n)$ and work complexity $\poly(T\Delta \log |\Sigma|)$.
\end{lemma}
We sketch \cite{BMVY25}'s proof of this lemma below. 
Take $\mathcal H$ to be the family of $r$-regular expander graphs from \Cref{lem:expander-family} with $|\lambda|(H_m) \leq \alpha/2$, where $\alpha$ is some universal constant (as in \Cref{thm:ext-paths}).

There are three main steps in constructing an ($\mathcal H$-regularized) link-to-link protocol on $X$:
\begin{inparaenum}[(i)]
    \item finding paths across links,
    \item finding paths inside links,
    \item and combining them to get a communication protocol on $G=(X(0),X(1))$.
\end{inparaenum}
Towards completing the first step, we show the existence of a `relaxed pebble routing protocol' on $Z=G\zprod \mathcal H$ (see \cite[Section 3.1]{BMVY25}).


\begin{definition}[Relaxed Pebble Routing]
We say that a graph $G$ has an $T$-round $(\Delta,\nu)$-relaxed pebble-routing protocol if the following holds.
For all permutations $\pi : V (G) \to V (G)$ there is a $T$-round communication protocol on $G$
such that in each round every vertex receives at most $\Delta$ message symbols in $\Sigma$,
and then sends these messages forward to at most $\Delta$ of its neighbors.
If the protocol starts with $f : V (G) \to \Sigma$, then at the $T$-th round we
have the function $g : V (G) \to \Sigma$ on the vertices satisfying $g(\pi(u)) = f (u)$ for at least $1-\nu$ fraction of $u$.
\end{definition}
We now state the relaxed pebble-routing result due to \cite{Upf94,BMVY25}.
\begin{theorem}[\cite{Upf94,BMVY25}]
\label{thm:ext-paths}
There exists a universal constant $\alpha > 0$ such that the following holds.
Let $G = (V, E)$ be a regular expander graph with $|\lambda|(G) \leq \alpha$.
Let $c \geq 0$ be a fixed constant and $\pi:V\to V$ be a permutation.
Then there is a $\poly(|E|)$-time algorithm to construct a protocol with $O(\log n)$-length paths $\mathcal P$ from $u$ to 
$\pi(u)$ for all but $O(\log^{-c} (n))$-fraction of $u$.
Furthermore, every vertex in $V$ is used in at most $t = O(\log^{c+1} n)$ paths in $\mathcal P$.
In other words, there is a $\poly(|E|)$-time algorithm to construct a $O(\log n)$-round $(O(\log^{c+1} n),O(\log^{-c} (n)))$-relaxed pebble-routing protocol on $G$. 
\end{theorem}

Now, we can proceed with completing the proof of \Cref{lem:link-to-link}.
The missing details of the rest of the proof may be be found in \cite[Section 4.4]{BMVY25}, although some relevant lemmas are quoted from elsewhere in \cite{BMVY25} as needed.
We also maintain the same paragraph names to help the reader.

\paragraph{Setting up paths over 1-links}
Fix a permutation $\pi:V(Z)\to V(Z)$ and let $\mathcal R$ be a relaxed-pebble-routing protocol as per \Cref{thm:ext-paths} that has $|V (Z)|$ paths each of same length $T' = O(\log n )$ (repeat the final vertex if a path is smaller) where every vertex and edge in $Z$ is used in at most $\log^{C_1} n$ paths,
for some constant $C_1$ that depends on $C$.
Note technically for at most $\frac{1}{\log^C n}$ of vertices, we do not have a path, but as a matter of convention, we say there is an invalid path between $u$ and $\pi(u)$ as per $\mathcal R$ when there is no path between them.
Thus, we have $|V(Z)|$ total paths.
Each valid path $P=u_1\to u_2 \to \dots \to u_{T'}$ can be interpreted as the following length-$T$ `link-to-link message transfer' on $X$ for $T=2T'-1$:
\[
X_{(u_1)_1} (0) \to  X_{(u_1)_1,(u_2)_1}(0) \to X_{(u_2)_1} (0) \to \dots \to X_{(u_{T'})_1}(0).
\]
Expand each path in $\mathcal R$ as shown above into vertex-links connected by edge-links.
For even $t \in [T]\cup \set{0}$, let $L_{j,t}$ denote the vertex-link that occurs in the $(t/2)$-th time-step of the $j$-th path in $\mathcal R$,
and for odd $t \in [T]$, let $L_{j,t}$ denote the intermediate edge-link that is the intersection of the vertex-links $L_{j,t-1}$ and $L_{j,t+1}$.
Thus, if the $j$-th path was $u_1\to u_2 \to \dots \to u_{T'}$ we would have
\[
L_{j,0}=X_{(u_1)_1} (0) \to  L_{j,1}=X_{(u_1)_1,(u_2)_1}(0) \to L_{j,2}=X_{(u_2)_1} (0) \to \dots \to L_{j,T}=X_{(u_{T'})_1}(0).
\]


\paragraph{Setting up paths inside 1-links.}
\cite{BMVY25} then show that a set of short paths that connects most pairs of vertices (others are declared invalid) within a link (hereafter referred to as internal paths) can be constructed in polynomial time for any symmetrically-connected complex $L$.
The following two lemmas from \cite[Section 4.1]{BMVY25} aid in this.
\begin{lemma}[{\cite[Lemma 4.1]{BMVY25}}]
\label{lem:paths}
Let $L$ be a $k$-partite symmetrically-connected complex such that the following hold.
Fix $\eps > 0$ and a pair of indices $i, j \in [k]$ with $|i-j|\geq k\eps$.
Then there is a $\poly(|L^{ij} (1)|)$-time algorithm that constructs a set of paths
$\mathcal P_{ij}= \set{P(u, v)}_{u\in L^i(0),v\in L^j(0)}$,
where each valid path is of length $O(1/\eps)$ and at most $O(\eps)$ fraction of paths are invalid. Furthermore, each edge in $L^{ij}(1)$ is used in at most $O(\frac{|L^i(0)||L^j(0)|}{\eps^3|L^{ij}(1)|})$ paths in $\mathcal P_{ij}$.
\end{lemma}
Now, using a probabilistic argument on the output of \Cref{lem:paths}, one can get \Cref{lem:bounded-corruption}.
\begin{lemma}[{\cite[Lemma 4.2]{BMVY25}}]
\label{lem:bounded-corruption}
Let $L$ be a $k$-partite symmetrically-connected complex.
Fix $\eps > 0$.
Then there is a $\poly(|L|)$-time algorithm to construct a
set of paths $\mathcal P = \set{P_{U, V} }_{U \ne V \in L}$ in which each valid path has length at most $O(1/\eps^{1/8})$.
Furthermore, any
adversary that corrupts at most an $\eps$-fraction of the edges in $L$ corrupts at most an $O(\eps^{1/8})$-fraction of the paths in $\mathcal P$.\footnote{Here, we call a path corrupted if any edge in the path is corrupted.}
\end{lemma}

For every vertex link $X_u$ for $u \in X(0)$, we use the algorithm in \Cref{lem:bounded-corruption}
with the parameter $\sqrt \eps$, to construct a collection of short paths $\mathcal P_{X_u} = \set{P_{X_u} (v, w)}_{v,w\in X_u (0)}$ 
between all pairs of vertices $v, w$ inside the link $X_u$.
We will refer to these as the internal paths in $X_u$.

\paragraph{The description of the routing protocol.} With this setup, we now describe the final global routing procedure. 
    \subparagraph{Encoding of $A_0:\mathcal S \to \Sigma$ on $G$:}
    for the $j$-th (interlink) path, we store the value $A_0(v_j,u)$ on the vertex $u\in X_{(v_j)_1}$, where $v_j\in V(Z)$ is the first vertex in the interlink path.
    \subparagraph{`Composing' interlink and internal paths for message passing:}
    At each time-step $t\in [T]$, consider the $j$-th path.
    A vertex $u\in L_{j,t}$ first sets its stored message to the corresponding $j$-th path as the majority (with threshold 1/2) of all the messages received (that were sent using internal paths) from vertices of $L_{j,t-1}$.
    Then, it sends the message it has stored to every vertex in $L_{j,t+1}$ using the paths in $\mathcal P_{L_{j,t}}$ if $t$ is even and $\mathcal P_{L_{j,t+1}}$ if $t$ is odd.
    At time-step 0, we skip the receive step. 

    \subparagraph{Alphabet:} The alphabet $\Sigma_0$ has to be large enough to accommodate the symbols received on a vertex from all the paths,
    i.e., we need $\Sigma_0$ to be of size $|\Sigma|^{O(\log^{C+1} n)}$, where the exponent comes from congestion guarantee of \Cref{thm:ext-paths}.
\begin{remark*}
As we also have labels on top each vertex, the order of the `send' and `receive' operations are reversed.
\end{remark*}

\paragraph{Bad links and bounding them.} In the following lemmas for a set of corrupted edges $\mathcal E\subseteq X(1)$, we call a vertex-link $X_u$ \emph{bad} if $\pi_1(X_u(1)\cap \mathcal E)\geq \sqrt \eps$ and \emph{good} otherwise.
The guarantees of the routing protocol are essentially immediate from the following lemmas which are essentially bounding probabilities of the bad events, and are the direct consequence of expansion of swap-walks in a local-spectral-expander (\Cref{lem:color-swap-expansion}),
the expander mixing lemma (\Cref{lem:expander-mixing}),
and a sampling lemma for bipartite expanders (\Cref{lem:expander-sampling}).
\begin{lemma}[{\cite[Claim 4.7]{BMVY25}}]
\label{lem:interlink-vertex}
Let $X$ be a $d$-partite $\lambda$-product.
Let $\mathcal E\subseteq X(1)$ be a set of corrupted edges.
Then, we have
\[
\pr{u\sim X(0)}{X_u\text{ is bad}}=\pr{u\sim X(0)}{\pi_1(X_u(1)\cap \mathcal E)\geq \sqrt \eps} \leq \poly(d) \lambda^2.
\]
\end{lemma}
\begin{lemma}[{\cite[Claim 4.8]{BMVY25}}]
\label{lem:intralink}
Let $X$ be a $d$-partite $\lambda$-product.
Let $\mathcal E\subseteq X(1)$ be a set of corrupted edges.
Let
\[
\mathcal D_u = \set{v \in X_u (0) ~\Big|~ \pr{w\sim X_u(0)}{P_{X_u} (v, w)\text{ is corrupted}} \geq \eps^{1/32}}.
\]
If $X_u$ is good, then $\pr{v\sim X_u(0)}{v \in \mathcal D_u} \lesssim \eps^{\frac{1}{32}}$.
\end{lemma}

\begin{lemma}[{\cite[Claim 4.9]{BMVY25}}]
\label{lem:interlink-edge}
Let $X$ and $\mathcal D_u$ be as above. Call an edge-link $X_{u,v}$ \emph{bad} if either $X_u$ or $X_v$ is a bad vertex-link, or one of $\pi_1(X_{uv}(1)\cap \mathcal D_u)$ or $\pi_1(X_{uv}(1)\cap \mathcal D_v)$ is at
least $\eps^{\frac{1}{64}}$.
Then
\[
\pr{uv\sim X(1)}{X_{uv}\text{ is bad}} \leq \poly(d) \lambda^2.
\]
\end{lemma}

Now, the performance of the internal paths can be bounded using \Cref{lem:intralink}
(see the ``Link to Link Transfer on Good Paths'' paragraph in the original proof),
and the performance of the interlink paths can be bounded by interlink paths using \Cref{lem:interlink-vertex} and \Cref{lem:interlink-edge}
(see the ``Bounding the number of Good Paths'' paragraph in the original proof), and then combining them using a union bound with the failure probability $\frac{1}{\log^C n}$ of the initial pebble routing protocol $\mathcal R$ to get the overall bound to be
\begin{align*}
&\eta+\poly_1(d)\lambda^2 \log^C n + \poly_2(d) r^2 \lambda^2 \log^C n + \frac{1}{\log^C n}\\
=&\eta+\poly(d) r^2 \lambda^2 \log^C n + \frac{1}{\log^C n}.
\end{align*}

\subsection{Embedding PCP on an HDX with an Edge-Tolerant Routing}

Finally, we sketch how \cite{BMVY25} use the link-to-link routing protocol to embed a PCP onto a simplicial complex $X$; hence sketch a proof of \cite[Lemma 5.3]{BMVY25}.
Let $G=(X(0),X(1))$ and $Z=G \zprod \mathcal H$, where $\mathcal H$ is a family of regular expanders chosen so that \Cref{lem:link-to-link} gives the required link-to-link routing.

\cite{BMVY25} start by reducing a general 2-CSP $\Psi_0$ to a 2-CSP $\Phi$ on a bipartite $r$-regular graph $G'$ on $|V (Z)|$ vertices such that if $\Psi_0$ is satisfiable then so is $\Phi$ and if $\val(\Psi_0) \leq 1 - 8\nu$ then $\val(\Phi) \leq 1 - \nu$.
We skip further discussion of this reduction here and refer the interested reader to the proof of Lemma 5.2 in \cite{BMVY25}.

Since the graph $G'$ is a bipartite $r$-regular graph, we may partition the edge set of $G'$ into $r$ perfect
matchings, i.e., permutations $\set{\pi_i}_{i\in [r]}$ with $\pi_i^2=Id$ for each $i\in [r]$.
We can view a satisfying assignment $A:V(Z)\to \Sigma$ for $\Phi$ as an assignment such that $(A(v),A(\pi_i(v)))$ satisfies $(v,\pi_i(v))$ for each $i \in [r]$.
We will use this variant of this observation to reduce the 2-CSP $\Phi$ to a generalized 2-CSP $\Psi$ on $G=(X(0),X(1))$.
Towards this, for each matching, we use \Cref{lem:link-to-link} to construct a corresponding routing protocol.
Thus, we will get a set of link-to-link protocols $\set{\mathcal R_i}_{i\in [r]}$.
From an assignment $A:V(Z)\to \Sigma$, we now obtain $A_0:\mathcal S \to \Sigma$ with $\mathcal S = \set{(v, u) : v=(v_1,v_2) \in V (Z), u \in X_{v_1}(0)}$ given by
$A_0(j,u) = A(j)$ for each $j\in V(Z)$ and $u\in X_{j_1}$.
If the assignment $A:V(Z)\to \Sigma$ is a satisfying assignment, then for any $i\in [r]$, starting with $A_0:\mathcal S \to \Sigma$ we will end up with $A_{i,T}:\mathcal S \to \Sigma$ such that $(A_0(j,u),A_{i,T}(j,u))$ satisfies $(j, \pi_i(j))$ in $\Phi$.


\paragraph{The alphabet.} The alphabet of $\Psi$ is the set of all possible transcripts for each protocol corresponding to each matching;
i.e., for any vertex $u$ a symbol in the alphabet would correspond to a possible value of $\paren{\inbound_{i,t}(u,\cdot)}_{t\in[T]\cup \set{0}, i\in [r]}$,
where $\set{\inbound^{}_{i,t}(v,\cdot)}_{t\in [T]\cup \set{0}, v\in V(Z)}$ is the sequence of messages corresponding to $\mathcal R_i$.
As discussed above, the starting message of all these protocols encodes the same satisfying assignment to the $\Phi$
and moreover each protocol computes the assignment to the matched vertex.
Towards enforcing this, we will also restrict the alphabet further to symbols where $\paren{\inbound_{i,0}(j,u), \inbound_{i,T}(j,u)}$ satisfy the constraint on the edge $(v, \pi_i(v))$ in the 2-CSP $\Phi$ for all $u$.

\paragraph{The constraints of $\Psi$:} 
For any edge $(u,v)$ in $G$ the constraint on the edge checks whether
\begin{enumerate}
    \item for each $j \in V (Z)$ for which $u, v$ are both in $X_{j_1}(0)$: for all $i \in [r]\cup \set{0}$ it holds that $A_{i,T} (j, u) = A_{i,T} (j, v)$, and
    \item for each $i \in [r]$ and $t \in [T - 1]$: $\inbound_{i,t+1} (u, v) = \mathtt{OUT}_{i,t} (v, u)$ and $\inbound_{i,t+1} (v, u) = \mathtt{OUT}_{i,t}(u, v)$.
\end{enumerate}
These constraints ensure the routing protocols are being followed correctly.

This construction, along with a careful analysis of the soundness and completeness, gives us the following two lemmas, which abstract \cite[Lemma 5.3]{BMVY25}. 
\begin{lemma}
There exist $r,C\in \N$ such that for large enough $n \in \N$ the following holds.
Let $\Sigma$ be an alphabet.
Let $X$ be $d$-partite $\lambda$-product with max degree $\Delta(X)$ that has a \emph{$\mathcal H$-regularized $(\eps(n), \eta(n), \eta(n)+\nu(n))$-link-to-link routing} that computes functions with values in $\Sigma$ using an alphabet $\Sigma_0$ of size $f(|\Sigma|)$ with work complexity $w(n)$, where $\mathcal H$ is some family of graphs.
Then there is a $\poly(n)$-time procedure mapping any 2-CSP $\Psi_0$ with $n$ vertices
on a $r$-regular graph $G_0$ on the alphabet $\Sigma$,
to a 2-CSP $\Psi$ on the graph $G = (X(0), X(1))$ with alphabet size at most $f(|\Sigma|)^{r\Delta w(n)}$,
satisfying the following properties.
\begin{enumerate}[(i)]
    \item \textbf{Completeness:} If $\val(\Psi_0) = 1$ then $\val(\Psi) = 1$.
    \item \textbf{Soundness:} If $\val(\Psi_0)\leq 1-\eta-\nu-\poly(d)\lambda^2$, then $\val(\Psi) \leq 1-\eps(n)$. 
\end{enumerate}
\end{lemma}
As we have link-to-link routing on symmetrically-connected complex with bounded degree (\Cref{lem:link-to-link}) we get the following as corollary of the above lemma.
\begin{corollary}
There exist $\alpha, \eps_0, d_0(\alpha) > 0$, and $C\in \N$ such that for large enough $n \in \N$ the following holds.
Let $X$ be $d$-partite $\lambda$-product such that
\begin{inparaenum}[(i)]
    \item its underlying graph $G$ has $|\lambda|(G)\leq \alpha/4$;
    \item its vertex links are symmetrically-connected;
    \item its max degree $\Delta(X)$ is bounded by $\Delta$;
    \item the dimension $d\geq d_0$ and its size $|X(0)|= n$.
\end{inparaenum}
Then there is a $\poly(n)$-time procedure mapping any 2-CSP $\Psi_0$ with $n$ vertices
on a $r$-regular graph $G_0$ and alphabet $\Sigma$,
to a 2-CSP $\Psi$ on the graph $G = (X(0), X(1))$ with alphabet size at most $|\Sigma|^{\poly(\Delta \log (n)  \log |\Sigma|)}$,
satisfying the following properties. 
\begin{enumerate}[(i)]
    \item \textbf{Completeness:} If $\val(\Psi_0) = 1$ then $\val(\Psi) = 1$.
    \item \textbf{Soundness:} If $\val(\Psi_0)\leq 1-\poly(d)r^2\lambda^2\log^C n-\frac{1}{\log^C n}$, then $\val(\Psi) \leq 1-\eps_0$.
\end{enumerate}
\end{corollary}
Now that we have embedded a 2-CSP onto the underlying graph of an HDX with edge-tolerant routing,
we will use $(\ell, \sqrt \ell)$-agreement test to get a label cover instance (PCP) with a better soundness (assuming that the soundness of the agreement test is also good). 
\begin{lemma}[{\cite[Lemma 6.2]{BMVY25}}]
For all $\delta > 0$ there exists $\ell \in \N$ such that the following holds.
Suppose that $X$ is a complex for which the $(\ell, \sqrt \ell)$-direct product tester has soundness $\delta^2$. Then there is a polynomial time procedure such that given a 2-CSP instance $\Psi$ over the weighted graph $G = (X(0), X(1))$ with alphabets $\Sigma$, $\set{\Sigma(u)}_{u\in G}$
, produces an instance of generalized Label Cover $\Phi$ over the
$(\ell,\sqrt \ell)$-containment graph $(X(\sqrt \ell), X(\ell))$ over left alphabet $\Sigma^\ell$ and right alphabet $\Sigma^{\sqrt \ell}$ such that the following properties hold.
\begin{enumerate}
    \item The projection map $\phi_{(A,B)}$ associated to the edge $(A, B)$ is defined as
    the restriction of the assignment to $A$ to the coordinates in $B$.
    That is, for all $\sigma \in \Sigma_L(A)$, we have $\phi(A,B)(\sigma) = \sigma|_B$.
    \item If $\val(\Psi) = 1$, then $\val(\Phi) = 1$.
    \item If $\val(\Psi) \leq 1-4\delta$, then $\val(\Phi)\leq \delta$.
\end{enumerate}
\end{lemma}

In conclusion, we get that for nice enough complex that supports an agreement test (with appropriate parameters) we can 
\begin{inparaenum}[(i)]
    \item first construct a routing protocol on it,
    \item then we can embed a 2-CSP using the routing protocol,
    \item finally, we can amplify the gap using the agreement test to get label cover instance with desired soundness.
\end{inparaenum}
\generalpcp*

\end{document}